\documentclass{article}
\usepackage{amsmath,amssymb,amsthm,graphicx,epsfig,float,url}
\usepackage[colorlinks=true]{hyperref}
\usepackage{pdfsync}
\usepackage[usenames, dvipsnames]{xcolor}
\usepackage{tikz}
\usepackage{subfig}
\usepackage{stmaryrd}

\newcommand{\R}{\textnormal{I\kern-0.21emR}}
\newcommand{\N}{\textnormal{I\kern-0.21emN}}

\newcommand{\C}{\mathbb{C}}
\newcommand{\U}{\mathcal{U}_{L,M}}
\newcommand{\Ub}{\overline{\mathcal{U}}_{L,M}}
\newcommand{\Hn}{\mathcal{H}^{n-1}}
\newcommand{\A}{\mathcal{A}}
\newcommand{\D}{\mathcal{D}}
\newcommand{\Ck}{\operatorname{Diff}^\alpha}
 
\renewcommand{\geq}{\geqslant}
\renewcommand{\leq}{\leqslant}
\renewcommand{\P}{\mathcal{P}_N}

\newtheorem{innercustomthm}{Theorem}
\newenvironment{customthm}[1]
{\renewcommand\theinnercustomthm{#1}\innercustomthm}
{\endinnercustomthm}

\newtheorem{theorem}{Theorem}  
\newtheorem{proposition}{Proposition}
\newtheorem{corollary}{Corollary}
\newtheorem{definition}{Definition}
\newtheorem{lemma}{Lemma}

\theoremstyle{definition}\newtheorem{remark}{Remark}

\title{\bf Spectral shape optimization for the Neumann traces of the Dirichlet-Laplacian eigenfunctions}

\author{
	Yannick Privat\footnote{IRMA, Universit\'e de Strasbourg, CNRS UMR 7501, 7 rue Ren\'e Descartes, 67084 Strasbourg, France ({\tt yannick.privat@unistra.fr}).}
	\and  Emmanuel Tr\'elat\footnote{Sorbonne Universit\'es, UPMC Univ Paris 06, CNRS UMR 7598, Laboratoire Jacques-Louis Lions, Institut Universitaire de France, F-75005, Paris, France ({\tt emmanuel.trelat@upmc.fr}).}
   \and Enrique Zuazua\footnote{DeustoTech, University of Deusto, 48007 Bilbao, Basque Country, Spain. Departamento de Matematicas, Universidad Autonoma de Madrid, 28049 Madrid, Spain. Facultad Ingeniera, Universidad de Deusto, Avda. Universidades, 24, 48007, - Basque Country, Spain. Sorbonne Universit\'es, UPMC Univ Paris 06, CNRS UMR 7598, Laboratoire Jacques-Louis Lions, F-75005, Paris, France. (\texttt{enrique.zuazua@deusto.es}).}}

\date{}

\begin{document}

\maketitle

\begin{abstract}
We consider a spectral optimal design problem involving the Neumann traces of the Dirichlet-Laplacian eigenfunctions on a smooth bounded open subset $\Omega$ of $\R^n$. The cost functional measures the amount of energy that Dirichlet eigenfunctions concentrate on the boundary and that can be recovered with a bounded density function. 

We first prove that, assuming a $L^1$ constraint on densities, the so-called {\it Rellich functions} maximize this functional.

Motivated by several issues in shape optimization or observation theory where it is relevant to deal with bounded densities, and noticing that the $L^\infty$-norm of {\it Rellich functions} may be large, depending on the shape of $\Omega$, we analyze the effect of adding pointwise constraints when maximizing the same functional. We investigate  the optimality of {\it bang-bang} functions and {\it Rellich densities} for this problem. We also deal with similar issues for a close problem, where the cost functional is replaced by a spectral approximations.
Finally, this study is completed by the investigation of particular geometries and is illustrated by several numerical simulations. 

\end{abstract}

\noindent\textbf{Keywords:} wave equation, boundary observability, shape optimization, calculus of variations, spectrum of the laplacian, quantum ergodicity at the boundary.

\medskip

\noindent\textbf{AMS classification:} 35P20, 93B07, 58J51, 49K20.

\tableofcontents


\section{Introduction}\label{secintro}

\subsection{Motivation}

This article is devoted to the investigation of spectral problems involving the Neumann traces of the Dirichlet-Laplacian eigenfunctions, having applications in shape sensitivity analysis, observation and control theory. 

Let $\Omega$ be a bounded connected open subset of $\R^n$ with Lipschitz boundary. Consider a Hilbert basis $(\phi_j)_{j\in \N^*}$ of $L^2(\Omega)$, consisting of real-valued eigenfunctions of the Dirichlet-Laplacian operator on $\Omega$, associated with the negative eigenvalues $(-\lambda_j(\Omega))_{j\in\N^*}$. In the whole article, the eigenvalues $\lambda_j(\Omega)$ will also be denoted $\lambda_j$ when there is no need to underline their dependence on $\Omega$.

In what follows, since all the geometrical quantities that will be handled are scale-invariant, we will assume that $\Omega$ satisfies the normalization condition 
\begin{equation}\label{eq:circum}
R(\Omega)=1
\end{equation}
where $R(\Omega)$ denotes the circumradius\footnote{In other words, the smallest radius of balls containing $\Omega$.} of $\Omega$. Obviously, other normalization choices would be possible, but the one we consider allows to slightly simplify the presentation of our results.

The Lipschitz set $\partial\Omega$ is endowed with the $(n-1)$-dimensional Hausdorff measure $\Hn$. In the sequel, measurability of a subset $\Gamma\subset \partial \Omega$ is understood with respect to the measure $\Hn$.
We will use the notation $\chi_\Gamma$ to denote the characteristic function\footnote{The characteristic function $\chi_\Gamma$ of the set $\Gamma$ is the function equal to $1$ in $\Gamma$ and $0$ elsewhere.} of the set $\Gamma$, $\nu$ is the outward unit normal to $\partial\Omega$ and ${\partial f}/{\partial\nu}$ is the normal derivative of a function $f\in H^{2}(\Omega)$ on the boundary $\partial\Omega$.

\medskip

The starting point is the famous Rellich identity\footnote{We also mention \cite{Liu} for a review of Rellich-type identities and their use in free boundary problems theory.}, discovered by Rellich in 1940 \cite{Rellich2}, stating that
\begin{equation}\label{Tao}
\boxed{
\forall x_0\in \R^n, \qquad 2 =\frac{1}{\lambda_j} \int_{\partial\Omega} \langle x-x_0,\nu(x) \rangle \left( \frac{\partial \phi}{\partial \nu}(x)\right)^2\, d\Hn(x)
}
\end{equation}
for every $\mathcal{C}^{1,1}$ or convex bounded domain $\Omega$ of $\R^n$, where $\langle \cdot ,\cdot \rangle$ is the Euclidean scalar product in $\R^n$, and for every eigenfunction $\phi$ of the Dirichlet-Laplacian operator. This identity can be interpreted in terms of the ``boundary Neumann energy'' of eigenfunctions on the boundary of $\Omega$. Indeed, let $x_0\in \R^n$ and set  
\begin{equation}\label{def:ax0}
\text{for a.e. }x\in\partial\Omega, \qquad  a_{x_0}(x) = \langle x-x_0,\nu(x)\rangle,
\end{equation}
the identity \eqref{Tao} states in particular that 
 $$
\forall N\in \N^*, \quad \min_{1\leq j\leq N}\frac{1}{\lambda_j}\int_{ \partial \Omega} a_{x_0}\left(\frac{\partial \phi_j}{\partial \nu}\right)^2 \, d\Hn = 2
 $$
 and moreover, the infimum is reached by every index $j\in \N^*$. Therefore, the function $a_{x_0}$ acts as a perfect spectral mirror. We will refer to {\it Rellich function} for designating functions of the form \eqref{def:ax0}.

\medskip

This leads us to introduce the functional $J_N$ defined by

\begin{equation}\label{defJaNrelax}
\boxed{
J_N(a)=\inf_{1\leq j\leq N}\frac{1}{\lambda_j}\int_{\partial \Omega} a\left(\frac{\partial \phi_j}{\partial \nu}\right)^2 \, d\Hn .
}
\end{equation}
involving the $N$ first modes of the Dirichlet-Laplace operator, as well as its infinite version $J_\infty$ defined by 
\begin{equation}\label{defJarelax}
\boxed{
\forall a\in L^\infty(\partial\Omega),\qquad J_\infty(a)=\inf_{j\in\mathbb{N}^*}\frac{1}{\lambda_j}\int_{ \partial \Omega} a\left(\frac{\partial \phi_j}{\partial \nu}\right)^2 \, d\Hn
}
\end{equation}

\bigskip
Note that each integral $\int_{ \partial\Omega} a\left(\frac{\partial \phi_j}{\partial \nu}\right)^2 \, d\Hn$ in the definition of $J_N$ or $J_\infty$ is well defined and is finite whenever $\Omega$ is convex or has a $\mathcal{C}^{1,1}$ boundary\footnote{Indeed , the outward unit normal $\nu$ is defined almost everywhere, the eigenfunctions $\phi_j$ belong to $H^2(\Omega)$ and their Neumann trace $\partial\phi_j/\partial\nu$ belongs to $L^2(\partial\Omega)$ for any $j\geq 1$. Hence $a\left(\frac{\partial \phi_j}{\partial \nu}\right)^2\in L^1(\partial\Omega)$ for every $a\in L^\infty(\Omega)$.}.
Interpreting $\lambda_j$ as the boundary Neumann energy of the eigenfunction $\phi_j$, it follows that
{the functional $J_\infty$ (resp. $J_N$) {measures the worst amount of Dirichlet eigenfunctions} (resp. the worst amount of the $N$ first Dirichlet eigenfunctions) boundary Neumann energy that can be recovered with the density function $a$. }
From \eqref{Tao} and the considerations above, one has $J_N(a_{x_0})=J_\infty(a_{x_0})=2$. 

Looking for the density functions enjoying optimal spectral properties in terms of the boundary Neumann energy, it appears relevant to maximize either $J_N(a)$ or $J_\infty(a)$. For this last criterion, we will see that Rellich functions are natural candidates for solving this problem.   

The aim of the ongoing study is to quantify this observation and analyze such problems. In particular and as underlined in what follows, these issues are connected with several concrete applications.

\medskip

As a first result, let us show that, in some sense, the density function involved in the Rellich identity is the best possible (for maximizing the criterion $J_\infty$) when considering a $L^1$-type constraint on densities. 

\begin{theorem} \label{prop:metz1756}
Assume that $\Omega$ either is convex or has a $\mathcal{C}^{1,1}$ boundary. Then,
\begin{equation}\label{estimJa}
\forall a\in L^\infty(\partial \Omega),\qquad -\infty < J_\infty (a)\leq \frac{2}{n|\Omega|}\int_{\partial \Omega}a\, d\Hn ,
\end{equation}
and this inequality is an equality for every {\it Rellich function} $a_{x_0}$ defined by \eqref{def:ax0} with $x_0\in \R^n$.
As a consequence, for a given $p_0\in \R$, we have
\begin{equation}\label{maxJapropsssign}
\max \left\{J_\infty (a), a\in L^\infty(\partial\Omega)\ \mid\  \int_{\partial\Omega}a\, d\Hn=p_0\right\}= \frac{2p_0}{n|\Omega|}
\end{equation}
and the maximum is reached by any function $\tilde a_{x_0}$ defined from the Rellich function $a_{x_0}$ on $\partial\Omega$ by 
\begin{equation}\label{def:tildeax0}
\tilde a_{x_0}( x)=  \frac{p_0}{n|\Omega|} \langle x-x_0,\nu(x)\rangle,\qquad \text{with $x_0\in \R^n$.}
\end{equation}
\end{theorem}

The proof of Theorem \ref{prop:metz1756} is postponed to Section \ref{sec:proofpropOptvalue}.
An important ingredient of the proof is the uniform convergence of the Ces\`aro means of the functions $\frac{1}{\lambda_j}\left(\frac{\partial \phi_j}{\partial \nu}\right)^2$ to a constant as $j\rightarrow+\infty$. 

\begin{figure}[h!]
\begin{center}
\includegraphics[height=3.5cm]{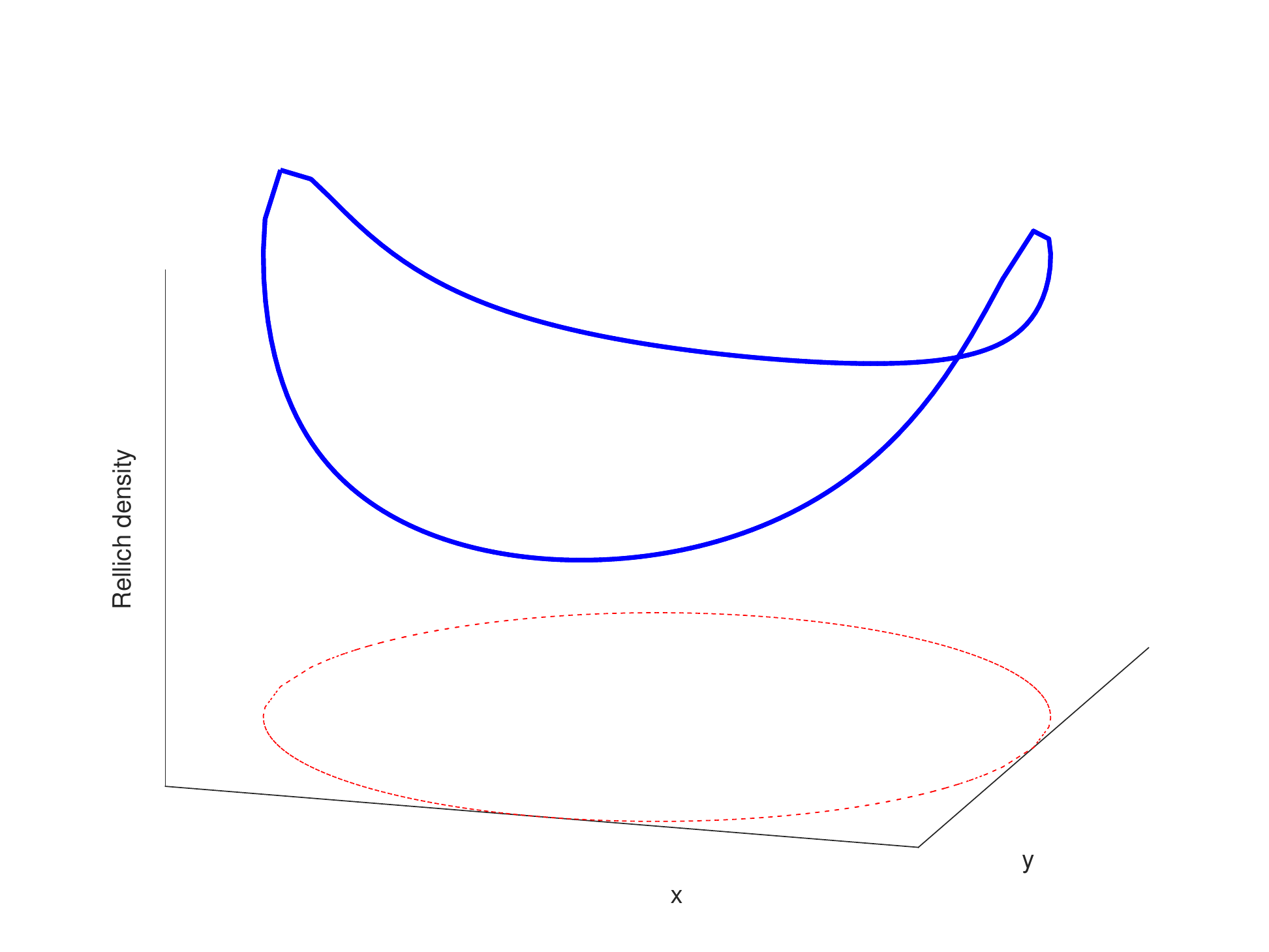}
\includegraphics[height=3.5cm]{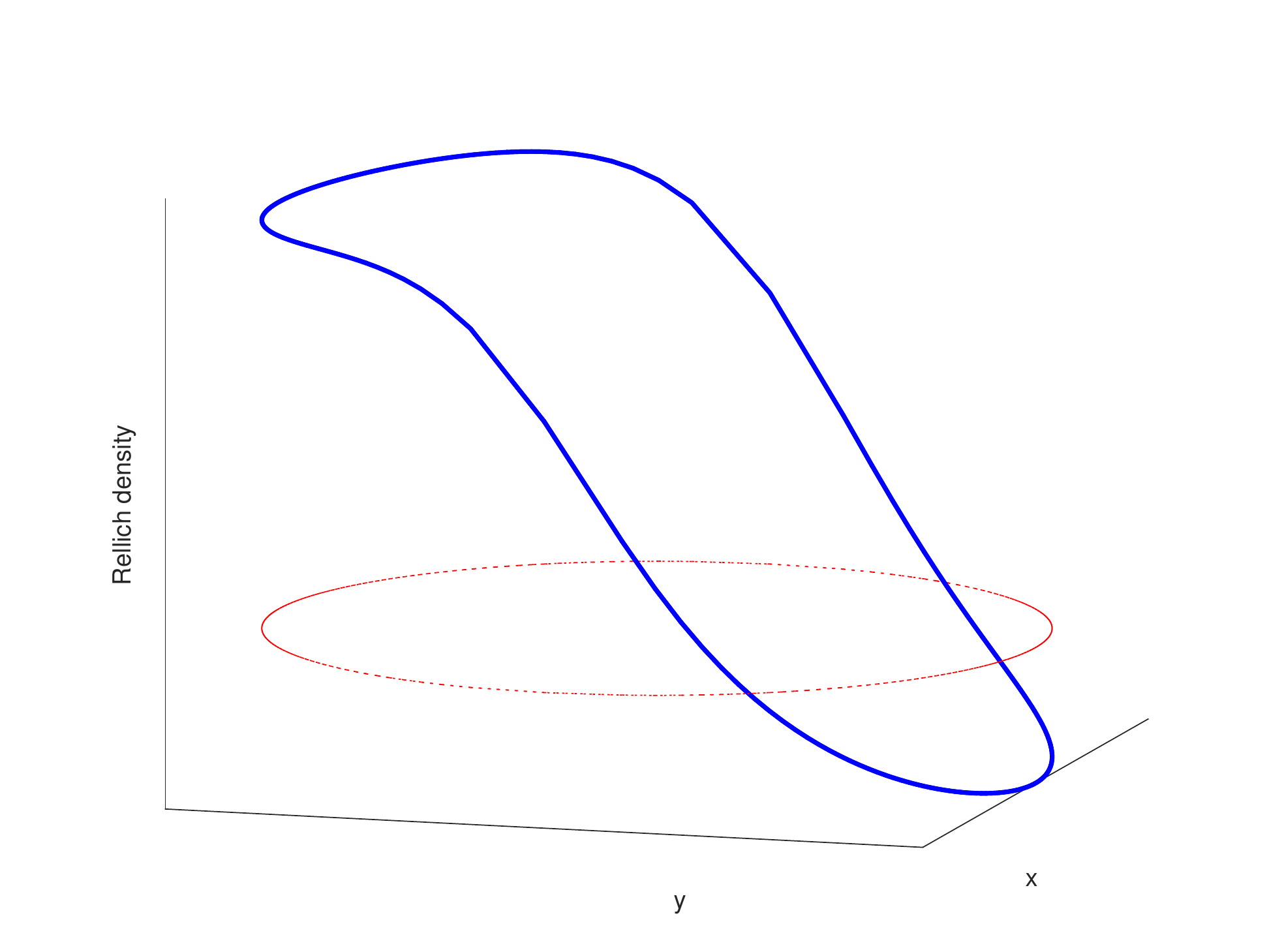}
\includegraphics[height=3.5cm]{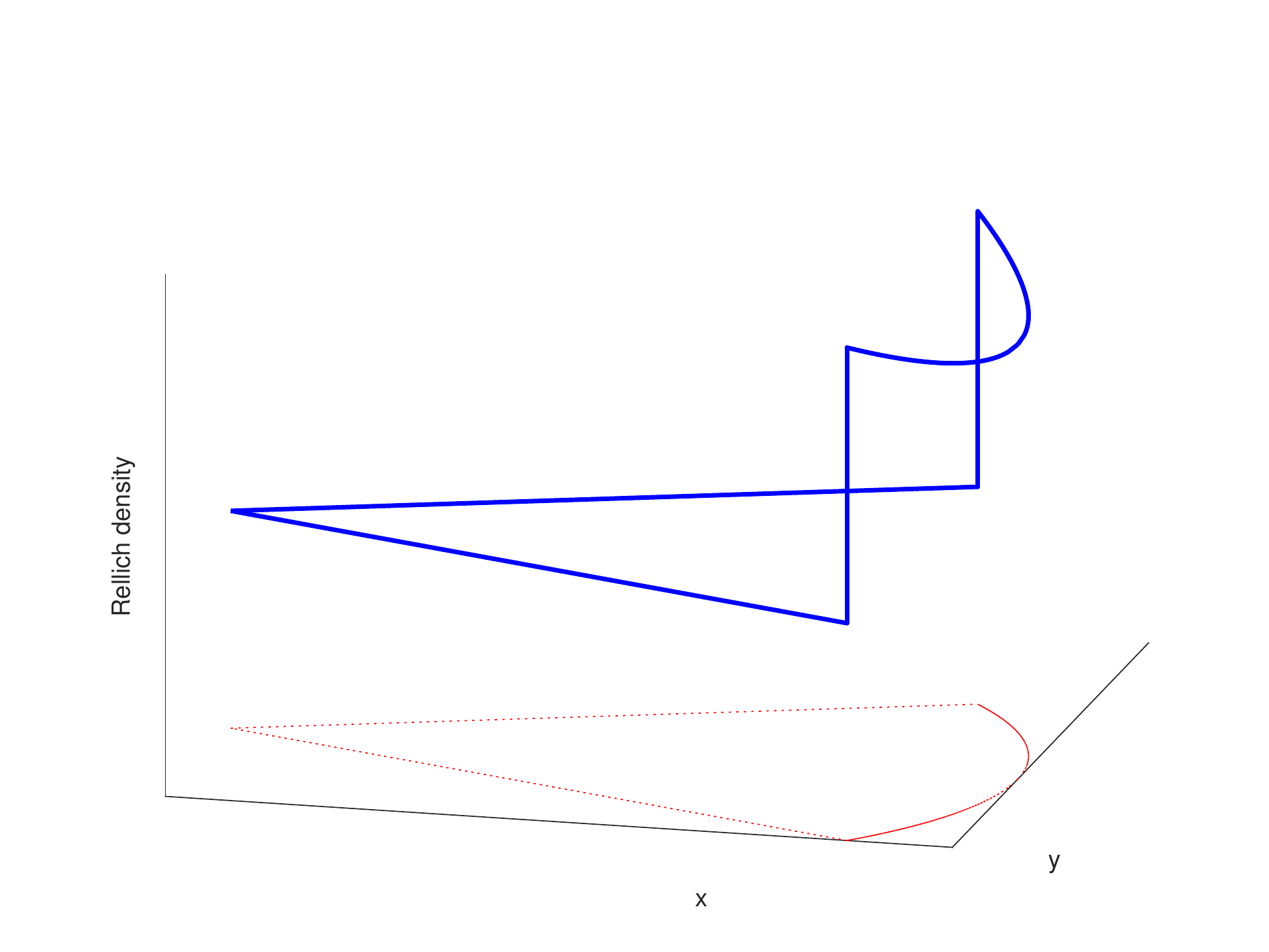}
\caption{Plot of three Rellich densities (continuous line): for an ellipse (dotted line, left and middle) with several locations of $x_0$ and for an angular sector (dotted line, right). The dotted figures are plotted in the plane $\{z=0\}$.  \label{Fig:densiteEllSA}}
\end{center}
\end{figure}

\medskip

The optimal design problems we will deal with are motivated by the following observation: the maximizers of Problem \eqref{maxJapropsssign} given by \eqref{def:tildeax0} may have an arbitrarily large $L^\infty$-norm depending on the choice of the domain $\Omega$. For instance, fix $\varepsilon>0$ arbitrarily small and assume that $\Omega$ is an ellipse in $\R^2$ with circumradius equal to 1. Then, the lowest $L^\infty$-norm of maximizers given by \eqref{def:tildeax0} is equal to $1/\varepsilon$ by choosing a small enough minor semi-axis for $\Omega$ (see Fig. \ref{Fig:Lentille}). This claim will be formalized in Proposition \ref{prop:Linftynormmax} and justifies to consider a modified version of Problem \eqref{maxJapropsssign}, where one assumes that the density function $a(\cdot)$ is uniformly bounded in $L^\infty$ by some positive constant $M$. As will be underlined in the sequel, such a remark also holds for the criterion $J_N$ whenever $N$ is large enough.  

\begin{figure}[h!]
\begin{center}
\includegraphics[height=5cm]{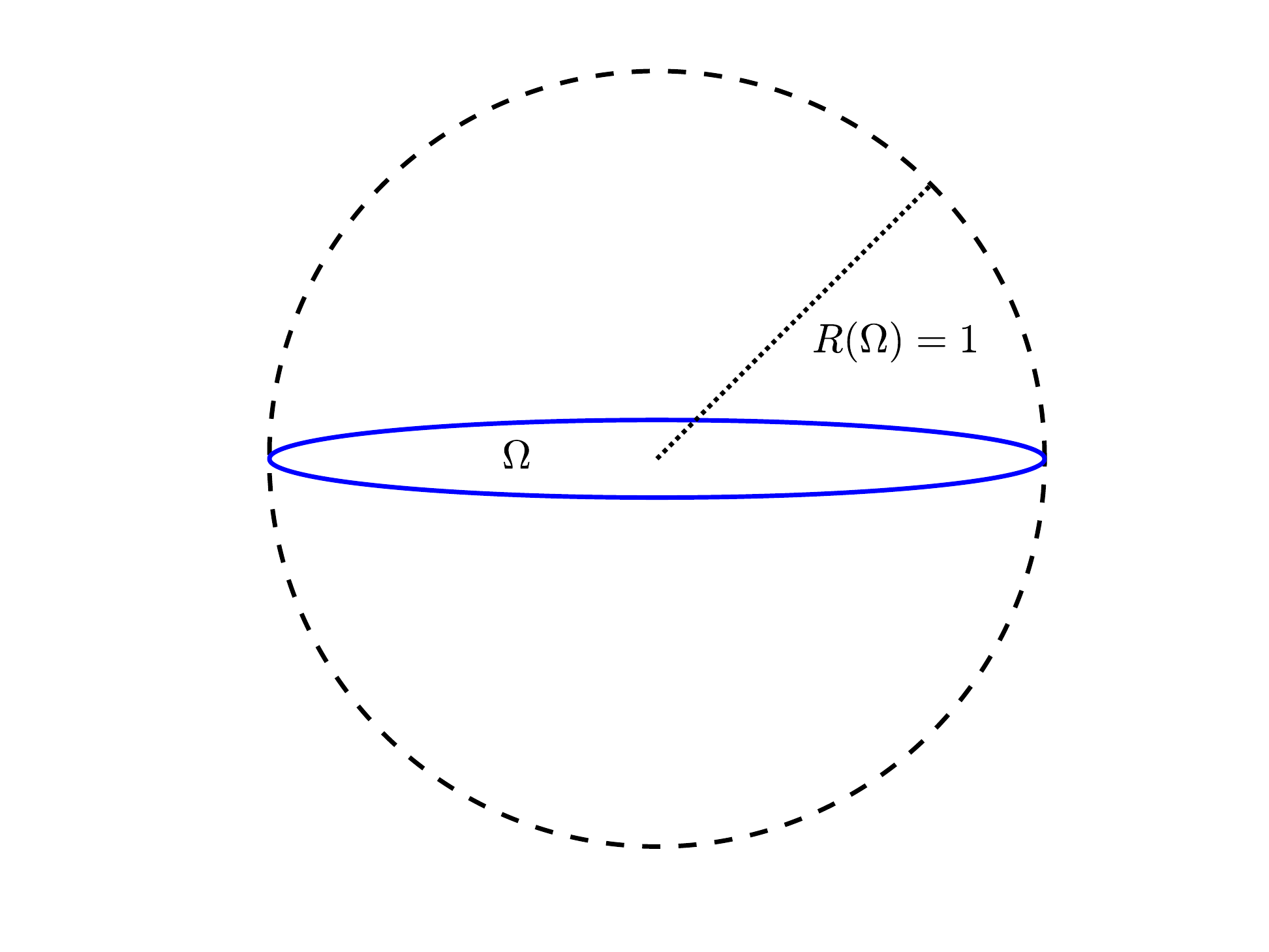}
\caption{Two convex sets with circumradius equal to 1: a lens (solid line) and a disk (dotted line). \label{Fig:Lentille}}
\end{center}
\end{figure}

Let us fix some $L\in(-1,1)$ and $M>0$. Summing-up the previous considerations and noting that for every $a\in L^\infty(\partial\Omega)$ such that $-M\leq a(\cdot)\leq M$, one has 
$$
-M\Hn(\partial\Omega)\leq \int_{\partial \Omega}a \, d\Hn\leq M\Hn(\partial\Omega),
$$
it is relevant to consider the class of density functions
$$
\boxed{\Ub=\left\{a\in L^\infty(\partial\Omega)\ \mid\  -M\leq a\leq M\text{ a.e. in }\partial\Omega\text{ and }\int_{\partial\Omega}a\, d\Hn=LM\Hn(\partial\Omega)\right\}}
$$
i.e., we consider measurable functions $a$ whose $L^\infty$-norm is bounded above by $M$ and that have a prescribed integral. 

Computing the supremum of $J_N(a)$ or $J_\infty(a)$ over $\Ub$ becomes much more difficult since one considers additional pointwise constraints. Indeed, the existence of Rellich functions satisfying at the same time a $L^\infty$ (pointwise upper bound) and a $L^1$ (integral) constraints is much dependent on the geometry of $\Omega$, as will be highlighted in the sequel.

Let $M>0$ and $L\in (-1,1)$. We will then analyze the optimization problem
\begin{equation}\label{truncPbOpt}\tag{$\mathcal{P}_N$}
\textsf{($N$-truncated spectral problem)}\quad \boxed{
\sup_{a\in \Ub} J_N(a)
}
\end{equation}
where $N\in \N^*$ is given, as well as its asymptotic version as $N$ tends to $+\infty$,
\begin{equation}\label{defJrelax}\tag{$\mathcal{P}_\infty$}
\textsf{(full spectral problem)}\quad\boxed{
 \sup_{a\in \Ub} J_\infty(a)
}
\end{equation}

According to Theorem \ref{prop:metz1756}, natural candidates for solving Problem \eqref{defJrelax} are defined from the Rellich functions $a_{x_0}$ as follows.

\begin{definition}[Rellich admissible functions]\label{def:Rellichadmf}
Let $\Omega$ be as previously and $x_0\in \overline{\Omega}$. We will say that the function $\tilde a_{x_0}$ defined by 
\begin{equation}\label{defax0}
\displaystyle \tilde a_{x_0}(x)=\frac{LM\Hn(\partial\Omega)}{n|\Omega|}\langle x-x_0,\nu(x)\rangle,
\end{equation}
for a.e. $x\in\partial\Omega$, is a \textsl{Rellich admissible function} of Problem \eqref{defJrelax} whenever it belongs to the admissible set $\Ub$\footnote{In other words whenever $-M\leq \tilde a_{x_0}\leq M$ a.e. in $\partial\Omega$ since the function $\tilde a_{x_0}$ is constructed in such a way that $\int_{\partial\Omega}\tilde a_{x_0}\, d\Hn=LM\Hn(\partial\Omega)$.}.
\end{definition}

\subsection{Statement of the results}
Let us first first state an existence result and highlight the connection between Problems \eqref{truncPbOpt} and \eqref{defJrelax}.

\begin{proposition}\label{GammaCvProp}
Let $L\in[-1,1]$ and $N\in \N^*$. Assume that $\Omega$ either is convex or has a $\mathcal{C}^{1,1}$ boundary. Then, Problem \eqref{truncPbOpt} has at least one solution $a_N$ in $\Ub$.
Moreover, every sequence $(a_N)_{N\in\N^*}$ of maximizers of $J_N$ in $\Ub$ converges (up to a subsequence) for the $L^{\infty}(\partial\Omega,[-M,M])$ weak-star topology to a solution of Problem \eqref{defJrelax}, and 
$$
\max_{a\in\Ub}J_\infty(a)=\lim_{N\rightarrow +\infty}J_N(a_N^*).
$$
\end{proposition}
We will not provide the proof of this result since it is a straightforward adaptation of the proof of Lemma \ref{lem:existPinfty} (see Section \ref{section2}) for the existence, and of the proof of \cite[Theorem 8]{PTZobsND} for the $\Gamma$-convergence property. 

\bigskip

For the sake of readability, we describe hereafter simplified versions of the main results of this article under the following assumption of the domain $\Omega$:

\begin{equation}\tag{$H$}\label{hyp:dom}
\text{$\Omega$ is a bounded connected domain of $\R^n$ with a $\mathcal{C}^{1,1}$ boundary}.
\end{equation} 

Further results in the case where $\Omega$ is convex will be stated in the body of the article.

\paragraph{\sf{\large{Analysis of Problem \eqref{truncPbOpt}.}}}

According to Proposition \ref{GammaCvProp}, Problem \eqref{truncPbOpt} has at least one solution $a_N$. In the following result, we aim at describing $a_n$, wondering whether it may be an extremal point of the convex set $\Ub$, in other words a function either equal to $M$ or $-M$ a.e. in $\partial\Omega$. In control theory, a function enjoying such a property is said to be {\it bang-bang}. Uniqueness of solutions for Problem  \eqref{truncPbOpt} can then be inferred from this property.

\begin{customthm}{A} \label{theo:paris1238}
Let $N\in\N^*$. For \emph{generic} domains $\Omega$ whose boundary is at least of class $\mathcal{C}^2$, Problem \eqref{truncPbOpt} has a unique solution which is moreover {\it bang-bang}. 
\end{customthm}

A complete version of this theorem is provided in Theorem \ref{analytics}. Here, genericity is understood in terms of analytic deformations of the domain. This result is proved in Section~\ref{sec:prooftheoAnalytics}.

\paragraph{\sf{\large{Analysis of Problem \eqref{defJrelax}.}}} 
As a preliminary remark, notice that Problem \eqref{defJrelax} has at least a solution, as stated in Lemma \ref{lem:existPinfty}.

The results provided in Theorems \ref{prop:metz1756} and \ref{theo:paris1238} suggest to investigate the two following issues:
\begin{enumerate}
\item for which values of the parameters do the Rellich functions defined by \eqref{def:tildeax0} still remain optimal for Problem \eqref{defJrelax}?
\item what happens when restricting the search of solutions to {\it bang-bang} densities for Problem \eqref{defJrelax}?
\end{enumerate}

\begin{customthm}{B}[Optimality of Rellich functions] \label{th:2019}
Let $\Omega\subset \R^n$ be a domain satisfying \eqref{hyp:dom}.
Introduce 
\begin{equation}\label{Lstar}
L^c_{n} = \min \left\{ 1, L^*_{n}(\Omega) \right\}\qquad{ with }\qquad L^*_{n}(\Omega) =\frac{n|\Omega|}{\Hn (\partial\Omega) \inf_{x_0\in\R^n}\ell_{\partial\Omega}(x_0)},
\end{equation}
where $\ell_{\partial\Omega}(x_0)$ denotes the distance from $x_0$ to the furthest point of $\partial\Omega$\footnote{Let $x_0\in \R^n$. The quantity $\ell_{\partial\Omega}(x_0)$ is defined by
$
\ell_{\partial\Omega}(x_0)=\max_{x\in \partial\Omega}\Vert x-x_0\Vert.
$
}
.Then, there exists a Rellich function $\tilde a_{x_0}$ (defined by \eqref{defax0}) solving Problem  \eqref{defJrelax} if and only if $L\in [-L^c_{n},L^c_{n}]$.
\end{customthm}
Note that, according to Theorem \ref{prop:metz1756}, the optimality of Rellich functions is equivalent to their admissibility, in other words the existence of a Rellich function belonging to $\Ub$. Thus, the number $L^c_{n}(\Omega)$ corresponds to the largest possible value of $L$ such that there exists $x_0\in \overline{\Omega}$ for which $-M\leq \tilde a_{x_0} \leq M$ pointwisely in $\Omega$.

Actually, we provide in Section \ref{sec:soveRelaxopt} a refined version of this result (see Theorem \ref{OptimalValue}) and we comment on the critical value $L^c_{n}$ and the function $\ell_{\partial\Omega}$ involved above, which we even compute explicitly in some particular cases in Section \ref{Invest}.

\medskip

According to Theorem \ref{theo:paris1238}, {\it bang-bang} functions of $\Ub$, in other words extremal points of $\Ub$, solve Problem \eqref{truncPbOpt} for generic choices of domains $\Omega$. This leads to investigate the relationships between Problem \eqref{defJrelax} and a close version where only {\it bang-bang} functions are involved. Let $a$ be an extremal point of $\Ub$. Then, there exists a measurable subset $\Gamma$ of $\partial \Omega$ such that $a=M\chi_\Gamma-M\chi_{\partial\Omega\backslash \Gamma}=M(2\chi_\Gamma-1)$ and the $L^1$ constraint $\int_{\partial\Omega}a=LM\Hn(\partial\Omega)$ reads in that case $\Hn(\Gamma)=\frac{L+1}{2}\Hn(\partial\Omega)$.  We then introduce the optimal design problem
\begin{equation}\label{mainPbOpt}\tag{$\mathcal{P}_\infty^{\rm bb}$}
\boxed{
\sup_{\chi_\Gamma \in \mathcal{U}_{L,M} }J_\infty(M\chi_\Gamma-M\chi_{\partial\Omega\backslash \Gamma})
}
\end{equation}
where $\mathcal{U}_{L,M}$ denotes the set of extremal points of $\Ub$, namely
\begin{equation}\label{def:UL}
\mathcal{U}_{L,M} = \left\{ M\chi_\Gamma-M\chi_{\partial\Omega\backslash \Gamma}\ \vert\ \Gamma\ \subset \partial\Omega \ \textrm{ and}\ \Hn(\Gamma)=\frac{L+1}{2}\Hn(\partial \Omega )\right\}
\end{equation}

Of course, one of the main difficulties in that issue is to deal with a hard binary non-convex constraint on the function $a$, preventing {\it a priori} the solution to be an element of the family $\{\tilde a_{x_0}\}_{x_0\in \overline{\Omega}}$ (since $\Omega$ has a $\mathcal{C}^{1,1}$ boundary, each Rellich function is continuous on $\partial\Omega$ and cannot be an element of  $\mathcal{U}_{L,M} $ whenever $L\notin \{-1,1\}$). 
As it will be  emphasized in the sequel, Problem \eqref{mainPbOpt} plays an important role when dealing with inverse problems involving sensors. This means that, among all subsets $\Gamma$ of $\partial \Omega$ having a prescribed Hausdorff measure, we want to recover the maximal part of the ``boundary Neumann energy measure''.

We will establish the following result. 

\begin{customthm}{D} [\textnormal{no-gap}]
Let $\Omega$ be a domain satisfying \eqref{hyp:dom}. Under strong assumptions on $\Omega$ (related to quantum ergodicity issues), and using the notations of Theorem \ref{th:2019} above, the optimal values of Problems \eqref{defJrelax} and \eqref{mainPbOpt} are the same.
\end{customthm}
A complete version of this result is provided in Theorem \ref{No-Gap}.  Although we do not know whether Problem \eqref{mainPbOpt} has a solution, the investigation of several particular cases in Section \ref{Invest} let us make the conjecture that for almost every value of the constraint parameter $L$,  Problem \eqref{mainPbOpt} has no solution.

\subsection{Structure of the article}

Section \ref{TruncSection} is devoted to solving Problem \eqref{truncPbOpt}. In \cite{PTZObs1,PTZobsND}, it had been proved that there exists a unique optimal set, depending however on $N$ in a very unstable way (spillover phenomenon). Here, existence and uniqueness are, by far, more difficult to state. By exploiting analytic perturbation properties, we prove the existence of a unique optimal set (depending on $N$) for generic domains $\Omega$.

In Section \ref{section2}, we focus on Problem \eqref{defJrelax}, highlighting an interesting geometric phenomenon that can be measured by {\em Rellich functions}.

Relationships between Problems \eqref{defJrelax} and \eqref{mainPbOpt} are investigated in Section \ref{sec:mainPbOptm}. One shows in particular that the optimal values of these two problems may coincide under adequate quantum ergodicity assumptions. We construct maximizing sequences for Problem \eqref{mainPbOpt}. We also consider several particular cases (square, disk, angular sector) as an illustration.

Finally, we provide an interpretation of the above problems in terms of shape sensitivity analysis in Section \ref{FurtherComments}. Other motivations related to observation theory and more specifically to the optimal location or shape or sensors for vibrating systems (as already shortly alluded) are evoked.


\section{Solving of the optimal design problem \eqref{truncPbOpt}}
\label{TruncSection}

\subsection{A genericity result}

In this section, we investigate uniqueness issues and the characterization of maximizers. We prove that Problem \eqref{truncPbOpt} has a unique solution which is moreover {\it bang-bang} for generic domains $\Omega$. The wording ``unique solution'' means that two solutions are equal almost everywhere in $\partial \Omega$.

Before stating the main result of this section, let us clarify the notion of genericity we will use.
Let $\alpha \in\N\backslash\{0,1\}$. In what follows, we will denote by $\Ck$ the set of $\mathcal{C}^\alpha$-diffeomorphisms in $\R^n$. 
We say that a subset $\Omega$ of $\R^n$ is a $\mathcal{C}^\alpha$ topological ball whenever there exists $T\in \Ck$ transforming the unit ball into $\Omega$. We consider the topological space
$$
\Sigma_\alpha=\{T(B(0,1)), \ T\in \Ck\}
$$ 
endowed with the metric induced by that of $\mathcal{C}^k$-diffeomorphisms\footnote{Recall that one can endow $\Ck$ with its topology $\tau$ inherited from the family of semi-norms defined by
$$ 
p_\eta (T) = \sup_{x\in K, j\in \llbracket 1,\alpha \rrbracket^n, |j|\leq \eta} | \partial^j T(x) |. 
$$
for every $\eta \in \{ 1,\dots, \alpha\}$, $K\subset \R^n$ compact,  and $T\in \mathcal{C}^\alpha(\R^n,\R^n)$, making it a complete metric space.}, making it a complete metric space. 
Since our approach is based on analyticity properties, it is convenient to introduce the set $\mathcal{D}$ of domains $\Omega\subset \R^n$ having an analytic boundary, as well as the subset $\A_\alpha = \D \cap \Sigma_\alpha$ of the topological space $\Sigma_\alpha$.

\begin{theorem}\label{analytics}
Let $\alpha \in\N\backslash\{0,1\}$ and $N\in\N^*$. Consider the property:
\begin{itemize}
\item[$(\mathcal{Q}_N)$] {\it for every $L\in [-1,1]$, the optimal design problem \eqref{truncPbOpt} has a unique solution $a_N$ which is therefore an extremal point of the convex set $\Ub$ (in other words, $a_N$ is bang-bang)}.
\end{itemize}
The set of the domains $\Omega\in \mathcal{A}_\alpha$ for which the property $(\mathcal{Q}_N)$ holds true is open and dense in $\A_\alpha$.
\end{theorem}

The proof of this theorem is provided in Section \ref{sec:prooftheoAnalytics}. It is quite lengthy and is based on genericity arguments, using analytic domain deformations.

\begin{remark}[On the uniqueness of solutions]\label{SwitchDisk}
It is notable that the uniqueness property stated in Theorem \ref{analytics} for Problem \eqref{truncPbOpt} does not hold whenever $\Omega$ is a  two-dimensional disk. 

Regarding the two-dimensional unit disk $\Omega=D(0;1)$, it is easily shown that the optimal value for Problem \eqref{truncPbOpt} is equal to $\pi L$ for every $N\in\N^*$. Indeed, explicit computations (see Section \ref{Invest}) yield that 
\begin{eqnarray*}
J_N(a)&=&\min \left( \inf_{1\leq n\leq N} \int_0^{2\pi}a(\theta)\cos(n\theta)^2\, d\theta , \ \inf_{1\leq n\leq N} \int_0^{2\pi}a(\theta)\sin(n\theta)^2\, d\theta  \right)\\
&=& \pi L-\frac{1}{2}\sup_{1\leq n\leq N}\left|\int_0^{2\pi}a(\theta)\cos(2n\theta)\, d\theta\right|,
\end{eqnarray*}
by combining the Riemann-Lebesgue lemma with the identity $\min \{-x,x\}=-|x|$ for all $x\in \R$. Hence, one has $J_N(a)\leq \pi L$ for every $a\in \Ub$, and moreover, the upper bound is reached by choosing $a=L$, leading to
$$
\max_{a\in\Ub}J_N(a)=\pi L.
$$

Moreover, we claim that there exist {\it bang-bang} functions of $\Ub$ solving problem \eqref{truncPbOpt}. Notice that the case of the disk is particular since the strategy developed within the proof of Theorem \ref{analytics} does not apply\footnote{Indeed, the proof rests upon the fact that Problem \eqref{truncPbOpt} has necessarily a {\it bang-bang} solution whenever the set of solutions of the equation 
$$
 \sum_{1\leq j \leq N}\frac{\beta_j^*}{\lambda_j}\left( \frac{\partial \phi_j}{\partial \nu} \right)^2=\text{constant}
$$
on $\partial\Omega$ is either empty or discrete, for all choices of the family $(\beta_j^*)_{1\leq j\leq N}$ of nonnegative numbers such that $\sum_{j=1}^N\beta_j^*=1$. When $\Omega$ denotes the two-dimensional unit disk, one shows easily that such a property does not hold true since the squares of the eigenfunctions normal derivatives involve the square of cosine and sine functions, whose combination may be constant on intervals. 
}. Nevertheless, by choosing 
$$
\omega_N=\bigcup_{i=1}^p\left(\frac{2\pi i}{p}-\frac{\pi L}{p},\frac{2\pi i}{p}+\frac{\pi L}{p}\right)\bigcup \left(0,\frac{\pi L}{p}\right)\bigcup \left(\pi-\frac{\pi L}{p},\pi\right),
$$
with $p=\left[\frac{N+1}{2}\right]$ (the notation $[\cdot]$ standing for the integer part function), one computes for $n\in \{1,\dots,N\}$,
\begin{eqnarray*}
\int_{\omega_N}\cos (2n\theta)\, d\theta &=& \frac{1}{2n}\sum_{i=0}^N \left(\sin\left(\frac{4\pi ni}{p}+\frac{2\pi nL}{p}\right)-\sin\left(\frac{4\pi ni}{p}+\frac{2\pi nL}{p}\right)\right)\\
&=& \frac{1}{n}\sin\left(\frac{2\pi nL}{p}\right) \sum_{i=0}^{p-1}\cos \left(\frac{4\pi ni}{p}\right)=0,
\end{eqnarray*}
since for all $n\in \{1,\dots,N\}$, the real number $4n\pi/p$ cannot be a multiple of $2\pi$ (indeed, $2p\pi \geq 2\pi \left[\frac{N+1}{2}\right]\geq N+1$). We hence infer that one has also $\int_{[0,2\pi]\backslash \omega_N}\cos (2n\theta)\, d\theta=0$ and the choice $a_N=M\chi_{\omega_N}-M\chi_{[0,2\pi]\backslash \omega_N}$ yields
$J_N(a_N)=\pi L$, according to the expression of $J_N$ above.

It il also interesting to notice that the uniqueness property of maximizers also fails whenever $\Omega$ is a two-dimensional rectangle. Indeed, the non-uniqueness property is an easy consequence of the rewriting of the criterion (see Section \ref{Invest}), since it can be observed that solutions are only described by their projections on the vertical or the horizontal axis.

Note that similar issues are investigated in \cite[Prop. 1 and Cor. 2]{PTZObs1}.  
\end{remark}

\subsection{Numerical simulations}

In this section, we illustrate the previous results by representing the solution $a_N^*$ of Problem \eqref{truncPbOpt}, whenever it exists. Moreover, according to the proof of Theorem \ref{analytics} (see Section \ref{sec:prooftheoAnalytics}), there exist Lagrange multipliers $\beta^*=(\beta_j^*)_{1\leq j\leq N}\in\R^N_+$ such that $\sum_{1\leq j\leq N}\beta_j^* = 1$ and a positive real number $\Lambda$ such that $\{\varphi^* > \Lambda\}\subset \left\lbrace a^*=M\right\rbrace$, and $\{\varphi^* < \Lambda\}\subset \left\lbrace a^*=-M\right\rbrace$ where
$$
\varphi^* = \sum_{1\leq j \leq N}\frac{\beta_j^*}{\lambda_j}\left( \frac{\partial \phi_j}{\partial \nu} \right)^2.
$$
Moreover, for generic domains $\Omega$, the previous inclusions are in fact equalities. In the description of the numerical method, we assume to be in such a case.

The underlying eigenvalue problem is first discretized by using finite elements to compute an approximation of the functions $\left(\frac{\partial\phi_j}{\partial\nu}\right)^2$, $1\leq j\leq N$ on $\partial\Omega$. This allows us to consider an approximation of Problem \eqref{truncPbOpt} writing as a finite-dimensional minimization problem under equality and inequality constraints. Hence, we determine an approximation of the Lagrange multipliers $\beta^*=(\beta_j^*)_{1\leq j\leq N}\in\R^N_+$  and $\Lambda$, evaluated by using a primal-dual approach combined with an interior point line search filter method\footnote{The basic idea behind this approach, inspired by barrier methods, is to interpret the discretized optimization problem as a bi-objective optimization problem with the two goals of minimizing the
objective function and the constraint violation, see \cite{IPOPT} for the complete description of the algorithm.}. 
At the end, the set $\left\lbrace a^*=M\right\rbrace$ is plotted by using that
$$
\{\varphi^* > \Lambda\}= \left\lbrace a^*=M\right\rbrace.
$$

Practically speaking, the dual problem is solved with the help of a software package for non-linear optimization, IPOPT  combined with AMPL (see \cite{AMPL,IPOPT}). 

On Figures \ref{TruncSquare} and \ref{TruncEllipse} hereafter, we assume that $\Omega$ denotes respectively a square and an ellipse. Problem \eqref{truncPbOpt} is solved for several values of $N$ and $L$, in the case $M=1$.

Regarding the case of $\Omega=[0,\pi]^2$ (see Figure \ref{TruncSquare}) and denoting by $a_N^*$ a solution of Problem \eqref{truncPbOpt} for given values of $L$ and $M$, we know that $(a_N^*)_{N\in \N^*}$ converges weakly-star in $L^\infty(\Omega)$, up to a subsequence, to a solution of Problem \eqref{defJrelax}, according to Proposition \ref{GammaCvProp}. Moreover, the simulations suggest that $(a_N^*)_{N\in \N^*}$ converges to a constant density, which is in accordance with the analysis of Problem \eqref{defJrelax} in the particular case where $\Omega$ is a rectangle, in Section \ref{prop:rect}. 

\begin{figure}[h]
\begin{center}
\includegraphics[width=4.5cm,height=2.6cm]{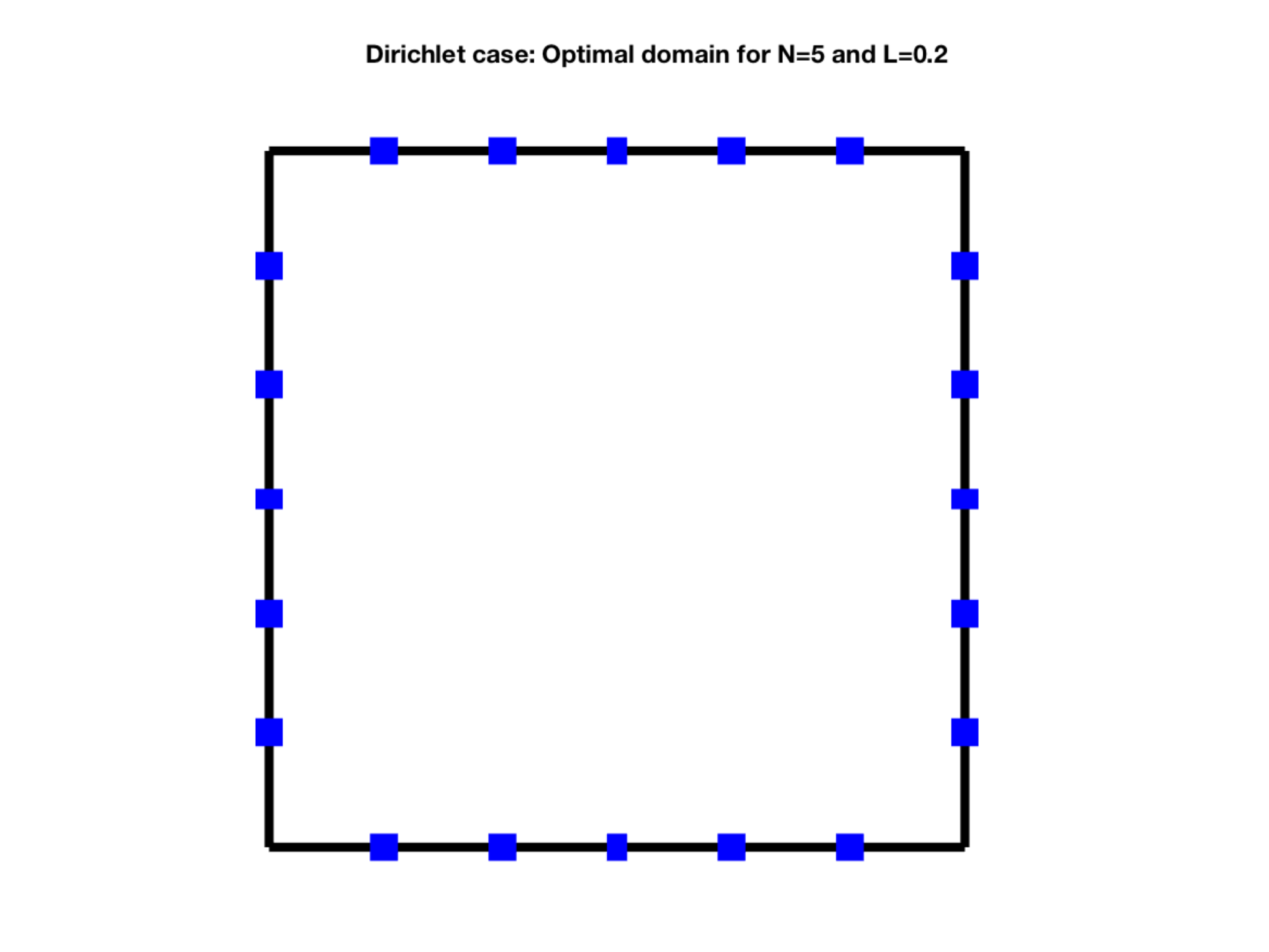}
\includegraphics[width=4.5cm,height=2.6cm]{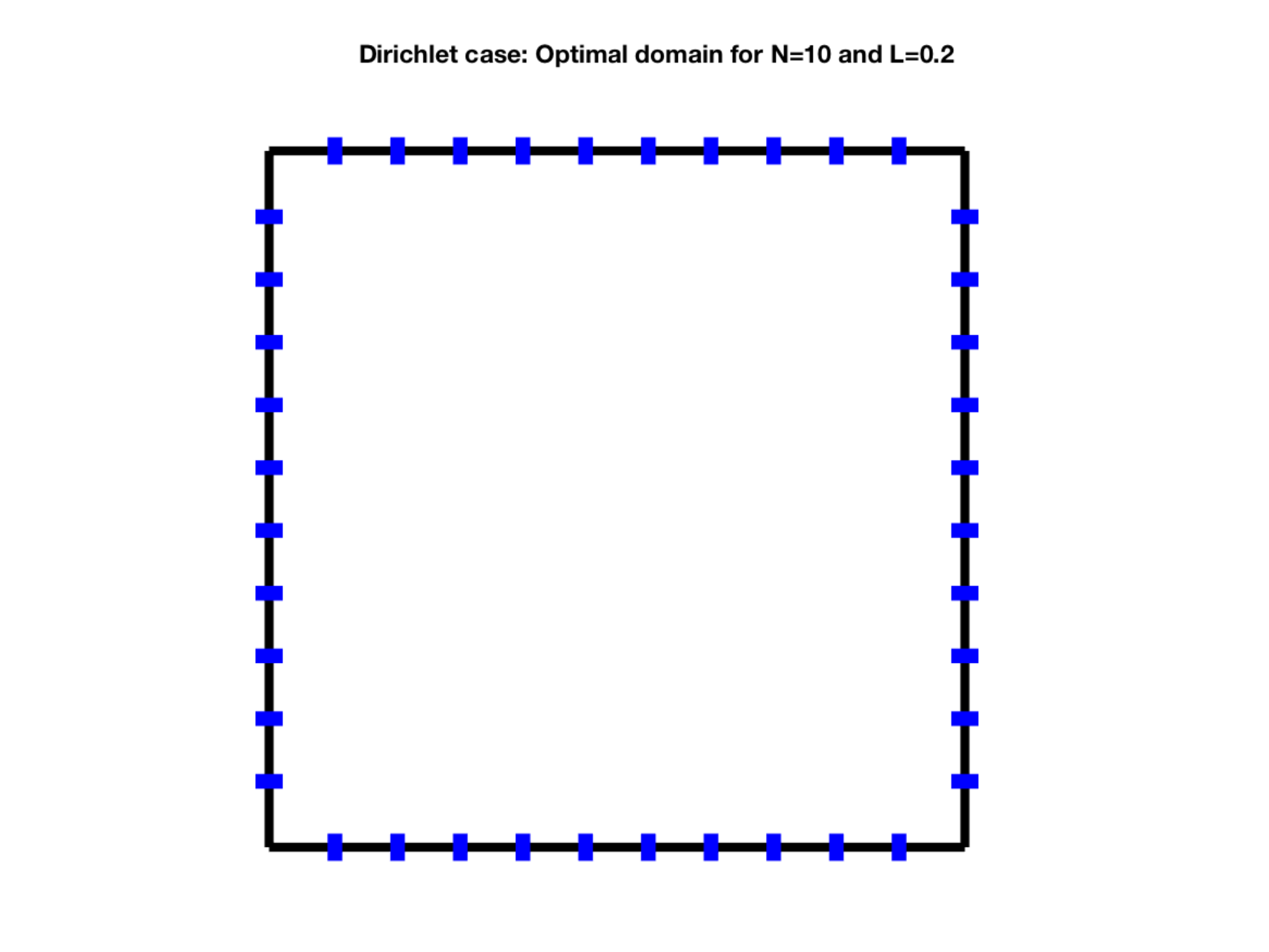}
\includegraphics[width=4.5cm,height=2.6cm]{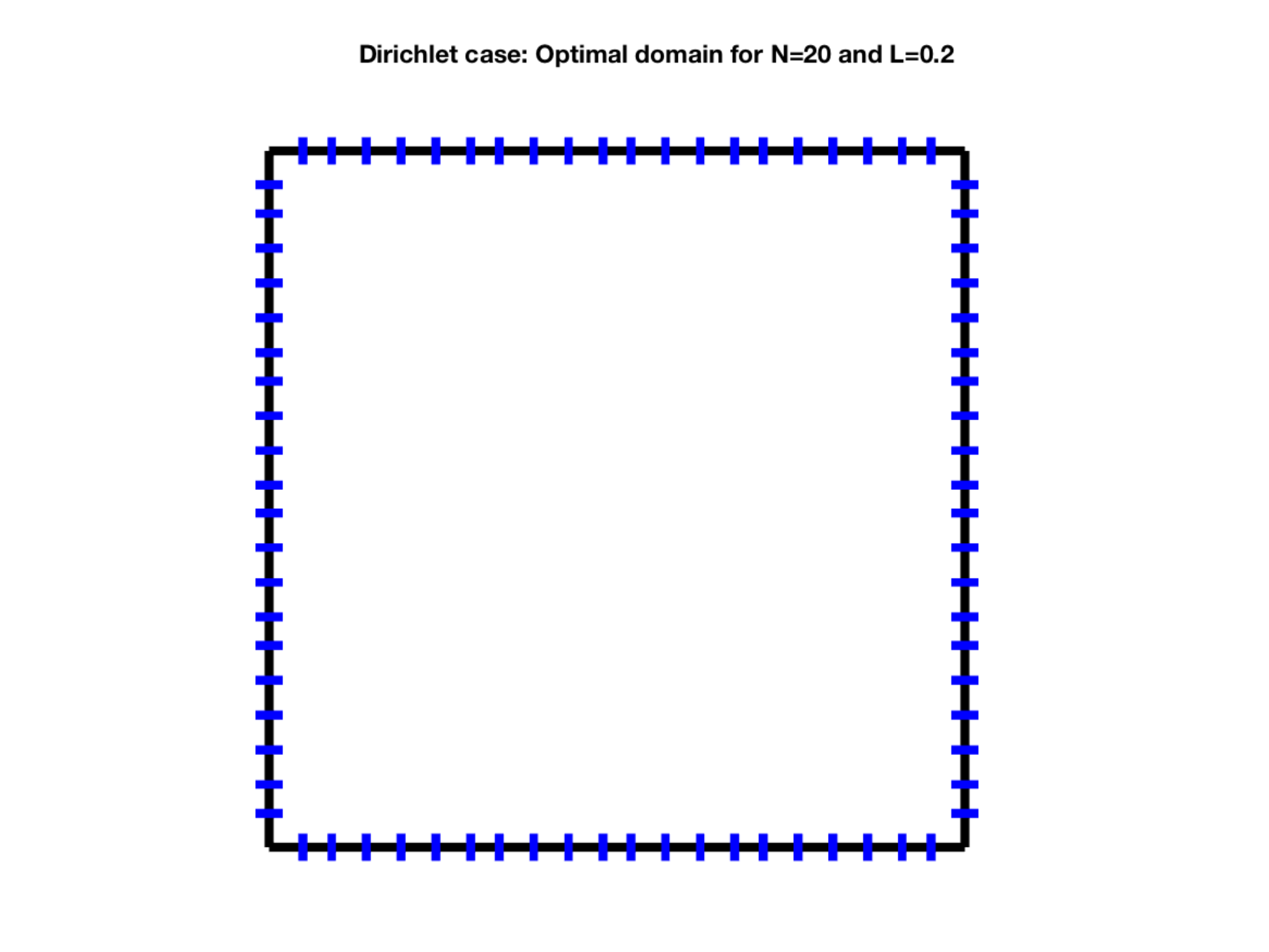}\\
\includegraphics[width=4.5cm,height=2.6cm]{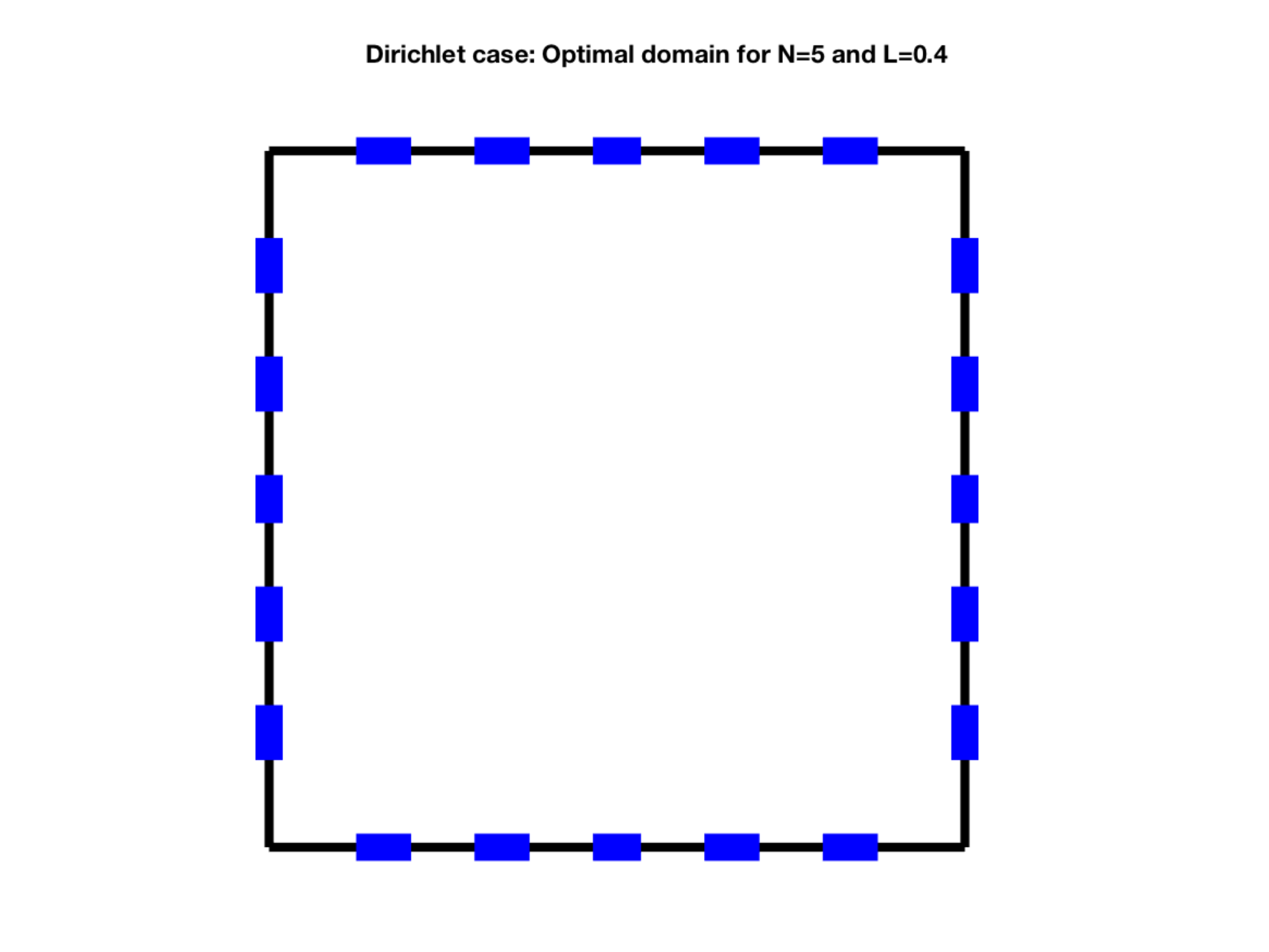}
\includegraphics[width=4.5cm,height=2.6cm]{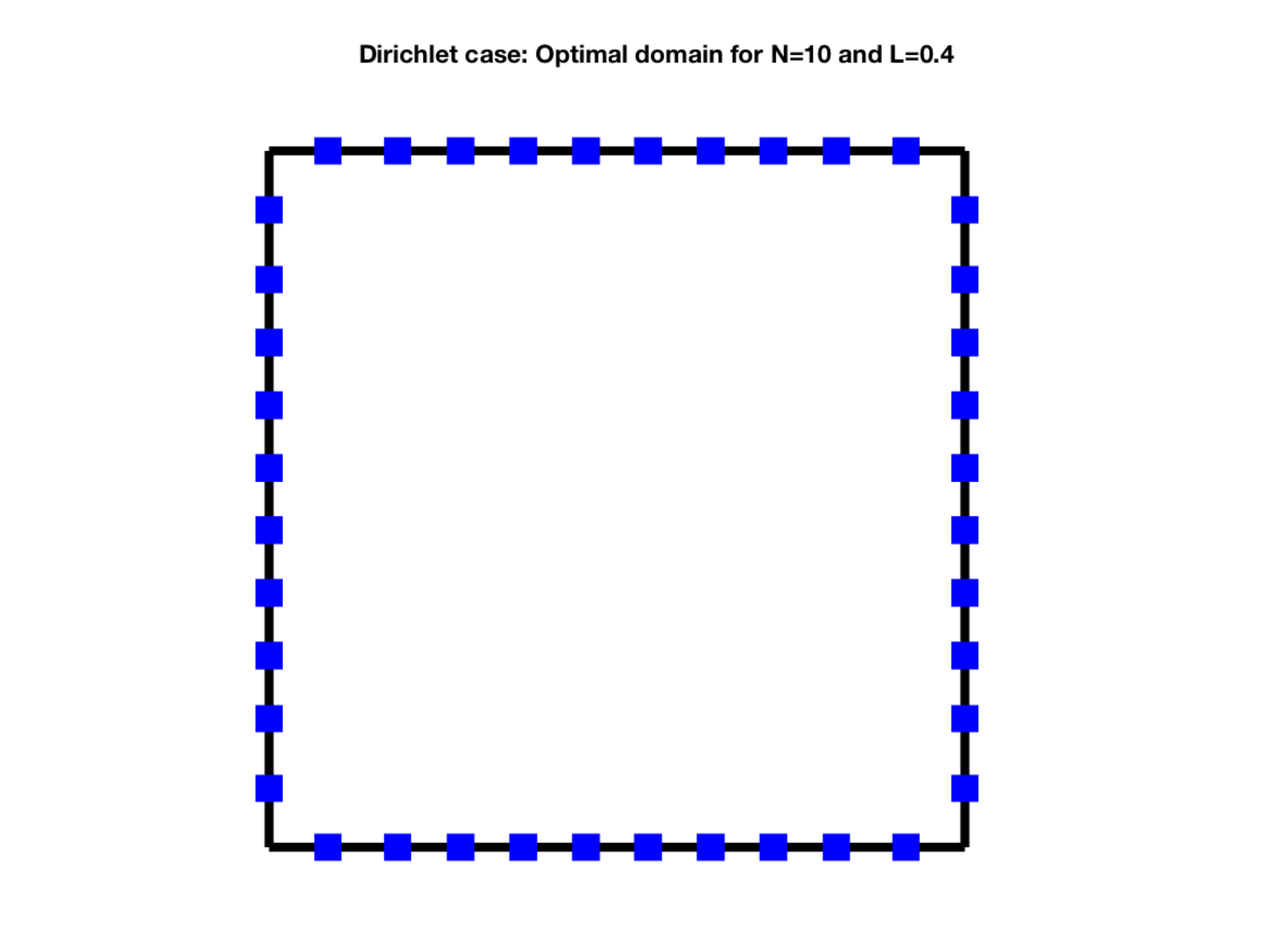}
\includegraphics[width=4.5cm,height=2.6cm]{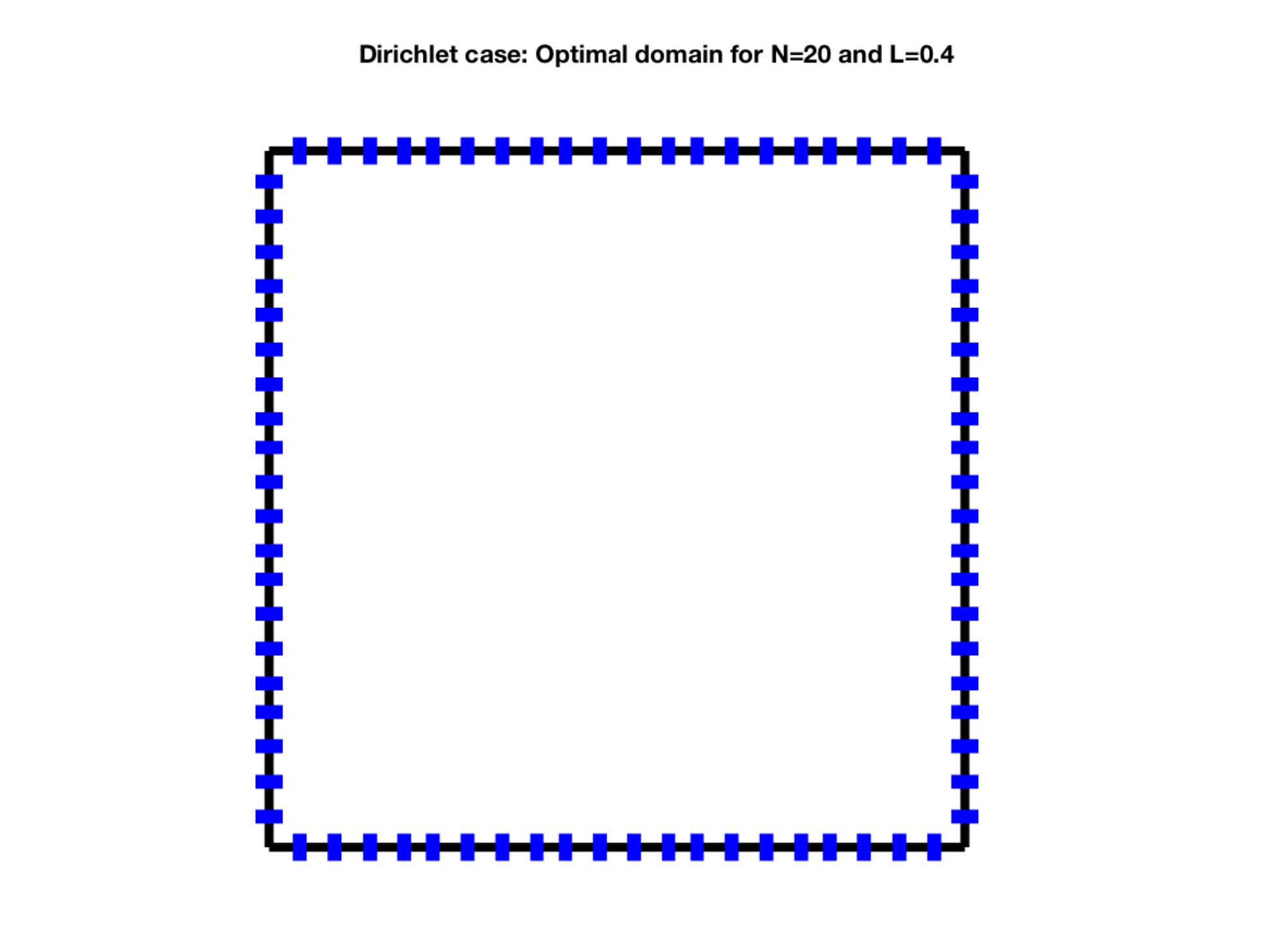}\\
\includegraphics[width=4.5cm,height=2.6cm]{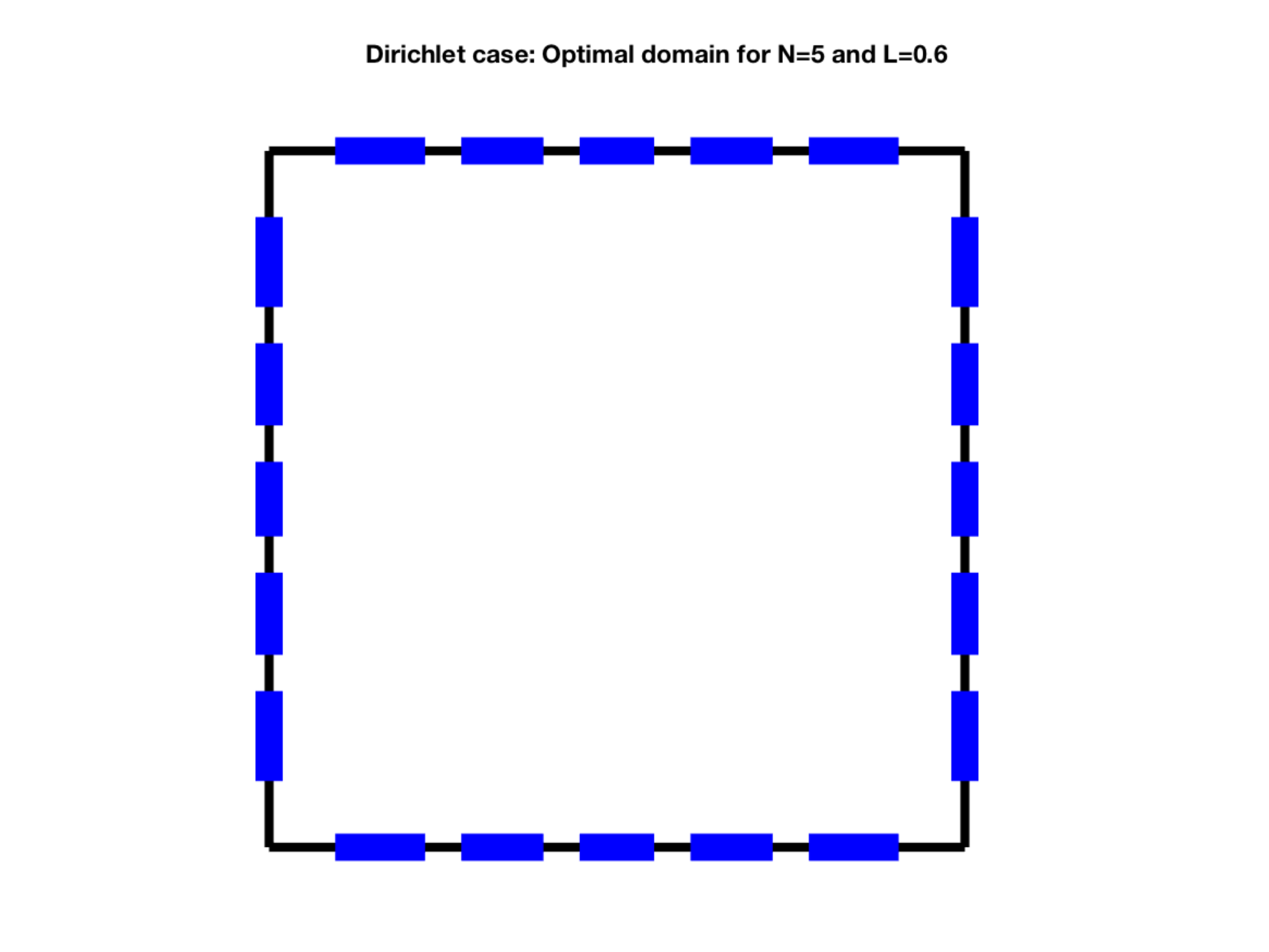}
\includegraphics[width=4.5cm,height=2.6cm]{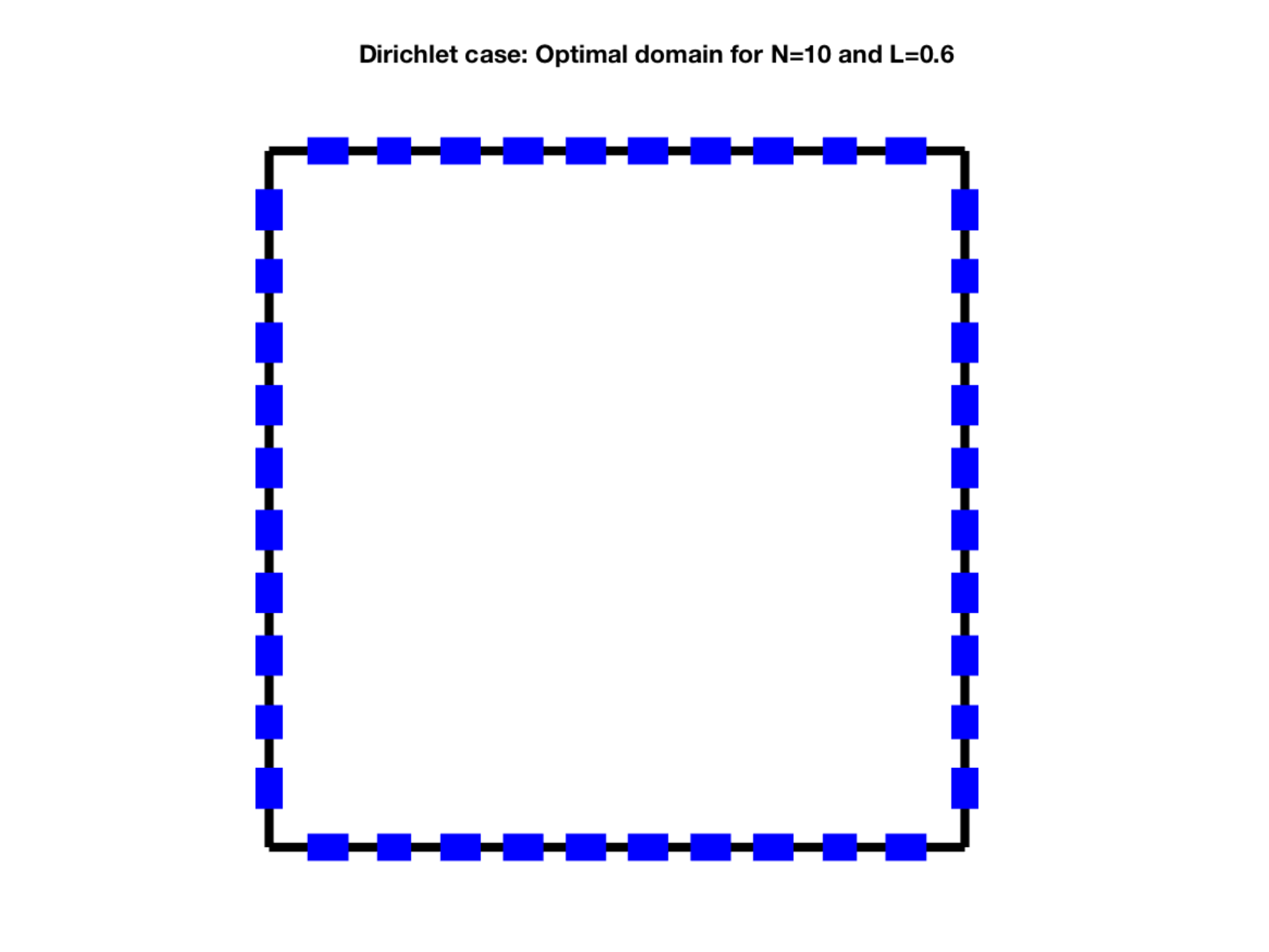}
\includegraphics[width=4.5cm,height=2.6cm]{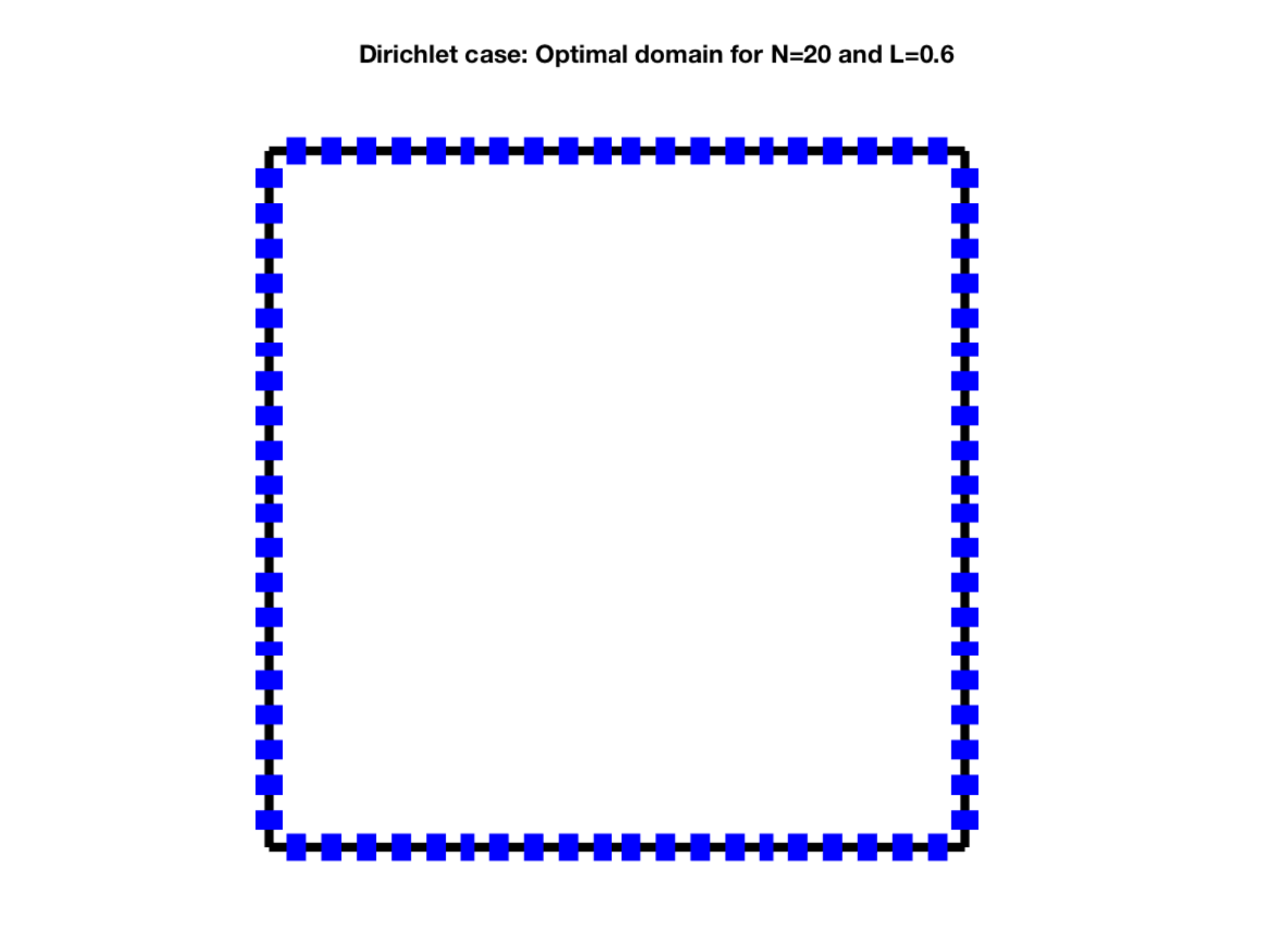}
\caption{$\Omega=[0,\pi]^2$ and $M=1$. Examples of maximizers $a_N^*$ for $J_N$. The bold line corresponds to the set $\Gamma=\{a_N^*=M\}$. Recall that $\int_{\partial \Omega}a \, d\Hn=LM\Hn(\partial\Omega)$ and $\Hn(\Gamma)=\frac{L+1}{2}\Hn(\partial \Omega )$.
Row1: $L=-0.6$ (i.e. $\Hn(\Gamma)=0.2\Hn(\partial\Omega)$); row 2: $L=-0.2$ (i.e. $\Hn(\Gamma)=0.4\Hn(\partial\Omega)$); row 3: $L=0.2$ (i.e. $\Hn(\Gamma)=0.6\Hn(\partial\Omega)$). From left to right: $N=20$, $N=50$, $N=90$.
\label{TruncSquare}}
\end{center}
\end{figure}

On Figure \ref{TruncEllipse}, computations of $a_N^*$ are made for the ellipse having as cartesian equation $x^2+y^2/2=1$, and for $M=1$ and several values of $L$. According to Proposition \ref{GammaCvProp}, $(a_N^*)_{N\in \N^*}$ converges, up to a subsequence, to a solution of Problem \eqref{defJrelax}. Although we were not able to determine all maximizers of Problem \eqref{defJrelax}, and thus, all the closure points of $(a_N^*)_{N\in \N^*}$ we know that one of them is given by
$$
\tilde a_{(0,0)}(x,y)=\frac{L\Hn(\partial\Omega)}{2\sqrt{2}\pi\sqrt{4x^2+y^2}}(2x^2+y^2), \qquad (x,y)\in \Omega,
$$ 
whenever $L\leq L_2^c$ with $\Hn(\partial\Omega)\simeq 7.5845$ and $L_2^c\simeq 0.4142$ (by using Theorem \ref{th:2019} and Proposition \ref{prop:ellcircum}). The profile of $\tilde a_{(0,0)}$ is similar to the left plot of Figure \ref{Fig:densiteEllSA}. All these considerations suggest that, unlike the case of the square, the closure points of $(a_N^*)_{N\in \N^*}$ are non constant densities looking more concentrated around the major axis extremities.  

\begin{figure}[h]
\begin{center}
\includegraphics[width=4cm,height=2.6cm]{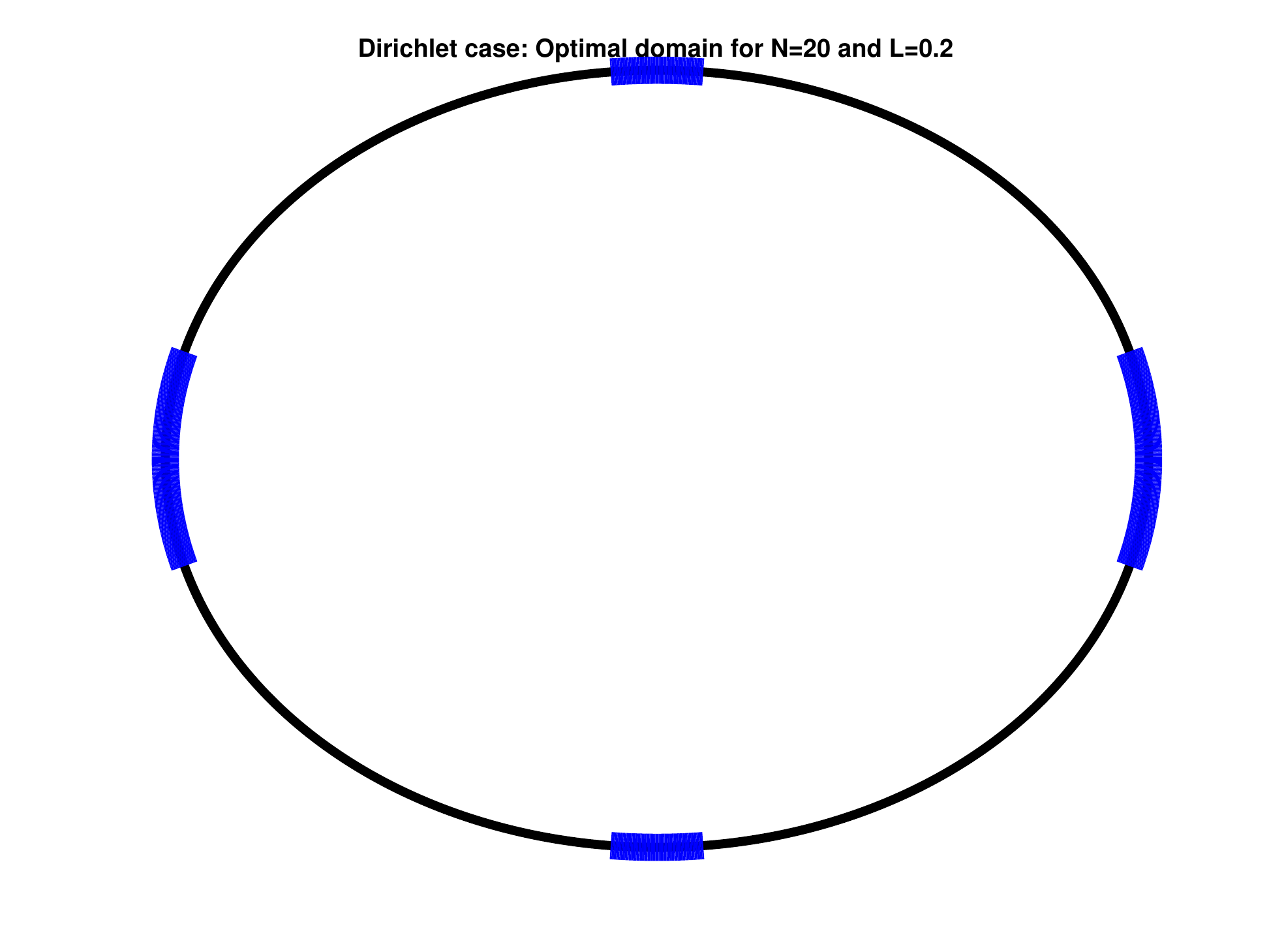}
\includegraphics[width=4cm,height=2.6cm]{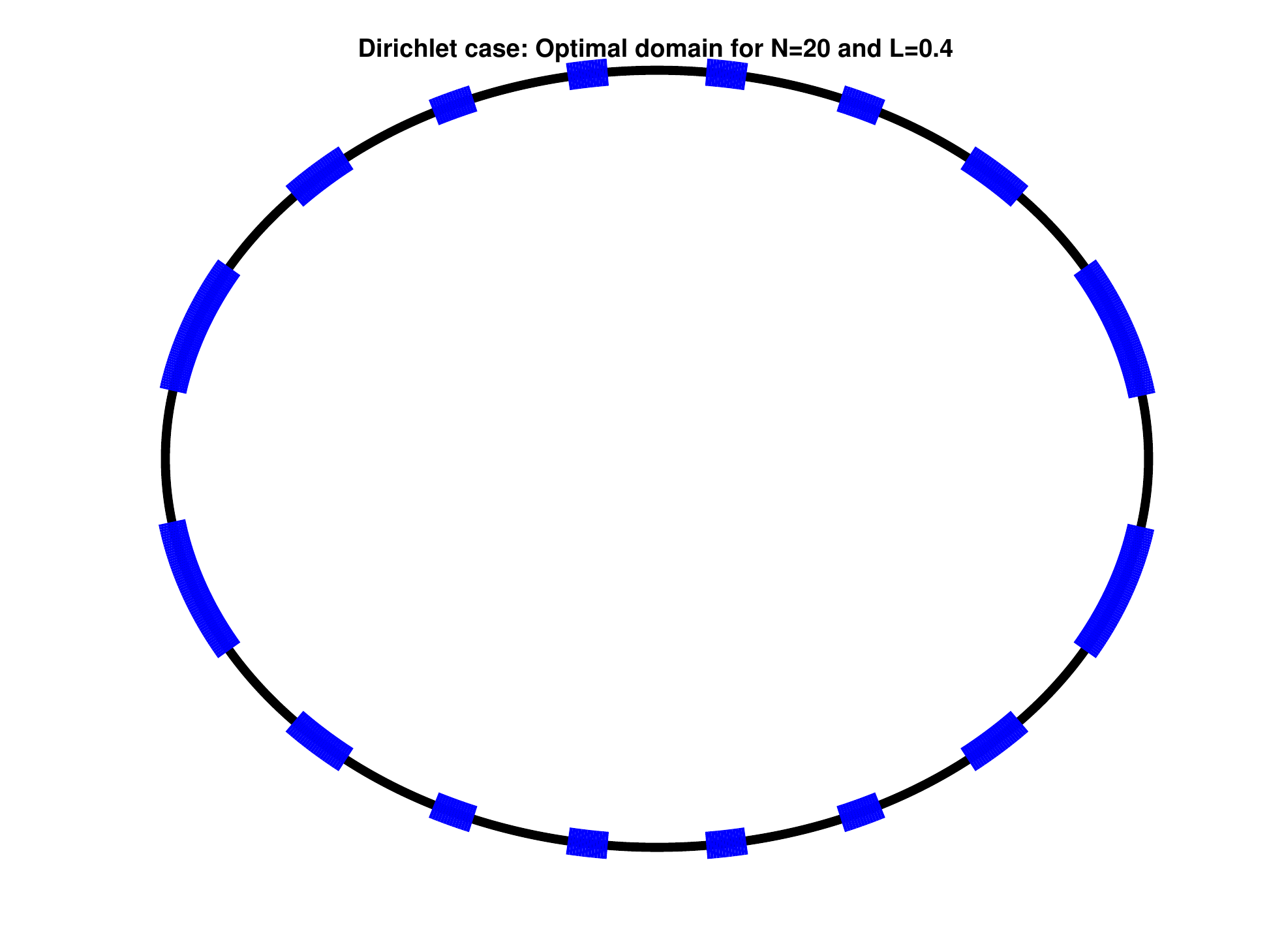}
\includegraphics[width=4cm,height=2.6cm]{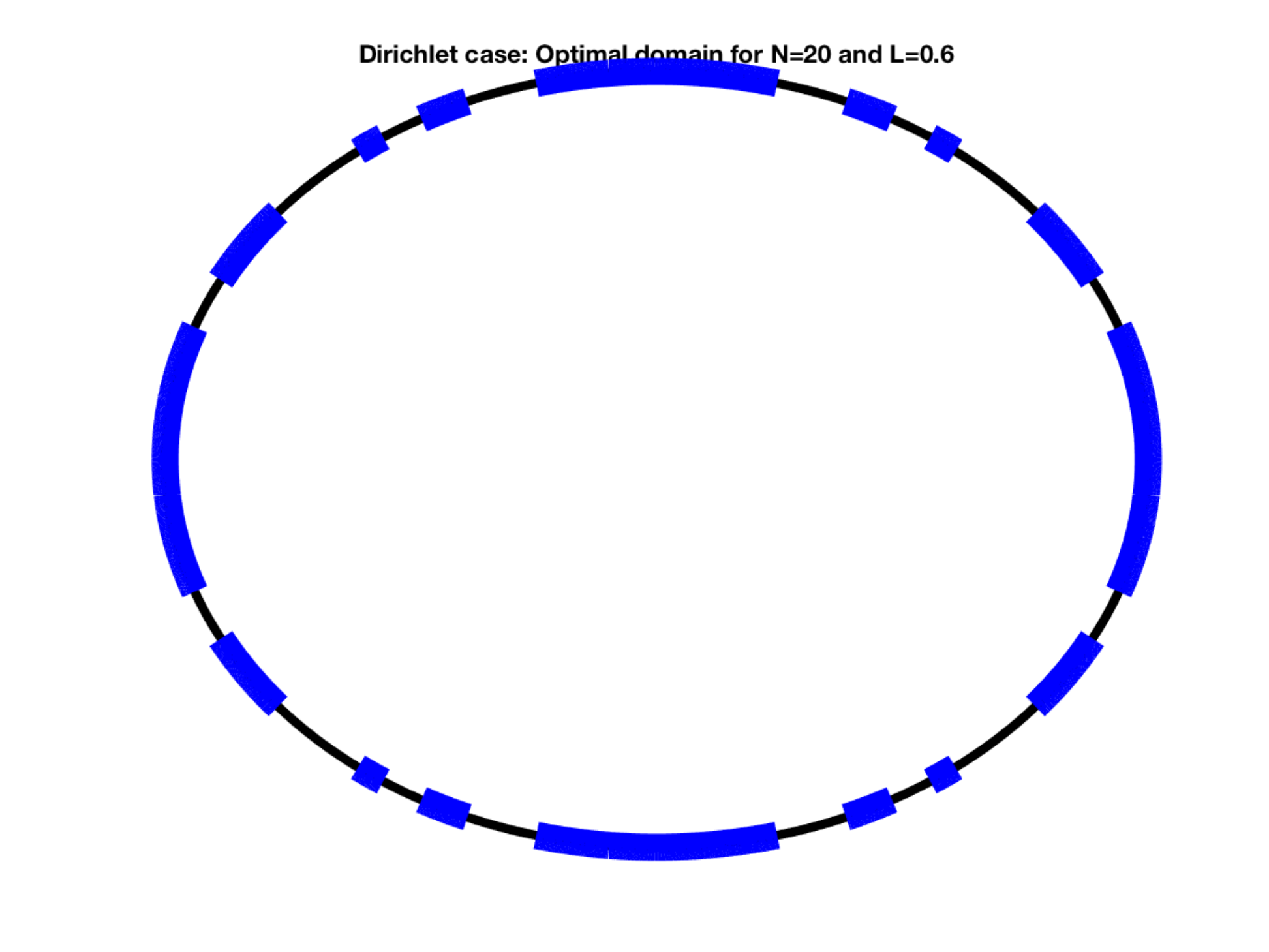}\\
\includegraphics[width=4cm,height=2.6cm]{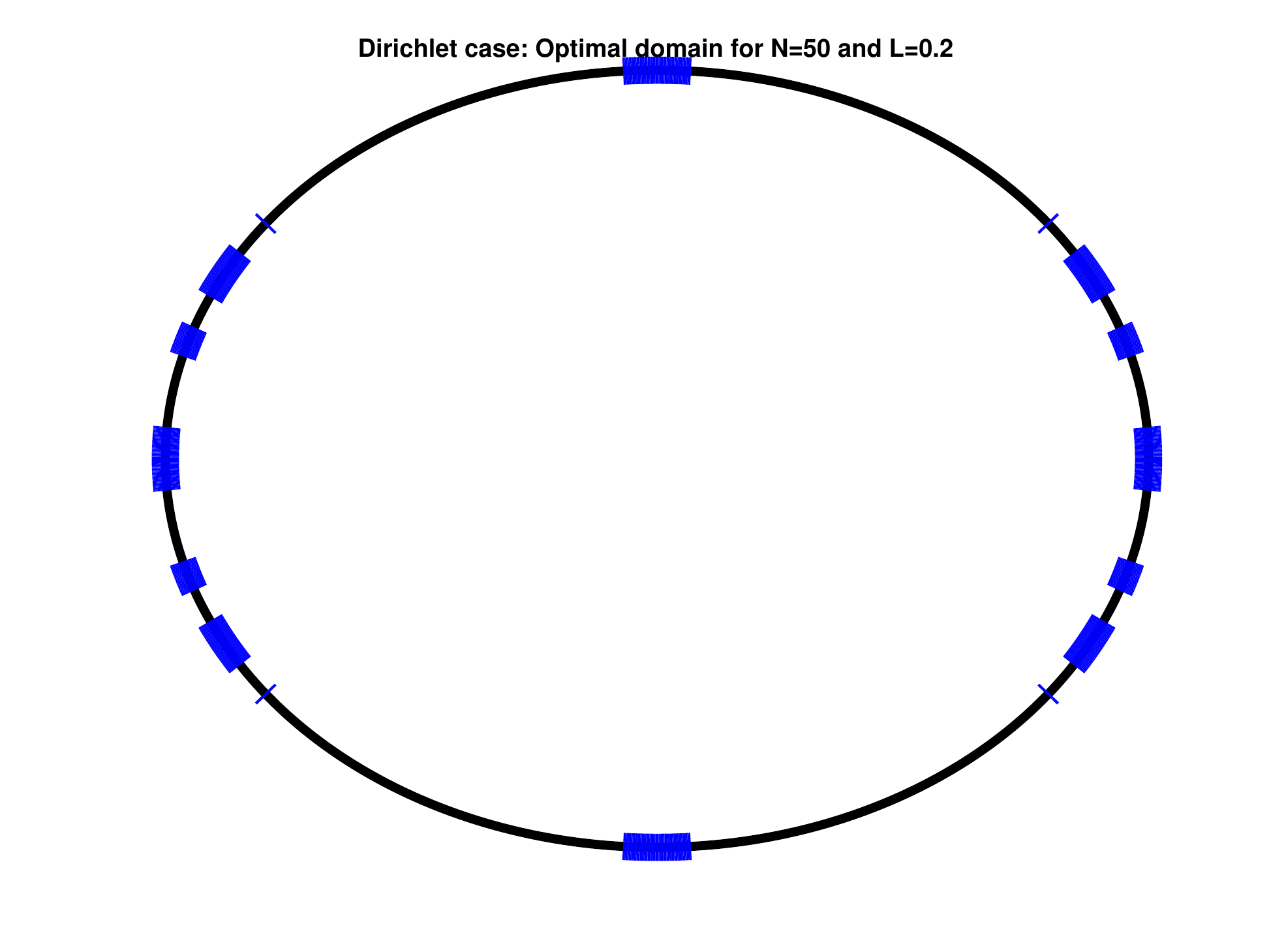}
\includegraphics[width=4cm,height=2.6cm]{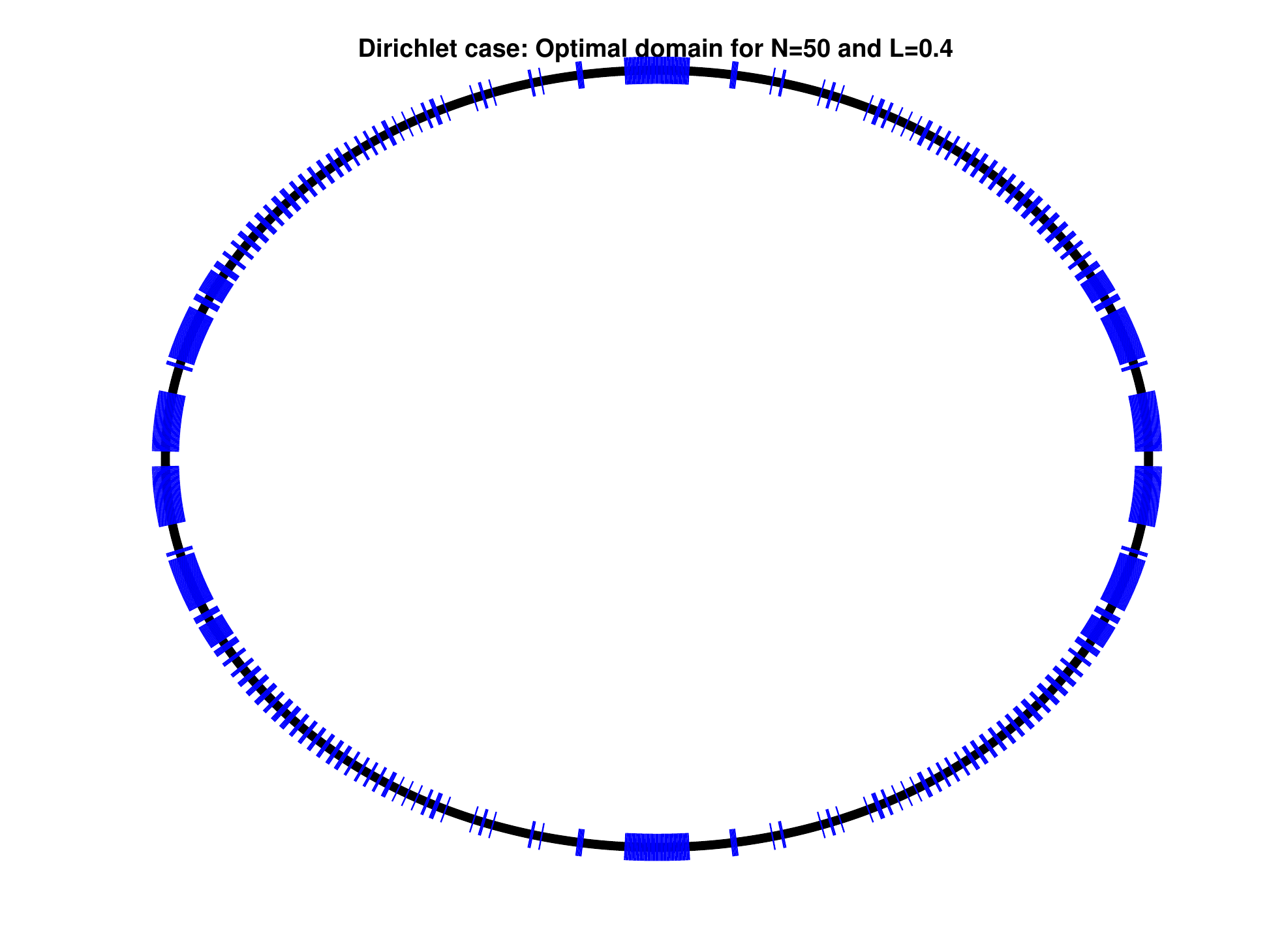}
\includegraphics[width=4cm,height=2.6cm]{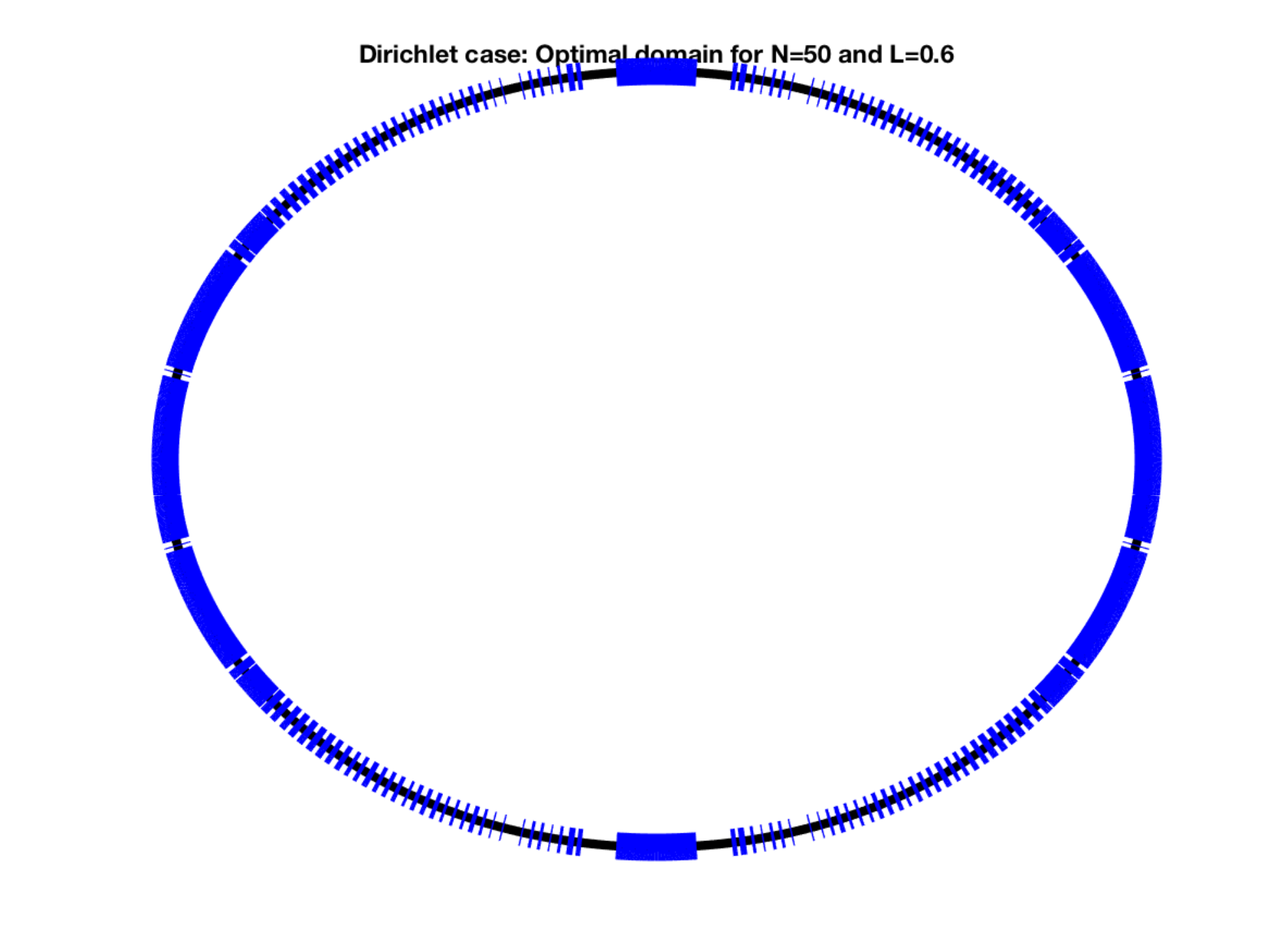}\\
\includegraphics[width=4cm,height=2.6cm]{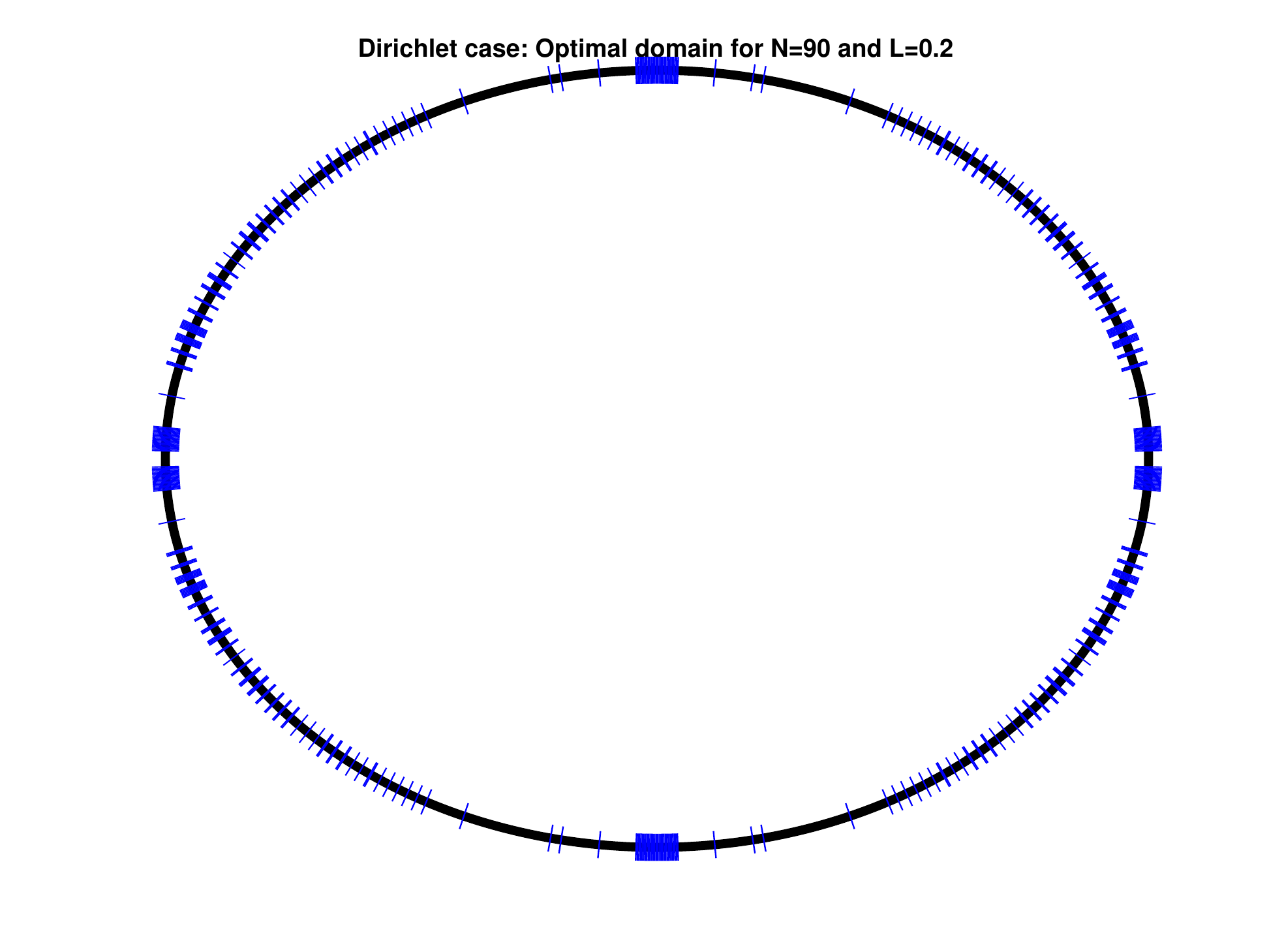}
\includegraphics[width=4cm,height=2.6cm]{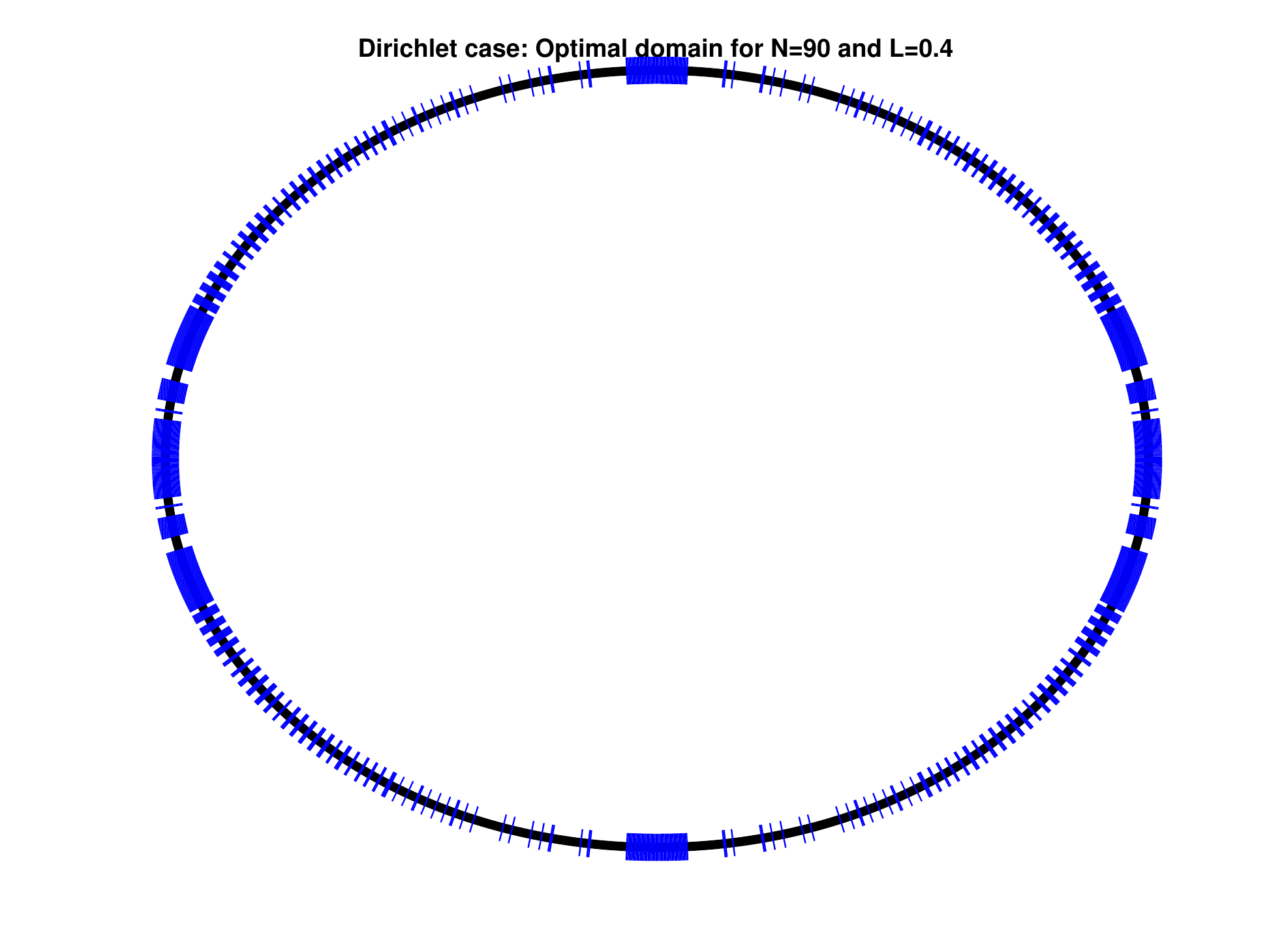}
\includegraphics[width=4cm,height=2.6cm]{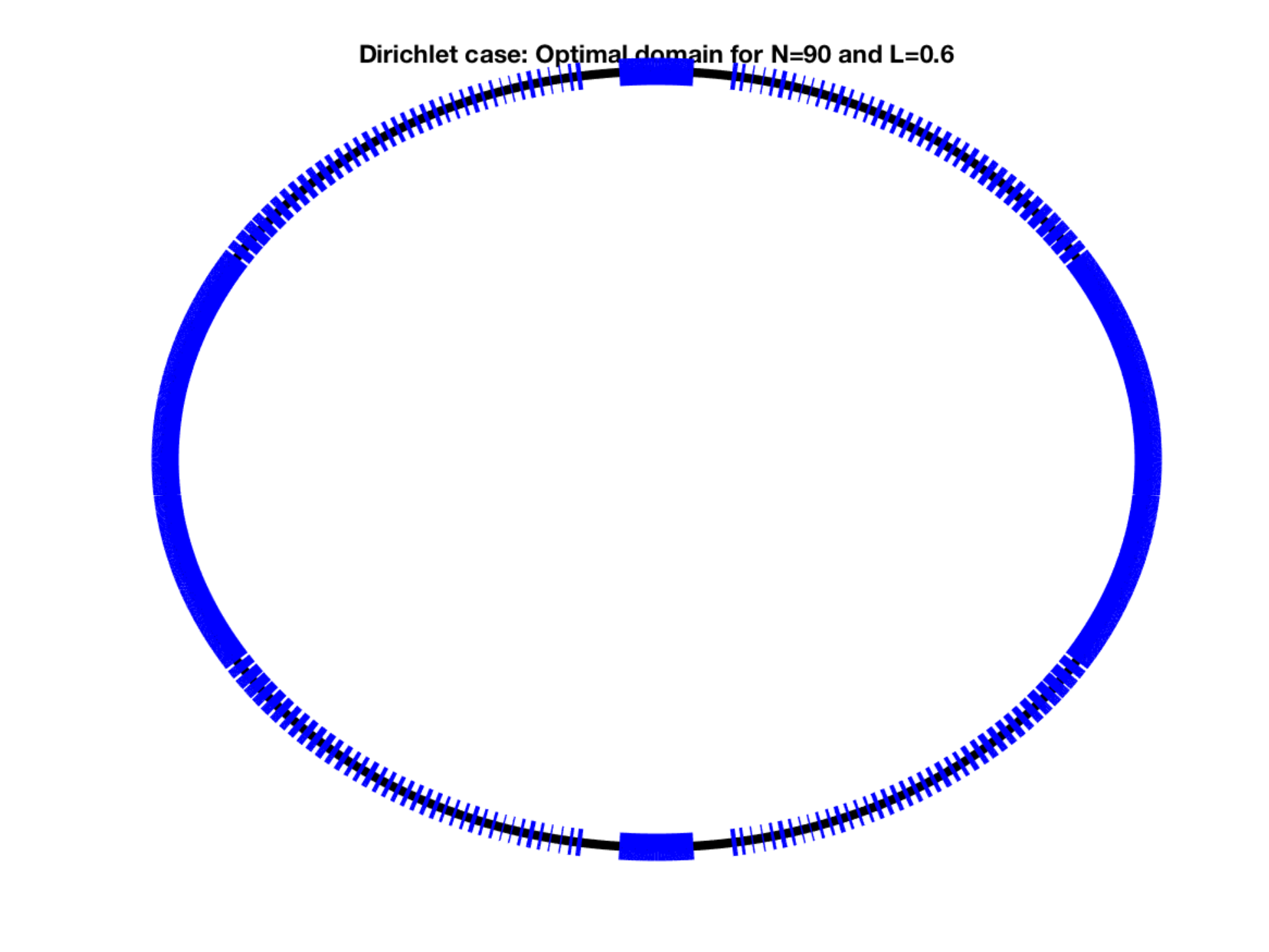}
\end{center}
\caption{$\Omega$ is the ellipse having as cartesian equation $x^2+y^2/2=1$ and $M=1$. Examples of maximizers $a_N^*$ for $J_N$. The bold line corresponds to the set $\Gamma=\{a_N^*=M\}$. Recall that $\int_{\partial \Omega}a \, d\Hn=LM\Hn(\partial\Omega)$ and $\Hn(\Gamma)=\frac{L+1}{2}\Hn(\partial \Omega )$. Row1: $L=-0.6$ (i.e. $\Hn(\Gamma)=0.2\Hn(\partial\Omega)$); row 2: $L=-0.2$ (i.e. $\Hn(\Gamma)=0.4\Hn(\partial\Omega)$); row 3: $L=0.2$ (i.e. $\Hn(\Gamma)=0.6\Hn(\partial\Omega)$). From left to right: $N=20$, $N=50$, $N=90$.\label{TruncEllipse}}
\end{figure}

\subsection{Proof of Theorem \ref{analytics}}\label{sec:prooftheoAnalytics}
Define the simplex 
$
\Pi_N=\left\{{\beta=(\beta_j)_{1\leq j\leq N}\in\R^N_+\text{ such that }\sum_{1\leq j\leq N}\beta_j = 1}\right\}.
$
Using standard arguments from convex analysis, one shows that the optimal design Problem \eqref{truncPbOpt} rewrites
$$
\sup_{ a\in\Ub}J_N(a)=\sup_{ a\in\Ub}\inf_{\displaystyle \beta\in\Pi_N}j(a,\beta)\quad \textnormal{where}\quad j(a,\beta)=\int_{\partial\Omega}a(x)\sum_{1\leq j\leq N}\frac{\beta_j}{\lambda_j}\left( \frac{\partial \phi_j}{\partial \nu}(x)\right)^2d\Hn .
$$
Combining the facts that $\Ub$ and $\Pi_N$ are convex, that the mapping $j$ is linear and continuous with respect to each variable (regarding the first variable, for $\beta\in \Pi_N$, the mapping $j(\cdot,\beta)$ is continuous for the weak star topology of $L^\infty$), one gets the existence of a saddle point $(a^*,\beta ^*)\in \Ub\times \R$ of $j$ solving this problem, according to the Sion minimax theorem (see \cite{Polak}). 

As a result, by introducing the so-called {\it switching function}
\begin{equation}\label{switchfunc}
\varphi^* = \sum_{1\leq j \leq N}\frac{\beta_j^*}{\lambda_j}\left( \frac{\partial \phi_j}{\partial \nu} \right)^2,
\end{equation}
we obtain that
$$
j(a^*,\beta^*)=\max_{a\in\Ub} j(a,\beta^*) \quad \textrm{with}\quad j(a,\beta^*)= \int_{\partial\Omega}a(x)\varphi^*(x)d\Hn .
$$
Solving the optimal design problem of the right-hand side is standard (see for instance \cite[Theorem 1]{PTZobspb1}) and leads to the following characterization of maximizers: there exists a positive real number $\Lambda$ such that $\{\varphi^* > \Lambda\}\subset \left\lbrace a^*=M\right\rbrace$, and $\{\varphi^* < \Lambda\}\subset \left\lbrace a^*=-M\right\rbrace$. 
Moreover, if the set $I = \lbrace -M < a^*< M\rbrace$ has a positive Hausdorff measure, there holds necessarily  $\varphi^*(x)=\Lambda$ a.e on $I$. 
Assuming that $\Omega $ has an analytic boundary, that is to say $\Omega\in\mathcal{D}$ yields that the squares of normal derivatives of eigenfunctions are analytic on $\partial\Omega$ according to \cite{morrey}.

With a slight abuse of notation, we denote by 1 the constant function equal to 1 everywhere on $\partial\Omega$. Introduce the family of functions $\mathcal{F}_N=\left\{1,\left( \frac{\partial \phi_1}{\partial \nu} \right)^2,\dots,\left( \frac{\partial \phi_N}{\partial \nu} \right)^2 \right\}$. The conclusion follows from the following proposition, where we establish that, for a generic $\Omega\in \A_\alpha$, the family $\mathcal{F}_N$ consists of linearly independent functions when restricted to any measurable subset of $\partial\Omega$ of positive $\Hn$-measure. Indeed, it implies that for a generic $\Omega\in \A_\alpha$, the level sets of $\varphi^*$ have zero $\Hn$-measure and therefore, every maximizer $a^*$ is {\it bang-bang}, i.e., equal to $-M$ or $M$ almost everywhere on $\partial\Omega$. 

\begin{proposition}\label{Generic}
Let $\alpha \in\N\backslash\{0,1\}$ and $N\in\N^*$. The set of all domains $\Omega\in \mathcal{A}_\alpha$ for which the family $\mathcal{F}_N$ consists of linearly independent functions is open and dense in $\A_\alpha$.
\end{proposition}

Finally, once we know that, for a generic $\Omega\in \mathcal{A}_\alpha$, every maximizer is {\it bang-bang}, the uniqueness follows from a convexity argument. Indeed, assume the existence of two maximizers $a_1$ and $a_2$. Thus, by concavity of $J_N$, any convex combination of $a_1$ and $a_2$ solves Problem \eqref{truncPbOpt} which is in contradiction with the fact that every maximizer is {\it bang-bang}.

\begin{proof}[Proof of Proposition \ref{Generic}]
In this proof, we follow the method used in \cite[Theorem 1]{PS_ESAIM}.

In the sequel and for every $1\leq j \leq N$, we will denote by $\phi_j$ the extension by $0$ of the $j$-th eigenfunction of the Dirichlet Laplacian to $\R^n$. 

Let us define the function $F:\R^{N^2+N}\longrightarrow\R$ by
$$ F(y_1,\cdots,y_{N(N+1)})=
\operatorname{det} \left(\begin{array}{cccc}
1 & y_1 & \cdots & y_N \\
\vdots & \vdots & & \vdots \\
1 & y_{N^2+1} &  \cdots &y_{N^2+N}\\
\end{array} \right). $$

Our strategy is based on the following remark: assume that the boundary $\partial\Omega$ is analytic. Then, by analyticity of $F$, the property of linear independence of functions of $\mathcal{F}_N$ is equivalent to the existence of $N$ points $x_1$, \dots, $x_N$ in $\partial\Omega$ such that 
\begin{equation}\label{Pn}
F \left( \left( \frac{\partial\phi_1}{\partial\nu}(x_1) \right)^2 , \dots , \left( \frac{\partial\phi_1}{\partial\nu}(x_N) \right)^2 , \dots , \left( \frac{\partial\phi_N}{\partial\nu}(x_{1}) \right)^2 , \dots,  \left( \frac{\partial\phi_N}{\partial\nu}(x_{N}) \right)^2 \right) \neq 0 .
\end{equation}
Indeed, by analyticity of $\partial\Omega$ and according to \cite{morrey}, the squares of normal derivatives of eigenfunctions are analytic on $\partial\Omega$ and therefore, the function 
$$
(x_1,\dots,x_{N})\mapsto F \left( \left( \frac{\partial\phi_1}{\partial\nu}(x_1) \right)^2 , \dots ,  \left( \frac{\partial\phi_1}{\partial\nu}(x_N) \right)^2 , \dots , \left( \frac{\partial\phi_N}{\partial\nu}(x_{1}) \right)^2 , \dots , \left( \frac{\partial\phi_N}{\partial\nu}(x_{N}) \right)^2 \right)
$$
is analytic on $(\partial\Omega)^N$.

As a consequence of these preliminary remarks, the proof of Proposition \ref{Generic} follows from the following lemma. For technical reasons, we need to handle families of domains satisfying moreover a simplicity assumption on the $N$ first eigenvalues. Indeed, forgetting this assumption would allow crossings of eigenvalues branches along the considered paths of domains and would not ensure good regularity properties of the eigenfunctions with respect to the domain $\Omega$.
\begin{lemma}\label{lem:1026}
The set $\Sigma_{P}$ of domains $\Omega$ in $\Sigma_\alpha$ for which the property 
\begin{center}
$\P$ : ``the $N$ first eigenvalues of $\Omega$ are simple and there exists $(x_1,\dots,x_{N})\in (\partial\Omega)^N$ such that \eqref{Pn} holds true'' 
\end{center}
is satisfied, is open and dense in $\Sigma_\alpha$. 
\end{lemma}
We will then infer that the set of the domains $\Omega\in \A_\alpha$ for which the family $\mathcal{F}_N$ consists of linearly independent functions is open in $\A_\alpha$ for the induced topology of $\mathcal{A}_\alpha$ (inherited from $\Ck$). Moreover, the density of this set in $\A_\alpha$ is a consequence of the next approximation result.

Admitting temporarily Lemma \ref{lem:1026}, we now provide the final (standard) argument to conclude the proof. It remains to prove that for every $\Omega\in \Sigma_\alpha$, there exists a domain $\Omega'$ in $\mathcal{A}_\alpha$, arbitrarily close to $\Omega$ for the topology on $ \Sigma_\alpha$. There exist $T\in \Ck$ and an analytic mapping $T'\in \Ck$ such that $\Omega=T(B(0,1))$, $\Omega'=T'(B(0,1))$ and $T$ is arbitrarily close to $T'$ for the $\Sigma_\alpha$-topology. Since $T^{-1}$ (resp. $T'^{-1}$) maps $\Omega$ (resp. $\Omega'$) to $B(0,1)$, the Laplace-Dirichlet eigenvalue problem on $\Omega$ (resp $\Omega'$) is equivalent to the eigenvalue problem for a Laplace-Beltrami operator on $B(0,1)$ relative to the pullback metric $\mathfrak{g}=(g_{ij} (T))_{1\leq i,j\leq n}$ (resp., $\mathfrak{g}'=(g'_{ij} (T))_{1\leq i,j\leq n}$). These operators on $B(0,1)$ are elliptic of second order with analytic coefficients with respect to the metrics. Since the $N$ first eigenvalues of $\Omega$ and $\Omega'$ are simple, standard arguments (see, e.g., \cite{Kato}) about parametric families of operators show that the $N$ first eigenvalues of $\Omega'$ are arbitrarily close to those of $\Omega$, and the $N$ first eigenfunctions of $\Omega'$ are arbitrarily close to those of $\Omega$ for the $\mathcal{C}^\alpha$ topology. As a consequence,  \eqref{Pn} holds also true for $\Omega'$ provided that $T'$ be close enough to $T$ for the $\Ck$-topology.

The desired conclusion follows.
\end{proof}

\begin{proof}[Proof of Lemma \ref{lem:1026}]
Let us show that $ \Sigma_P $ is nonempty, open and dense in $\Sigma_\alpha$. 
 
\paragraph{Step 1: $ \Sigma_P$ is nonempty.}
Let $(\mu_1,\cdots,\mu_d)$ be a non-resonant sequence\footnote{It means that every nontrivial rational linear combination of finitely many of its elements is different from zero} of positive real numbers and $\Omega$ be the orthotope $\displaystyle \Pi_{i=1}^n (0,\mu_i \pi)$. 
Then easy computations based on the fact that for $K=(k_1,\ldots,k_d)\in {\N^*}^d$, the (un-normalized) Laplacian-Dirichlet eigenfunction are the functions
$$
\Omega\ni (x_1,\ldots,x_d)\mapsto 
\prod_{i=1}^d\sin \left(\frac{k_i x_i}{\mu_i}\right),
$$
show that $\Omega$ satisfies $\P$.
Moreover, as proved in \cite{PS_ESAIM}, there exists a sequence $(\Omega_k)_{k\in\N} \in \Sigma_\alpha^\N$ compactly converging\footnote{Recall that a sequence of domains $(\Omega_k)_{k\in \N}\in\Sigma_\alpha^\N$ is said to compactly converge to $\Omega$ if for every compact $K\subset (\Omega\cup\overline{\Omega}^c)$, there exists $k_0\in\N$ such that one has $K\subset(\Omega_k\cup\overline{\Omega}_k^c)$ for every $k\geq k_0$. } to $\Omega$. For $k\in\N$, let us denote by $(\phi_j^k)_{j\in\N^*}$ the Hilbert basis (in $L^2(\Omega_k)$) made of the Laplace-Dirichlet operator eigenfunctions in $\Omega_k$.

Fix $j\in \N^*$. According to \cite{Daners} and since each eigenvalue of $\Omega_k$ is simple, the sequence $(\phi_j^k)_{k\in\N}$ converges uniformly to $\phi_j$ on $\R^n$. 
Consequently $(\Delta\phi_j^k)_{k\in\N}$ converges locally to $\Delta \phi_j$ in $\mathcal{C}^0(\R^n)$. 
Using the elliptic regularity (see \cite{Brezis}), we then infer that $ \phi_j^k $ converges to $\phi_j$ in $H^2(\R^n)$
\footnote{Note that each eigenfunction $\phi_j^k$ is supported by the bounded set $\Omega_k$ which allows to apply the stanfdard elliptic regularity results.}. Denote by $\Lambda^k$ the element of $\Ck$ transforming $\Omega$ into $\Omega^k$.
Since $\Lambda^k $ converges to $\operatorname{Id}$ in $\Ck$, one has
$$ 
\lim_{k\to +\infty}\Vert \phi_j^k \circ \Lambda^k -\phi_j \Vert_{H^2(\R^n)} = 0.
$$

As a consequence, up to a subsequence, $(\nabla \phi_j^k \circ \Lambda^k)_{k\in\N}$ converges to $\nabla \phi_j$ almost everywhere in $\R^n$. 
Finally, according to \cite{morrey}, since $\Omega \in\A_\alpha$, for every $k\in\N$, the vectorial function $\nabla \phi_j^k$ is smooth in $\overline{\Omega}$ and in particular continuous.
It follows that for $k$ large enough, $\Omega_k$ satisfies $\P$.

\paragraph{Step 2: $ \Sigma_P$ is open in $\Sigma_\alpha$.}
Fix $\Omega \in\Sigma_P$, a choice of eigenfunctions $\phi_1$, \dots, $\phi_N$ and $N$ points $x_1$, \dots, $x_N$ such that $\P$ holds true. 
Assume by contradiction that there exists a sequence $(\Omega_k)_{k\in \N}$  of domains in $\Sigma_\alpha\backslash \Sigma_P$ converging to $\Omega$ in $\Sigma_\alpha$. 
Recall that, according for instance to \cite{Daners}, denoting by $\lambda_j^k$ the $j$-th eigenvalue on $\Omega_k$ and by $\phi_j^k$ the associated eigenfunction extended by 0 on $\R^n$, 
the sequence  $(\lambda_j^k)_{k\in\N}$ converges to $\lambda_j$. As a result, the $N$ first eigenvalues of  $\Omega_k$ are simple provided that $k$ be large enough.

In a second time, let us denote by $\Lambda^k$ the element of $\Ck$ transforming $\Omega$ into $\Omega_k$. For $k\in\N$, introduce the analogous of $\mathcal{F}_N$ for $\Omega_k$, namely 
\begin{equation}\label{DefineFNK}
\mathcal{F}_N^k = \left\lbrace 1,( \nabla \phi_1^k\circ\Lambda^k)^2,\cdots,( \nabla \phi_N^k\circ\Lambda^k)^2 \right\rbrace .
\end{equation}
We claim that 
\begin{equation}\label{CVUnif}
\text{each element of }\mathcal{F}_N^k \text{ converges pointwise to } \mathcal{F}_N \text{ in $\Omega$, as $k$ tends to $+\infty$}.
\end{equation}
Indeed, this is obtained by using a similar argument as previously, by reinterpreting the convergence of $(\Omega_k)_{k\in\N}$ in $\Sigma_\alpha$ in terms of a parametric family of operator with smooth coefficients.

It follows that for $k$ large enough, one has
$$
F( (\nabla \phi_1^k (\Lambda_j^k(x_1)))^2 ,\cdots, (\nabla \phi_N^k (\Lambda_j^k(x_N)))^2 ) \neq 0.
$$
since $F( (\nabla \phi_1^k (\Lambda_j^k(x_1)))^2 ,\cdots, (\nabla \phi_N^k (\Lambda_j^k(x_N)))^2 )$ converges to $ F( (\nabla \phi_1 (x_1))^2 ,\cdots, (\nabla \phi_N (x_N))^2 ) \neq 0$.

One gets a contradiction since one has $\Omega_k$ belongs to $\Sigma_P$.

\medskip

\paragraph{Step 3: $ \Sigma_P$ is dense in $\Sigma_\alpha$.}
Fix $\Omega_0\in\Sigma_P$ with the corresponding $x_1,\cdots, x_N$, and $\Omega_1\in\Sigma_\alpha$.  
According to \cite{PS_ESAIM,Teytel}, there exists an analytic curve $[0,1]\ni t\mapsto \Lambda_t$ of $\mathcal{C}^m$-diffeomorphisms such that $\Lambda_0$ is equal to the identity, 
$\Lambda_1(\Omega_0)=\Omega_1$ and  
every domain $\Omega_t=\Lambda_t(\Omega_0)$ has simple spectrum for $t$ in the open interval $(0,1)$ and for every $j\in\llbracket 1,N \rrbracket$, the mapping $t\mapsto \lambda_j^t$ is analytic.. 
Let us introduce $\mathcal{F}_N^t$ the analogous of $\mathcal{F}_N$ for $\Omega_t$.

Moreover, by using analytic perturbation theory arguments (see \cite{Kato}), we also know that there exists a choice of eigenfunctions $(\phi_j^t)_{j\in\llbracket 1,N \rrbracket}$ such that,
$ \phi_j^t \circ \Lambda_j^t$ varies analytically with respect to $t$ in $\mathcal{C}^\alpha(\overline{\Omega_0})$.
Consequently $\Delta \phi_j^t \circ \Lambda^t$ is analytic with respect to $t$ in $\mathcal{C}^\alpha(\overline{\Omega_0})$.
In particular, it follows that $\phi_j^t\circ \Lambda^t$ is analytic with respect to $t$ in $H^2(\Omega_0)$ (by using elliptic regularity).
Using the same arguments as in step 1, one has that $\nabla \phi_j^t\circ \Lambda^t$ is analytic with respect to $t$ successively in $L^2(\Omega_0)$ and almost everywhere in $\Omega_0$.

We then infer that the mapping
$$ \Psi : t \mapsto F(1, (\nabla \phi_1^t (\Lambda^t(x_1)))^2 ,\cdots, (\nabla \phi_N^t (\Lambda^t(x_N)))^2 )$$
is analytic from $[0,1]$ to $\R$. Since $\Psi(0)\neq 0$, we get that, except for a finite number of values in $[0,1]$, $\Psi(t)\neq 0$, and thus $\Omega_t$ is in $\Sigma_P$. 
In particular $\Sigma_P$ is dense in $\Sigma_\alpha$.
\end{proof}


\section{Solving of the optimal design problem \eqref{defJrelax}}
\label{section2}

\subsection{Preliminaries}

Let us first prove the existence of solutions for Problem \eqref{defJrelax}.

\begin{lemma}\label{lem:existPinfty}
Assume that $\Omega$ is either convex or has a $\mathcal{C}^{1,1}$ boundary. Then, Problem \eqref{defJrelax} has at least a solution.
\end{lemma}
\begin{proof}
For every $j\in\N^*$, the functional $a\in\overline{\mathcal{U}}_L\mapsto \frac{1}{\lambda_j}\int_{\partial \Omega} a\left(\frac{\partial \phi_j}{\partial \nu}\right)^2 \, d\Hn$ is linear and continuous for the $L^\infty (\partial\Omega,[-M,M])$ weak-star topology. 
Thus, the functional $\Ub\ni a\mapsto J_\infty(a)$ is upper semi-continuous as the infimum of continuous linear functionals. Since $\Ub $ is compact for the $L^\infty(\partial\Omega,[-M,M])$ weak-star topology, the existence of an optimal density $a^*$ in $\Ub$ follows.
\end{proof}

The optimal design problem we will investigate in the sequel is motivated by the following observation about the maximizers given by \eqref{def:tildeax0}.

\begin{lemma}{($L^\infty$-norm of maximizers)}\label{prop:Linftynormmax}
Assume that $\Omega$ has a $\mathcal{C}^1$ boundary. Then, one has
\begin{equation}\label{dm11h43}
\frac{\operatorname{diam}(\Omega)}{2}\leq \inf_{x_0\in\R^n} \sup_{x\in\overline{\Omega}} |\left\langle x-x_0 , \nu(x) \right\rangle|=\inf_{x_0\in\R^n} \sup_{x\in\overline{\Omega}} \Vert x-x_0\Vert\leq R(\Omega),
\end{equation}
where $\operatorname{diam}(\Omega)$ (resp. $R(\Omega)$) denotes the diameter (resp. the circumradius) of $\Omega$.

As a consequence, the lowest $L^\infty$-norm among all maximizers of Problem \eqref{maxJapropsssign} satisfies
$$
\frac{\operatorname{diam}(\Omega)p_0}{2n|\Omega|}\leq \inf_{x_0\in \R^n}\Vert \tilde{a}_{x_0}\Vert_{L^\infty(\partial\Omega)}\leq \frac{p_0}{n|\Omega|}.
$$
\end{lemma}
\begin{proof}
Let us first show the equality. Note that 
$$
\inf_{x_0\in\R^n} \sup_{x\in\overline{\Omega}} |\left\langle x-x_0 , \nu(x) \right\rangle|\leq \inf_{x_0\in\R^n} \sup_{x\in\overline{\Omega}} \Vert x-x_0\Vert,
$$
by maximality and by using the Cauchy-Schwarz inequality. Let $x_0\in \R^n$ and let $x^*$ solve the problem $\sup_{x\in\overline{\Omega}} \left\langle x-x_0 , \nu(x) \right\rangle$. Then, one has necessarily $\nu(x^*)=(x^*-x_0)/\Vert x^*-x_0\Vert$ and $\max_{x\in\overline{\Omega}} \left\langle x-x_0 , \nu(x) \right\rangle=\Vert x^*-x_0\Vert$. Indeed, this is easily inferred by writing
$$
x^*-x_0= \left\langle x^*-x_0 , \nu(x^*) \right\rangle\nu(x^*)+t(x^*),
$$
where $t(x^*)$ denotes a vector of the tangent hyperplane to $\Omega$ at $x^*$. Furthermore, the first order optimality conditions write $\langle x^*-x_0,h\rangle \leq 0$ for every element $h$ of the cone of admissible perturbations. Finally, noting that $h=t(x^*)$ is an admissible perturbation yields that necessarily $t(x^*)=0$ and the expected conclusion follows. 

\medskip

To prove the right-hand side inequality, it suffices to choose for $x_0$ the center $O$ of the circumradius and we get
$$ 
 \inf_{x_0\in \R^n} \max_{x\in\overline{\Omega}} \Vert x-x_0 \Vert \leq \max_{x\in\overline{\Omega}} \Vert x-O \Vert \leq R(\Omega).
$$
To prove the left-hand side inequality, let us introduce two points $A$ and $B$ such that $\operatorname{diam}(\Omega)=[AB]$. 
Since $\Omega$ has a $\mathcal{C}^1$ boundary, it follows that $\nu(A)=-\nu(B)=\frac{\overrightarrow{BA}}{BA}$, and therefore
\begin{eqnarray*}
\inf_{x_0\in\R^n} \max_{x\in\overline{\Omega}} |\left\langle x-x_0 , \nu(x) \right\rangle| &\geq & \inf_{x_0\in\R^n}\max_{x\in \{A,B\}} |\left\langle x-x_0 , \nu(x) \right\rangle|\\
&=& \inf_{x_0\in\R^n}\max \left\{|\left\langle A-x_0 , \nu(A) \right\rangle|,|\left\langle B-x_0 , \nu(A) \right\rangle|\right\}\\
&=&  \inf_{x_0\in\R^n}\max_{t\in [0,1]} \left\langle t(A-x_0)-(1-t)( B-x_0) , \nu(A) \right\rangle\\
&\geq & \frac{1}{2}\left\langle A-B, \nu(A) \right\rangle=\frac{\operatorname{diam}(\Omega)}{2}.
\end{eqnarray*}

Finally, the last claim follows directly from the expression of $\tilde a_{x_0}$ given by \eqref{def:tildeax0}. 
\end{proof}
As a consequence, the lowest $L^\infty$-norm among all maximizers can be either small or large depending on the shape of $\Omega$. Indeed, denoting by $\mathcal{D}$ the set of bounded connected domains $\Omega$ of $\R^n$ having as circumradius $R(\Omega)=1$,  there holds
$$
\inf_{\Omega\in \mathcal{D}}\frac{1}{|\Omega|}=\frac{1}{|B_n|}\qquad \text{and}\qquad \sup_{\Omega\in \mathcal{D}}\frac{1}{|\Omega|}=+\infty ,
$$
where $B_n$ denotes the Euclidean $n$-dimensional ball with circumradius 1. The first equality is a consequence of the standard isoperimetric inequality and the second one is obtained by considering for $\Omega$ a sequence of particular lenses (namely, ellipses/ellipsoids having for circumradius 1 and a semi-minor axis tending to 0, see Fig. \ref{Fig:Lentille}). 

\medskip

It follows that, given $\Omega$ in $\mathcal{D}$, the solutions $\tilde a_{x_0}$ of Problem \eqref{maxJapropsssign} may have an arbitrarily large $L^\infty$ norm. In view of reducing the complexity of maximizers, it is relevant to deal with density functions $a(\cdot)$ that are essentially bounded by a positive constant $M$. This justifies to consider Problems \eqref{defJrelax} and \eqref{mainPbOpt}.

\subsection{Optimality of Rellich functions}\label{sec:soveRelaxopt}

Next results are devoted to the computation of the optimal value for Problem \eqref{defJrelax} under adequate geometric assumptions on $\Omega$. 

\begin{theorem}\label{OptimalValue}
Let $\Omega$ be a bounded connected  domain of $\R^n$ either convex  or with a $\mathcal{C}^{1,1}$ boundary. 
Let $a^*$ be a solution of Problem \eqref{defJrelax}. Then:
\begin{itemize}
\item one has necessarily $J_\infty(a^*)\leq \frac{2L\Hn (\partial\Omega)}{n|\Omega|}$.
\item one has 
$$
\max_{a\in\Ub}J_\infty(a)=\frac{2L\Hn (\partial\Omega)}{n|\Omega|}.
$$
if, and only if $L\in [-L^{c}_{n},L^{c}_{n}]$, where $L^{c}_{n}$ is given by 
\begin{equation}\label{md11h30}
L^{c}_{n}=\min \left\{ 1,\frac{n|\Omega|}{\Hn (\partial\Omega) \displaystyle \inf_{x_0\in\R^n}\ell_{\partial\Omega}(x_0)}\right\},\text{ with }\ell_{\partial\Omega}(x_0)=\underset{x\in \partial\Omega}{\operatorname{sup\ ess}}\left|\langle x-x_0,\nu(x)\rangle\right|.
\end{equation}
\end{itemize}
Furthermore, in the case where $\partial\Omega$ is $\mathcal{C}^{1,1}$, then the expression of $L^{c}_{n}$ simplifies into \eqref{Lstar}.
\end{theorem}

This proof of this result strongly uses Theorem \ref{prop:metz1756}. Both proofs are postponed to Section \ref{sec:proofpropOptvalue}. Let us comment on the condition \eqref{md11h30} and the function $\ell_{\partial\Omega}$ that is involved. This condition is equivalent to the existence of a Rellich function $\tilde a_{x_0}$ (see Def. \ref{def:Rellichadmf}) with $x_0\in \overline{\Omega}$, belonging to the set $\Ub$.
Moreover, one has the following (partial) characterization, also proved in Section \ref{sec:proofpropOptvalue}.

\begin{proposition} \label{prop:ellcircum}
Assume that $\Omega$ has a $\mathcal{C}^1$ boundary. Then the expression of $L^c_{n}$ simplifies into \eqref{Lstar}. Moreover, if ones assumes that the intersection between $\overline{\Omega}$ and its circumsphere reduces to two points, one has
$$
\inf_{x_0\in \R^n}\ell_{\partial\Omega}(x_0)=R(\Omega).
$$
\end{proposition}

Notice that the conclusion of Proposition \ref{prop:ellcircum} is not true in general. Indeed, if one considers a domain made of a flat triangle whose edges have been smoothed with unit circumradius (for instance obtained from the triangle plotted on Fig. \ref{Fig:triangledisk}), by choosing a particular test point $x_0$ inside $\Omega$, one has $\inf_{x_0\in \R^n}\ell_{\partial\Omega}(x_0)\leq \operatorname{diam}(\Omega)<R(\Omega)$.
\begin{figure}[h!]
\begin{center}
\includegraphics[height=5cm]{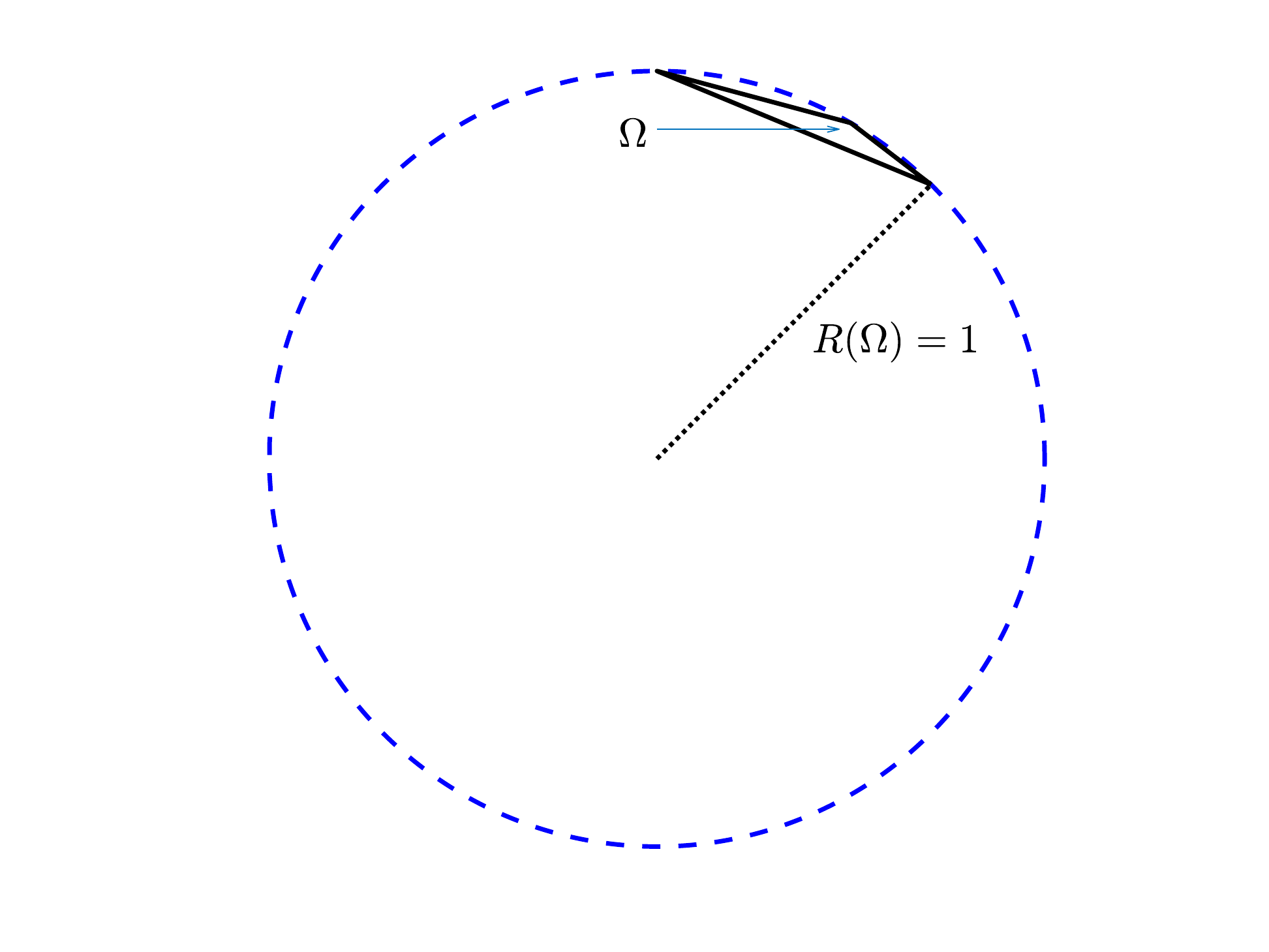}
\caption{A flat triangle with unit circumradius. \label{Fig:triangledisk}}
\end{center}
\end{figure}

In Section \ref{Invest}, we will investigate the particular cases where $\Omega$ is either a rectangle, a disk or an angular sector in $\R^2$. In particular, we will explicitly compute at the same time the critical value $L^c_{n}$ in such cases as well as the optimal value for the convexified problem \eqref{defJrelax} when $L>L^c_{n}$, in other words when the assumptions of Theorem \ref{OptimalValue} are not satisfied anymore. 
\subsection{Proofs of Theorems \ref{prop:metz1756}, \ref{OptimalValue} and Proposition \ref{prop:ellcircum}}\label{sec:proofpropOptvalue}
\paragraph{Proof of Theorem \ref{prop:metz1756}.}
Note that each integral $\int_{ \partial\Omega} a\left(\frac{\partial \phi_j}{\partial \nu}\right)^2 \, d\Hn$ in the definition of $J_\infty$ is well defined and is finite under the above regularity assumptions on $\partial\Omega$\footnote{Since $\Omega$ is convex or has a $\mathcal{C}^{1,1}$ boundary, the outward unit normal $\nu$ is defined almost everywhere, the eigenfunctions $\phi_j$ belong to $H^2(\Omega)$ and their Neumann trace $\partial\phi_j/\partial\nu$ belongs to $L^2(\partial\Omega)$ for any $j\geq 1$. Hence $a\left(\frac{\partial \phi_j}{\partial \nu}\right)^2\in L^1(\partial\Omega)$ for every $a\in L^\infty(\partial \Omega)$.}.

According to the Rellich identity \eqref{Tao}, there holds
$$
\sup_{a\in L^\infty(\partial\Omega)}J_\infty(a)-\frac{2}{n|\Omega|}\int_{\partial\Omega}a\, d\Hn\geq J_\infty(a_{x_0})-\frac{2}{n|\Omega|}\int_{\partial\Omega}a_{x_0}\, d\Hn=0,
$$
by using that $\int_{\partial\Omega} \langle x-x_0,\nu(x)\rangle\, d\Hn(x)=\int_\Omega \operatorname{div}(x-x_0)\, dx=n|\Omega|$.

In order to prove the converse inequality, we consider Ces\`aro means of eigenfunctions. Indeed, introducing the family $(\mu_j)_{j\in\N^*}$ of measures 
$$
\mu_j(a)=\frac{1}{\lambda_j}\int_{\partial\Omega}a\left(\frac{\partial \phi_j}{\partial \nu}\right)^2 \, d\Hn
$$
we have
$$
\sup_{a\in L^\infty(\partial\Omega)}\inf_{j\in \N^*}\mu_j(a) =\sup_{a\in L^\infty(\partial\Omega)}\inf_{\substack{(\alpha_j)\in \ell^1(\R_+)\\ \sum_{j=1}^{+\infty}\alpha_j=1}} \sum_{j=1}^{+\infty}\frac{\alpha_j}{\lambda_j}\int_{\partial \Omega} a\left(\frac{\partial \phi_j}{\partial \nu}\right)^2 \, d\Hn
$$
and by considering particular choices of sequences $(\alpha_j)_{j\in\N^*}$, we get
$$
\sup_{a\in L^\infty(\partial\Omega)}\inf_{j\in \N^*}\mu_j(a) \leq \sup_{a\in \Ub}\inf_{N\geq 1} \frac{1}{N}\sum_{j=1}^{N}\mu_j(a).
$$
According to \cite[Theorem 7]{MR520983}, the sequence 
$$
\left(\frac{1}{N}\sum_{j=1}^N\frac{1}{\lambda_j} \left( \frac{\partial \phi_j}{\partial \nu} \right)^2\right)_{N\in\N^*}
$$
is uniformly bounded and converges uniformly to some positive constant $C_\Omega$ on every compact subset of $\partial\Omega$ for the $C^0$-topology, and in particular weakly in $L^1(\Omega)$. As a consequence, considering $a\in L^\infty(\partial\Omega)$,  we infer that
$$
\inf_{N\geq 1} \frac{1}{N}\sum_{j=1}^{N}\mu_j(a)\leq \lim_{N\to +\infty}\frac{1}{N}\sum_{j=1}^{N}\mu_j(a)=C_\Omega \int_{\partial\Omega}a\, d\Hn.
$$
To compute $C_\Omega$, let us use the Rellich identity \eqref{Tao}. For $x_0\in \R^n$, there holds
$$
2=\lim_{N\to +\infty}\frac{1}{N}\sum_{j=1}^{N}\mu_j(a_{x_0})=C_\Omega \int_{\partial\Omega}a_{x_0}\, d\Hn=C_\Omega n|\Omega|,
$$
and therefore $C_\Omega=\frac{2}{n|\Omega|}$. As a consequence, we infer that
$$
\sup_{a\in L^\infty(\partial\Omega)}\left(\frac{1}{N}\inf_{N\geq 1} \sum_{j=1}^{N}\mu_j(a)-\frac{2}{n|\Omega|}\int_{\partial\Omega} a\, d\Hn\right)\leq 0.
$$
Combining all the estimates, it follows that
$$
\sup_{a\in L^\infty(\partial\Omega)}\left(J_\infty(a)-\frac{2}{n|\Omega|}\int_{\partial\Omega}a\, d\Hn\right)=0
$$
and one easily checks that any function $ca_{x_0}$ with $c\in \R$ reaches the supremum, whence the inequality \eqref{estimJa}. 

According to \eqref{estimJa}, one has for a given $p_0\in \R$,
$$
\forall a\in L^\infty(\partial\Omega)\ \mid\ \int_{\partial\Omega}a\, d\Hn=p_0, \qquad
J_\infty(a)\leq \frac{2p_0}{n|\Omega|} .
$$
By noting that the right-hand-side is reached by the Rellich function $ca_{x_0}$ where $c\in \R$ is chosen in such a way that the integral constraint is satisfied, the expected result follows.

To conclude, it remains to show that $J_\infty(a)$ is finite for every $a\in L^\infty(\Omega)$. To this aim, let us argue by contradiction, considering $a\in L^\infty(\Omega)$ and assuming that $J_\infty(a)=-\infty$. Then, there exists an increasing sequence of integers $(j_{k})_{k\in \N}$ such that $\mu_{j_k}(a)\to -\infty$ as $k\to +\infty$. But according to the aforementioned convergence result, one has $\mu_{j_k}(a)\to C_\Omega \int_{\partial\Omega}a\, d\Hn$, which is impossible.

\paragraph{Proof of Theorem \ref{OptimalValue}.}

The first inequality is obvious, according to Theorem \ref{prop:metz1756}, since it can be recast as
$$
\sup_{a\in \Ub} J_\infty(a)\leq \max \left\{J_\infty (a), a\in L^\infty(\partial\Omega)\ \mid\  \int_{\partial\Omega}a\, d\Hn=LM\Hn(\partial\Omega)\right\}.
$$

Let us prove the second item. Still using Theorem \ref{prop:metz1756}, the equality is true if, and only if there exists a Rellich admissible function $\tilde a_{x_0}$ (defined by \eqref{defax0}), in other words a Rellich function belonging to $\Ub$.
Recall that $\int_{\partial\Omega} \langle x-x_0,\nu(x)\rangle\, d\Hn(x)=\int_\Omega \operatorname{div}(x-x_0)\, dx=n|\Omega|$, and therefore $\int_{\partial\Omega}\tilde a_{x_0}\, d\Hn=LM\Hn(\partial\Omega)$. To investigate the existence of a Rellich admissible function, we will then concentrate on the pointwise constraints. The admissibility condition on $\tilde a_{x_0}$ rewrites
$$
\exists x_0\in \R^n\mid \forall x\in \partial\Omega, \qquad -M\leq \frac{LM\Hn (\partial\Omega)}{n|\Omega|} \langle x-x_0,\nu(x)\rangle\leq M
$$
or similarly
$$
\exists x_0\in \R^n\mid \forall x\in \partial\Omega, \qquad |L| \left|\langle x-x_0,\nu(x)\rangle\right|\leq \frac{n|\Omega|}{\Hn (\partial\Omega)}
$$
leading to the condition 
$$
|L|\inf_{x_0\in \R^n}\sup_{x\in \partial\Omega}\left|\langle x-x_0,\nu(x)\rangle\right|\leq \frac{n|\Omega|}{\Hn (\partial\Omega)}
$$
We conclude by noting that one must moreover have $L\in (-1,1)$ by assumption.

\paragraph{Proof of Proposition \ref{prop:ellcircum}.}
This result is a direct consequence of Lemma \ref{prop:Linftynormmax}. Indeed, under the geometric condition on the circumsphere to $\Omega$, one easily shows that any diameter of $\Omega$ is also a diameter of the circumsphere and all inequalities in \eqref{dm11h43} are equalities, whence the claim.


\section{Solving of Problem \eqref{mainPbOpt}}\label{sec:mainPbOptm}
Let us now investigate Problem \eqref{mainPbOpt}. Note that 
$$
\sup_{\chi_\Gamma\in\mathcal{U}_{L,M} } \inf_{j\in\mathbb{N}^*}\frac{1}{\lambda_j}\int_{\partial \Omega} \chi_\Gamma\left(\frac{\partial \phi_j}{\partial \nu}\right)^2 \, d\Hn \leq \sup_{a\in\overline{\mathcal{U}}_L} \inf_{j\in\mathbb{N}^*}\frac{1}{\lambda_j}\int_{\partial \Omega} a\left(\frac{\partial \phi_j}{\partial \nu}\right)^2 \, d\Hn .
$$
Nevertheless, as an infimum of linear and continuous functions  for the usual $L^\infty$ weak-star topology, the mapping $a\mapsto \inf_{j\in\mathbb{N}^*}\frac{1}{\lambda_j}\int_{\partial \Omega} a\left(\frac{\partial \phi_j}{\partial \nu}\right)^2 \, d\Hn$ is upper semicontinuous and not lower semicontinuous for this topology. Because of this lack of continuity, it is not clear whether the converse sense holds true or not. 

\subsection{A no-gap result}
In the sequel, we will consider two kinds of geometric assumptions on the domain $\Omega$. Let us make them precise.

\begin{quote}
{\bf Uniform boundedness (UB) property.} There exists $A>0$ such that 
\begin{equation} \label{hypNorm} 
\forall j\in\N^*, \qquad \left\Vert \frac{\partial \phi_j}{\partial \nu} \right\Vert_{L^\infty(\partial\Omega)}^2 \leq A \lambda_j. \end{equation}  
\end{quote}

\begin{quote}
{\bf Quantum Uniform Ergodicity of Boundary values (QUEB) property.} The sequence of measures $ \frac{1}{\lambda_j} \left( \frac{\partial \phi_j}{\partial \nu} \right)^2d\Hn$ converges vaguely to the uniform measure $ \frac{2}{n |\Omega|}d\Hn$ as $j\to +\infty$.
\end{quote}

In the following result, one shows that under several geometric assumption, every maximizing sequence $(\chi_{\Gamma_k})_{k\in \N}$, seen as a Radon measure, has to converge vaguely to a Rellich-admissible function as $k$ diverges. Nevertheless, the converse sense is not true. This last claim has been discussed and commented in \cite{PTZobsND} for another spectral functional.

\begin{theorem}[No-Gap]\label{No-Gap}
Let $\Omega$ be a bounded connected domain of $\R^n$ either with a $\mathcal{C}^{1,1}$ boundary, or convex. 

Assume that $\Omega$ satisfies the (UB) and the (QUEB) properties.
Then the optimal values of Problems \eqref{defJrelax} and \eqref{mainPbOpt} are the same.
In particular, one has
$$
\forall L\in[-L^c_{n},L^c_{n}],\qquad
\sup_{\chi_\Gamma \in \mathcal{U}_{L,M} }J_\infty(M\chi_\Gamma-M\chi_{\partial\Omega\backslash \Gamma})=\frac{2L\Hn (\partial\Omega)}{n|\Omega|},
$$
where $L^c_{n}$ is defined in Theorem \ref{OptimalValue}.
\end{theorem}

The statement of Theorem \ref{No-Gap} can be reformulated as follows: there is no gap between the optimal values of the problems \eqref{mainPbOpt} and \eqref{defJrelax}. Moreover, an explicit maximizing sequence $(\chi_{\Gamma_k})_{k\in\N}$ is provided in the proof of this theorem, whatever the value of $L$, although the knowledge of the optimal value is only known in the case where $|L|\leq L^c_{n}$. 

Finally, the assumptions of Theorem \ref{No-Gap} are not empty. As it will be highlighted in the discussion on the disk below, these assumptions are satisfied in particular when $\Omega$ is the unit disk of $\R^2$.

\subsection{Comments on the assumptions and the results}\label{sec:comments}

The following remarks are in order.
\begin{itemize}
\item The two properties  (UB)) and (QUEB) depend on the choice of the Hilbert basis $(\phi_j)_{j\in\N}$. 
\item The property (UB) holds  true for whenever $\Omega$ is a $n$-dimensional orthotope $\Omega=(0,\pi)^n$, a two-dimensional disk (see Section \ref{Invest}), or the flat torus $\Omega=\mathbb{T}^n$.
\item Concerning the (QUEB) property, very little is known about it. It is nevertheless notable that the following close property is well known and referenced.
\begin{quote}
{\bf Weak Quantum Ergodicity of Boundary values (WQEB) property.} There exists a subsequence of $ \frac{1}{\lambda_j} \left( \frac{\partial \phi_j}{\partial \nu} \right)^2d\Hn$ converging vaguely to the uniform measure $ \frac{2}{n |\Omega|}d\Hn$.
\end{quote}
\medskip
It has been proved in particular in \cite{BurqErgo,HasselErgo} that the (WQEB) property holds true in the flat torus $\Omega=\mathbb{T}^n$, and in all piecewise smooth ergodic domains $\Omega$. More precisely, in these articles it is shown that for such a domain $\Omega$, and for every semi-classical operator $A$ of order $0$ on $\partial\Omega$, there exists a density one\footnote{The expression ``\textit{density one}'' means that there exists $\mathcal{I}\subset \N^*$ such that $\# \{j\in\mathcal{I}\ \vert\  j\leq N\}/N$ converges to $1$ as $N$ tends to $+\infty$.} subsequence $(j_k)_{k\in\N}$ such that 
$$
\lim_{k\rightarrow +\infty} \lambda_{j_k}^{-1} \left\langle A \left(\frac{\partial \phi_{j_k}}{\partial \nu} \right), \frac{\partial \phi_{j_k}}{\partial\nu}  \right\rangle = \frac{4}{|\mathbb{S}^{n-1}| \; |\Omega|} \int_{T^*(\partial\Omega)} a(y,\eta) \sqrt{1-|\eta|^2} \, d\eta\,  d\Hn , 
$$ 
where $\mathbb{S}^{n-1}$ is the unit sphere in $\R^n$ and $a$ is the symbol of $A$. It says that the boundary traces of eigenfunctions are, on the average, distributed in phase space $T^*(\partial\Omega)$, according to $(1-|\eta|^2)^{1/2}$ (\cite{BarnettHassel}) for the Dirichlet boundary conditions. Finally, one recovers (WQEB) by choosing $a(y,\eta)=a(y) \in L^\infty (\partial\Omega, [0,1])$.
\item Even if the (WQEB) property is satisfied in a large class of domains, very few of them may satisfy the more restrictive (QUEB) property. Up to our knowledge, the only domain known to satisfy the (QUEB) property is the Euclidean disk in $\R^n$.
An interesting issue would consist in determining a sufficient geometric condition guaranteeing this property.

Let us sum-up what is known about such properties for particular choices of domains $\Omega$.
If $\Omega$ is a rectangle with the usual Hilbert basis of eigenfunctions of $\triangle$ made of products of sine functions, the (WQEB) property is satisfied but not the (QUEB) property. If $\Omega$ is a disk of $\R^2$ with the usual Hilbert basis of eigenfunctions of $\triangle$ defined in terms of Bessel functions, the (QUEB) property holds true. These results are in particular recovered in Section \ref{Invest}. Finally, concerning the three-dimensional Euclidean unit ball, a weak consequence of the main quantum ergodicity results is the existence of a Hilbert basis of eigenfunctions such that the (WQEB) is satisfied.
\end{itemize}                                                                                                         

\subsection{Proof of Theorem \ref{No-Gap}}\label{sec:proofNG}

This proof is inspired by \cite[Theorem 6]{PTZobsND}, but important adaptations and changes have been necessary.
In the whole proof and for the sake of simplicity, we will assume that $M=1$ (the expected general result will be easily inferred by an immediate adaptation of this proof) and use the notation $\widehat M= \frac{2L\Hn(\partial\Omega)}{n|\Omega|}$. 

We will distinguish between the two cases $|L|\leq L^c_{n}$ and $|L|> L^c_{n}$.
\begin{center}
\textbf{First case: $|L|\leq L^c_{n}$}
\end{center}
This case is the hardest one. Assume without loss of generality that $x_0=0$. Therefore, according to Theorem \ref{OptimalValue}, the function $\partial\Omega\ni x\mapsto \frac{\widehat M}{2}\langle x,\nu(x)\rangle$ belongs to $\Ub$ and is a solution of the convexified problem \eqref{defJrelax}.

Introduce
\begin{eqnarray*}
I_j(\Gamma)& =& \int_{\partial\Omega} (\chi_\Gamma-\chi_{\partial\Omega\backslash \Gamma})\frac{1}{\lambda_j}\left(\frac{\partial\phi_j}{\partial\nu}\right)^2\, d\Hn\\ 
&=& \int_{\partial\Omega} (2\chi_\Gamma-\chi_{\partial\Omega})\frac{1}{\lambda_j}\left(\frac{\partial\phi_j}{\partial\nu}\right)^2\, d\Hn
\end{eqnarray*}
for every $j\in\N^*$, so that $J_\infty(\chi_\Gamma-\chi_{\partial\Omega\backslash \Gamma}) = \inf_{j\in\N^*} I_j(\Gamma)$.

According to the geometric assumptions on $\Omega$ and to Theorem \ref{OptimalValue}, we have
$$
\sup_{\chi_\Gamma\in\mathcal{U}_{L,M} } \inf_{j\in\mathbb{N}^*}I_j(\Gamma) \leq \sup_{a\in\overline{\mathcal{U}}_L} \inf_{j\in\mathbb{N}^*}\frac{1}{\lambda_j}\int_{\partial \Omega} a\left(\frac{\partial \phi_j}{\partial \nu}\right)^2 \, d\Hn 
= \widehat{M}.
$$
To prove that the latter inequality is in fact an equality, we will construct a sequence $\left(\chi_{\Gamma_k}\right) _{k\in\N}\in (\U)^\N$ in such a way that $(J_\infty(\chi_{\Gamma_k}-\chi_{\partial\Omega\backslash\Gamma_k}))_{k\in\N}$ converges to $\widehat M$. 
 
Let $\Gamma_0$ be an open, connected and Lipschitz subdomain of $\partial\Omega$ such that $\Hn(\Gamma_0)=\frac{L+1}{2}\Hn(\partial \Omega )$. Let us assume that $J_\infty(\Gamma_0)<\widehat M$ (either we are done). According to the (QUEB) property, there exists $j_0 \in \N^*$ such that
\begin{equation}\label{highfreq}
I_j(\Gamma_0) \geq \widehat M-\frac{1}{4}(\widehat M-J_\infty(\Gamma_0)),
\end{equation}
for every  $j> j_0$.

Since $\partial\Omega$ and $\Gamma_0$ are supposed to be Lipschitz, then $\Gamma_0$ and $\partial\Omega\backslash \Gamma_0$ satisfy a $\delta$-cone property\footnote{\label{footnoteEpscone}We recall that an open smooth surface $A$ in $\R^n$ verifies a $\delta$-cone property if, for every $x\in A$, there exists a normalized vector $\xi_x$ such that $C(y,\xi_x,\delta)\subset A$ for every $y\in \overline{A}\cap B(x,\delta)$, where $C(y,\xi_x,\delta)=\{z\in\R^n\ \vert\ \langle z-y,\xi\rangle \geq \cos\delta \Vert z-y\Vert\text{ and }0<\Vert z-y\Vert <\delta\}$, see, e.g., \cite{henrot-pierre}}.
Consider now two partitions 
\begin{equation}\label{labsubdivisions}
\overline{\Gamma}_0=\bigcup_{i=1}^K F_i\quad\textrm{and}\quad\Gamma_0^c=\bigcup_{i=1}^{\tilde{K}}
\widetilde F_i,
\end{equation}
respectively of $\overline{\Gamma_0}$ and $\Gamma_0^c$, such that each $F_i$ and $\widetilde{F_i}$ is a subset of a $\partial\Omega$-strata. From the $\delta$-cone property, there exist $c_\delta>0$ and a choice of family $(F_i)_{1\leq i \leq K}$ (resp. $(\widetilde F_i)_{1\leq i \leq \tilde{K}}$) such that, for $| F_i |$ small enough, one has
\begin{equation}\label{regularMesh}
\forall i\in \{1,\cdots,K\} \ \left(\textrm{resp. }\forall i\in \{1,\cdots,\tilde{K}\} \right), \ \frac{\eta_i}{\operatorname{diam}( F_i) }\geq c_\delta \ \left(\textrm{resp. }\frac{\widetilde \eta_i}{\operatorname{diam}(\widetilde F_i)}\geq c_\delta \right),
\end{equation}
where $\eta_i$ (resp., $\widetilde\eta_i$) is the inradius\footnote{In other words, the largest radius of balls contained in $F_i$.} of $F_i$
(resp., $\widetilde F_i$), and $\operatorname{diam}(F_i)$ (resp., $\operatorname{diam}(\widetilde F_i)$) the diameter of $F_i$ (resp., $\widetilde F_i$).
Finally for all $i\in\{1,\ldots,K\}$ (resp., for all $i\in\{1,\ldots,\tilde{K}\}$), there exist $\xi_i\in F_i$ (resp., $\tilde{\xi}_i\in\widetilde F_i$) such that $B(\xi_i,\eta_i/2)\subset F_i\subset B(\xi_i,\eta_i/c_\delta)$ (resp.,
$B(\tilde{\xi}_i,\widetilde\eta_i/2)\subset \widetilde F_i\subset B(\tilde{\xi}_i,\widetilde\eta_i/c_\delta)$), 

Now, choosing $\xi_i$ and $\tilde{\xi}_i$ as Lebesgue points of the functions $\left(\frac{\partial\phi_j}{\partial\nu}\right)^2$, for all $j\leq j_0$ yields 
\begin{eqnarray*}
\int_{F_i}\frac{1}{\lambda_j}\left( \frac{1}{2}- \frac{\widehat M}{4}\langle x,\nu \rangle\right) \left( \frac{\partial\phi_j}{\partial\nu}(x)\right) ^2\, d\Hn(x) =& \\ & \displaystyle \hspace{-4cm} \frac{1}{\lambda_j}\left( \frac{\Hn(F_i)}{2} - \int_{F_i}\frac{\widehat M}{4}\langle x,\nu\rangle d\Hn(x)\right)\left( \frac{\partial\phi_j}{\partial\nu}(\xi_i)\right)^2+\mathrm{o}(| F_i|)\quad\textrm{as}\
\eta_i\rightarrow 0, \\
\end{eqnarray*}
 for all $j\leq j_0$, $i\in\{1,\ldots,K\}$ and
\begin{eqnarray*}
\int_{\widetilde F_i} \left( \frac{1}{2}+ \frac{\widehat M}{4}\langle x,\nu \rangle\right)  \left(\frac{\partial\phi_j}{\partial\nu}(x)\right)^2\, d\Hn(x) = &\\ &\displaystyle \hspace{-4cm} \frac{1}{\lambda_j}\left(\frac{|{\widetilde{F_i}}|}{2}+\int_{\widetilde{F_i}} \frac{\widehat M}{2}\langle x,\nu\rangle\, d\Hn(x)\right)  \left( \frac{\partial\phi_j}{\partial\nu}(\widetilde{\xi_i})\right)^2+\mathrm{o}(|\widetilde{F_i}|)\quad\textrm{as}\ \widetilde \eta_i\rightarrow 0,\\
\end{eqnarray*}
for every $i\in\{1,\ldots,\tilde{K}\}$. Setting $\eta=\displaystyle \max\left(\max_{1\leq i \leq K}\operatorname{diam}(F_i),\max_{1\leq i \leq \tilde{K}}\operatorname{diam}(\widetilde F_i) \right)$ and using that $ \sum_{i=1}^K \Hn(F_i) =\frac{L+1}{2} \Hn(\partial\Omega) $ and $ \sum_{i=1}^{\tilde{K}} \Hn(\widetilde F_i) =\frac{1-L}{2} \Hn(\partial\Omega)$, there holds $ \sum_{i=1}^K \operatorname{o}(\Hn(F_i)) = \sum_{i=1}^{\tilde{K}} \operatorname{o}(\Hn(\widetilde F_i)) =\operatorname{o}(1)$ as $\eta\to 0$. Then,
\begin{multline}\label{eggrandO1}
 \int_{\Gamma_0} \frac{1}{\lambda_j}\left( \frac{1}{2}- \frac{\widehat M}{4}\langle x,\nu\rangle \right)\left( \frac{\partial \phi_j}{\partial \nu} (x) \right)^2 d\Hn(x) \\
 = \sum_{i=1}^K \frac{1}{\lambda_j}\left( \frac{\Hn(F_i)}{2} - \int_{F_i} \frac{\widehat M\langle x,\nu\rangle }{4} d\Hn \right) \left( \frac{\partial \phi_j}{\partial \nu} (\xi_i) \right)^2 + \operatorname{o}(1) 
\end{multline}
and
\begin{multline}\label{eggrandO2}
\int_{\Gamma_0^c}  \frac{1}{\lambda_j}\left( \frac{1}{2}+ \frac{\widehat M}{4}\langle x,\nu\rangle \right) \left( \frac{\partial \phi_j}{\partial \nu}(x) \right)^2 d\Hn \\
=  \sum_{i=1}^{\tilde{K}}  \left( \int_{\widetilde{F}_i}  \frac{1}{\lambda_j}\left( \frac{1}{2}+ \frac{\widehat M}{4}\langle x,\nu\rangle \right) \, d\Hn(x) \right)\left( \frac{\partial \phi_j}{\partial \nu} (\widetilde{\xi_i}) \right)^2 + \operatorname{o}(1)
\end{multline}
for every $j\leq j_0$.

We set
\begin{eqnarray*}
h_i &=& \frac{\Hn(F_i)}{2} - \int_{F_i} \frac{\widehat M }{4}\langle x,\nu\rangle\, d\Hn(x), \quad i\in \{1,\cdots,K\} \\
\textrm{and} \qquad \ell_i &=&  \frac{\Hn(\widetilde{F}_i)}{2}+ \int_{\widetilde{F_i}} \frac{\widehat M }{4}\langle x,\nu\rangle\, d\Hn(x), \quad i\in \{1,\cdots,\tilde{K}\}
\end{eqnarray*}
Then, for $\varepsilon>0$ small enough, we define the perturbation $\Gamma^\varepsilon$ of $\Gamma_0$ by
$$
\Gamma^\varepsilon =\left(\Gamma_0\backslash \bigcup_{i=1}^K \overline{B(\xi_i,\varepsilon_i)}\right)\quad \bigcup\quad \bigcup_{i=1}^{\tilde{K}} B(\tilde{\xi}_i,\widetilde \varepsilon_i),
$$
where $\varepsilon_i$ and $\widetilde \varepsilon_i$ are chosen so that $\Hn(B(\xi_i,\varepsilon_i))=\varepsilon^{n-1}h_i $ and $\Hn(B(\widetilde{\xi_i},\widetilde{\varepsilon_i}))=\varepsilon^{n-1}\ell_i $.
This is possible provided that
$$
0<\varepsilon< \min\left(\min_{1\leq i \leq K}\frac{\eta_i \Hn(B(\xi_i,1))^{1/(n-1)}}{h_i^{1/(n-1)}},\min_{1\leq i \leq \tilde{K}}\frac{\widetilde \eta_i \Hn(B(\tilde\xi_i,1))^{1/(n-1)}}{\ell_i^{1/(n-1)}}\right).
$$
By the isodiametric inequality\footnote{The isodiametric inequality states that, for every compact $K$ of the Euclidean space $\R^n$, there holds $\vert K\vert\leq \vert B(0,\operatorname{diam}(K)/2)\vert$. 
The same result holds, up to a multiplicative constant, for a compact stratified manifold endowed with the measure $\Hn$ and the geodesic distance on each strata.} 
and a compactness argument, there exists a constant $V_n>0$ (depending only on $\Omega$) such that
$\Hn(F_i) \leq V_n({\operatorname{diam}(F_i)})^{(n-1)}$ for every $i\in \{1,\cdots,K\}$, and $\Hn(\widetilde{F}_i) \leq V_n({\operatorname{diam}(\tilde F_i)})^{(n-1)}$ for every $i\in \{1,\cdots,\tilde K\}$, independent of the considered partitions. 
Because of the compactness of $\partial\Omega$, there also exists $v_n>0$ (depending only on $\Omega$) such that $\Hn(B(x,1)) \geq v_n$ for all $x\in\partial\Omega$.
Set now $\varepsilon_0=\min(1,c_\delta v_n/V_n^{1/(n-1)})$. 
From \eqref{regularMesh}, one has 
$$
\frac{\eta_i \Hn(B(\xi_i,1))^{1/(n-1)}}{h_i^{1/(n-1)}}\geq \frac{v_n}{(1-L)^{1/(n-1)}V_n^{1/(n-1)}}\frac{\eta_i}{\operatorname{diam}(F_i)}\geq
\varepsilon_0 ,
$$
for every $i\in \{1,\cdots,K\}$, and
$$
\frac{\widetilde\eta_i \Hn(B(\tilde\xi_i,1))^{1/(n-1)}}{\ell_i^{1/(n-1)}}\geq \varepsilon_0 ,
$$
for every $i\in \{1,\cdots,\tilde{K}\}$. 
This perturbation is well defined for $\varepsilon\in (0,\varepsilon_0)$.
Moreover,
\begin{eqnarray*}
\Hn(\Gamma^\varepsilon)  & = &  \Hn(\Gamma_0) -
\sum_{i=1}^K  | B(\xi_i,\varepsilon_i)| +\sum_{i=1}^{\tilde{K}} | B(\tilde\xi_i,\widetilde \varepsilon_i)| \\
 & = & \Hn(\Gamma_0) -\varepsilon^{n-1} \sum_{i=1}^Kh_i+\varepsilon^{n-1}\sum_{i=1}^{\tilde{K}}\ell_i \\
&=&\Hn(\Gamma_0) -\varepsilon^{n-1} \sum_{i=1}^K  \left( \frac{\Hn(F_i)}{2} - \int_{F_i} \frac{\widehat M}{4} \langle x,\nu\rangle\, d\Hn \right) \\ 
&&+\varepsilon^{n-1} \sum_{i=1}^{\tilde{K}}    \left(\frac{\Hn(\widetilde{F}_i)}{2} + \int_{\widetilde{F_i}} \frac{\widehat M}{4} \langle x,\nu\rangle\, d\Hn \right)   \\
&=&  \Hn(\Gamma_0) = \frac{L+1}{2}\Hn(\partial\Omega).
\end{eqnarray*}

\begin{paragraph}{Low frequencies estimates.}
Let us write
\begin{eqnarray*}
I_j(\Gamma^\varepsilon)  & =&  2\int _{\Gamma^\varepsilon}\frac{1}{\lambda_j}\left( \frac{\partial\phi_j}{\partial\nu}\right) ^2\, d\Hn-\int _{\partial\Omega}\frac{1}{\lambda_j}\left( \frac{\partial\phi_j}{\partial\nu}\right) ^2\, d\Hn \\
&= & I_j(\Gamma_0)-2\sum_{i=1}^K\int_{B(\xi_i,\varepsilon_i)}\frac{1}{\lambda_j}\left(\frac{\partial\phi_j}{\partial\nu}\right)^2\, d\Hn+2\sum_{i=1}^{\tilde{K}}\int_{B(\tilde{\xi}_i,\widetilde \varepsilon_i)}\frac{1}{\lambda_j}\left(\frac{\partial\phi_j}{\partial\nu}\right)^2\, d\Hn 
\end{eqnarray*}
and using that $\xi_i$ and $\tilde{\xi}_i$ are Lebesgue points of the function $\left(\frac{\partial\phi_j}{\partial\nu}\right)^2$, there holds
\begin{eqnarray} 
 I_j(\Gamma^\varepsilon) & = & \displaystyle  I_j(\Gamma_0)-
2\sum_{i=1}^K \frac{|B(\xi_i,\varepsilon_i)|}{2\lambda_j}\left( \frac{\partial\phi_j}{\partial\nu}(\xi_i)\right)^2 +\operatorname{o}(|B(\xi_i,\varepsilon_i)|)\nonumber \\
& & +2\sum_{i=1}^{\tilde{K}}  \frac{|B(\tilde\xi_i,\widetilde\varepsilon_i)|}{2\lambda_j}\left(\frac{\partial\phi_j}{\partial\nu}(\tilde{\xi}_i)\right)^2 +\operatorname{o}(|B(\tilde\xi_i,\widetilde\varepsilon_i)|) \nonumber\\
& = & \displaystyle  I_j(\Gamma_0) - 2\varepsilon^{n-1} \left( \sum_{i=1}^K \frac{h_i}{2\lambda_j}\left(\frac{\partial\phi_j}{\partial\nu}(\xi_i)\right)^2
- \sum_{i=1}^{\tilde{K}} \frac{\ell_i}{2\lambda_j}\left( \frac{\partial\phi_j}{\partial\nu}(\tilde{\xi}_i)\right)^2 \right) +\varepsilon^{n-1}\operatorname{o}(1), \label{Calc}
\end{eqnarray}
by using that $\sum_{i=1}^K\mathrm{o}(|B(\xi_i,\varepsilon_i)|)+\sum_{i=1}^{\tilde{K}}\mathrm{o}(|B(\widetilde\xi_i,\widetilde\varepsilon_i)|)=\varepsilon^{n-1}\mathrm{o}(1)$ as $\varepsilon\to 0$, and thus as $\eta\to 0$.

Thus, according to \eqref{eggrandO1} and \eqref{eggrandO2}, and noting that
\begin{multline*}
\int_{\Gamma_0^c}  \frac{1}{\lambda_j}\left( \frac{1}{2}+ \frac{\widehat M}{4}\langle x,\nu\rangle \right) \left( \frac{\partial \phi_j}{\partial \nu}(x) \right)^2 d\Hn-  \int_{\Gamma_0} \frac{1}{\lambda_j}\left( \frac{1}{2}- \frac{\widehat M}{4}\langle x,\nu\rangle \right)\left( \frac{\partial \phi_j}{\partial \nu} (x) \right)^2 d\Hn(x)\\
=\frac{1}{2}\left(\widehat M-\left(2 \int_{\Gamma_0} \frac{1}{\lambda_j}\left( \frac{\partial \phi_j}{\partial \nu} (x) \right)^2 d\Hn(x) - \int_{\partial\Omega} \frac{1}{\lambda_j}\left( \frac{\partial \phi_j}{\partial \nu} (x) \right)^2 d\Hn(x)\right)\right),
\end{multline*}

one has
\begin{eqnarray*}
I_j(\Gamma^\varepsilon) &=&  I_j(\Gamma_0)+\varepsilon^{n-1} \left(  \widehat M-I_j(\Gamma_0)\right)+\varepsilon^{n-1} \operatorname{o}(1)\qquad\textrm{as}\ \eta\rightarrow 0,
\end{eqnarray*}
for all $j\leq j_0$ and all $\varepsilon\in(0,\varepsilon_0)$.
Then, since $\varepsilon_0^{n-1}\leq 1$, one has
\begin{equation}\label{lowfreq}
I_j(\Gamma^\varepsilon) \geq J_\infty(\Gamma_0)+\varepsilon^{n-1}
(\widehat M-J_\infty(\Gamma_0))+\varepsilon^{n-1} \operatorname{o}(1)\qquad\textrm{as}\
\eta\rightarrow 0 ,
\end{equation}
for every $j\leq j_0$ and $\varepsilon\in(0,\varepsilon_0)$.

Choosing the subdivisions $(F_i)_{1\leq i \leq K}$ and $(\widetilde F_i)_{1\leq i \leq \tilde{K}}$ fine enough, in other words $\eta>0$ small enough, allows to write that
\begin{equation}\label{in29eps_1}
I_j(\Gamma^\varepsilon) \geq J_\infty(\Gamma_0)+\frac{\varepsilon^{n-1}}{2} 
(\widehat M-J_\infty(\Gamma_0)),
\end{equation}
for every $j\leq j_0$ and every $\varepsilon\in(0,\varepsilon_0)$.
\end{paragraph}

\begin{paragraph}{High frequencies estimates.}
According to \eqref{hypNorm}, the sequence $\left(\frac{1}{\lambda_j}\left(\frac{\partial\phi_j}{\partial\nu}\right)^2\right)_{j\in\N^*}$ is bounded by a constant $A>0$ in $L^\infty(\partial\Omega)$. As a consequence, one has
$$
\vert I_j(\Gamma^\varepsilon)-I_j(\Gamma_0)\vert =  2\left\vert\int_{\partial\Omega} \left(\chi_{\Gamma^\varepsilon}-\chi_{\Gamma_0}\right)\frac{1}{\lambda_j}\left(\frac{\partial\phi_j}{\partial\nu}\right)^2\,d\Hn\right\vert 
\leq 2A^2 \left(\int_{\partial\Omega}\vert\chi_{\Gamma^{\varepsilon}}-\chi_{\Gamma_{0}}\vert  \, d\Hn\right) 
$$
for all $j\in\N$ and every $\varepsilon\in(0,\varepsilon_0)$.
Moreover,
\begin{equation*}
\int_{\partial\Omega}\vert\chi_{\Gamma^{\varepsilon}}-\chi_{\Gamma_{0}}\vert \, d\Hn
= \varepsilon^{n-1} \left(\sum_{i=1}^Kh_i+\sum_{i=1}^{\tilde{K}}\ell_i\right)
= \varepsilon^{n-1} \Hn(\partial\Omega),
\end{equation*}
and thus
$
\vert I_j(\Gamma^\varepsilon)-I_j(\Gamma_0)\vert \leq 2 A^{2} \varepsilon^{n-1}  \Hn(\partial\Omega) .
$
Finally setting,
$$ \varepsilon_1 = \min\left(\varepsilon_0, \left(
\frac{(\widehat M-J_\infty(\Gamma_0))}{2^3 A^{2}\Hn(\partial\Omega)}\right)^\frac{1}{n-1}\right),$$
one has, using \eqref{highfreq} that
\begin{equation}\label{highfreqeps}
I_j(\Gamma^\varepsilon) \geq \widehat M - \frac{1}{2}(\widehat M-J_\infty(\Gamma_0))
\end{equation}
for every $j\geq j_0$ and every $\varepsilon\in(0,\varepsilon_1)$.
\end{paragraph}

\begin{paragraph}{Conclusion.}
We now use the fact that $J_\infty(\Gamma_0)+\frac{\varepsilon^{n-1}}{2} (\widehat M-J_\infty(\Gamma_0)) \leq \widehat M-\frac{1}{2}(\widehat M-J_\infty(\Gamma_0))$ for all $\varepsilon\in(0,\varepsilon_0)$ (and thus for $\varepsilon \in (0,\varepsilon_1)$). Combining \eqref{in29eps_1} and \eqref{highfreqeps} , it follows that
\begin{equation}
J_\infty(\Gamma^\varepsilon)\geq J_\infty(\Gamma_0)+\frac{\varepsilon^{n-1}}{2} (\widehat M-J_\infty(\Gamma_0)),
\end{equation}
for every $\varepsilon\in(0,\varepsilon_1)$. 
In particular this inequality holds for $\varepsilon$ such that $\varepsilon^{n-1} = C_1 \min(C_2,\widehat M-J_\infty(\Gamma_0) )$, with $C_1= 1/8A^2\Hn(\partial\Omega)$ and $C_2 = (1/C_1)\min(1,(c_\delta v_n)^{n-1}/V_n) $ which are positive constants.
For this particular value of $\varepsilon$, we set $\Gamma_1=\Gamma^\varepsilon$, which ensure to have
\begin{equation}\label{debutiter}
J_\infty(\Gamma_1)\geq J_\infty(\Gamma_0)+\frac{C_1}{2} \min(C_2, \widehat M-J_\infty(\Gamma_0))\,(\widehat M-J_\infty(\Gamma_0)).
\end{equation}
Notice that the constants only depend on $A$ and $\partial\Omega$, and by construction $\Gamma_1$ satisfies a $\delta$-cone property.

Now, if $J_\infty(\Gamma_1)\geq \widehat M$ then we are done. 
Otherwise we apply the same procedure for $\Gamma_1$. According to the (QUEB) property, there exist a new $j_0$ such that \eqref{highfreq} holds with $\Gamma_1$ instead of $\Gamma_0$. 
This gives a lower bound for the higher modes. The low modes  $j\leq j_0$ are bounded below as before leading to an estimate similar to \eqref{in29eps_1} for $\Gamma_1$. 
Therefore, one gets the existence of $\Gamma_2$ such that \eqref{debutiter} holds with $\Gamma_1$ replaced by $\Gamma_2$ and $\Gamma_0$ replaced by $\Gamma_1$.

This way, one builds a sequence $\left( \Gamma_k \right)_{k\in\N}$ such that $\Hn(\Gamma_k)=\frac{L+1}{2}\Hn(\partial\Omega)$, as long as $J_\infty(\Gamma_k)<\widehat M$, and satisfying
$$
J_\infty(\Gamma_{k+1})\geq J_\infty(\Gamma_k)+ \frac{C_1}{2} \min(C_2, \widehat M-J_\infty(\Gamma_k))\, (\widehat M-J_\infty(\Gamma_k)).
$$
If $J_\infty(\Gamma_k)<\widehat M$ for all $k\in\N$, then the sequence $(J_\infty(\Gamma_k))_{k\in\N}$ is increasing, bounded above by $\widehat M$, and converges to $\widehat M$, which concludes the proof in the case where $L\leq L^c_{n}$.
\end{paragraph}
\begin{center}
\textbf{Second case: $|L|> L^c_{n}$}
\end{center}
The proof is very similar to the one in the case $L\leq L^c_{n}$, and even easier since we do not need to use sharp estimates for high-frequencies modes. For this reason, we only provide the main steps of the proof, explaining the (small) changes that must be done to adapt the proof of the first case.

As before, we know by Theorem \ref{OptimalValue} that there exists $a^*\in \Ub$ solving Problem \eqref{defJrelax} and moreover, $J_\infty(a^*)<\widehat M$ (since $L^c_{n}$ is the critical value such that $\widehat M$ is the optimal value for this problem).

Let $\Gamma_0$ be an open, connected and Lipschitz subdomain of $\partial\Omega$ such that $\Hn(\Gamma_0)=\frac{L+1}{2}\Hn(\partial\Omega)$. Let us assume that $J_\infty(\Gamma_0)<J_\infty(a^*)$ (otherwise we are done). Thanks to (QUEB) and (UB), there exists $j_0 \in \N^*$ such that
\begin{equation}\label{highfreq}
I_j(\Gamma_0) \geq J_\infty(a^*)-\frac{1}{4}(J_\infty(a^*)-J_\infty(\Gamma_0)),
\end{equation}
for every  $j> j_0$. Replacing the quantity $\frac{\widehat M}{2}\langle x,\nu\rangle$ by $a^*$ in the estimates, we reproduce the proof and use the same notations as before. Roughly speaking, it suffices to replace everywhere the number $\widehat M$ by $J_\infty(a^*)$ and the function $x\mapsto \frac{\widehat M}{2}\langle x,\nu\rangle$ by $a^*$. In particular, this allows to define $\varepsilon_0$ and $\varepsilon_1$ as in the first part of the proof.

This way, we build from $\Gamma_0$ a new set $\Gamma^\varepsilon$ having a Lipschitz boundary, such that $\Hn(\Gamma^\varepsilon)=\Hn(\Gamma_0)=L\Hn(\partial\Omega)$ and 
\begin{itemize}
\item ({\bf Low frequencies estimates})
$$
\forall j\leq j_0, \ \forall \varepsilon\in(0,\varepsilon_0),\qquad I_j(\Gamma^\varepsilon) \geq J_\infty(\Gamma_0)+\frac{\varepsilon^{n-1}}{2} (J_\infty(a^*)-J_\infty(\Gamma_0)),
$$
\item ({\bf High frequencies estimates})
$$
\forall j\geq j_0, \ \forall \varepsilon\in(0,\varepsilon_1),\qquad I_j(\Gamma^\varepsilon) \geq J_\infty(a^*) - \frac{1}{2}(J_\infty(a^*)-J_\infty(\Gamma_0)).
$$
\end{itemize}
Combining these estimates, it follows that
$$
\forall \varepsilon\in(0,\varepsilon_1), \qquad J_\infty(\Gamma^\varepsilon)\geq J_\infty(\Gamma_0)+\frac{\varepsilon^{n-1}}{2} (J_\infty(a^*)-J_\infty(\Gamma_0)),
$$
The end of the proof is then exactly similar to the previous case.

\subsection{Solving Problems \eqref{defJrelax} and \eqref{mainPbOpt} in 2D}\label{Invest}
This section is devoted to stating no-gap type results in particular situations that are not covered by  Theorem \ref{No-Gap} and to  to move further on the analysis of Problem \eqref{mainPbOpt} in such cases. 
More precisely, we investigate the two-dimensional cases where $\Omega$ is either a rectangle, a disk or an angular sector. 

\begin{paragraph}{Case of a rectangle.}
Let $\alpha$, $\beta$ be two positive numbers. We investigate here the case where $\Omega = (-\alpha\pi/2,\alpha\pi/2)\times (-\beta\pi/2,\beta\pi/2)$ and we consider the normalized eigenfunctions of the Dirichlet-Laplacian defined by
\begin{equation}\label{EigenRect}
\phi_{n,k}(x,y)=\frac{2}{\pi \sqrt{\alpha\beta}}\sin \left(\frac{n}{\alpha}\left(x+\frac{\pi\alpha}{2}\right)\right)\sin\left(\frac{k}{\beta}\left(y+\frac{\pi\beta}{2}\right)\right),
\end{equation}
associated to the eigenvalue
$$
\lambda_{n,k}=\frac{n^2}{\alpha^2} + \frac{k^2}{\beta^2},
$$
for all $(n,k)\in(\N^*)^2$.
The notations we will use are summarized on Figure \ref{RectFig}.
\begin{figure}[h]
\begin{center}
\begin{tikzpicture}[scale=0.75]
\draw [black, very thin, fill=blue!10] (-4,2) --++ (8,0)--++(0,-4)--++(-8,0)-- cycle;
\draw[latex-latex,blue] (-4,1) --++ (8,0); 
\draw[blue] (1.5,1) node[below ] {$\alpha\pi$};
\draw[latex-latex,blue] (-3,2) --++ (0,-4);
\draw[blue] (-3,-0.5) node[right] {$\beta\pi$};
\draw[->] (0,-2.5)--++(0,5);
\draw (0,2.5) node[above] {$y$};
\draw[->] (-4.5,0)--++(+9,0);
\draw[->] (4.5,0) node[right] {$x$};
\draw (0,0) node [below left] {$O$};
\draw[red] (0.5,2) node [above right] {$\Sigma_2$};
\draw[red] (0.5,-2) node [below right] {$\Sigma_4$};
\draw[red] (-4,-0.5) node [left] {$\Sigma_3$};
\draw[red] (4,-0.5) node [right] {$\Sigma_1$};
\end{tikzpicture}
\end{center}
\caption{Case of a rectangle}\label{RectFig}
\end{figure}
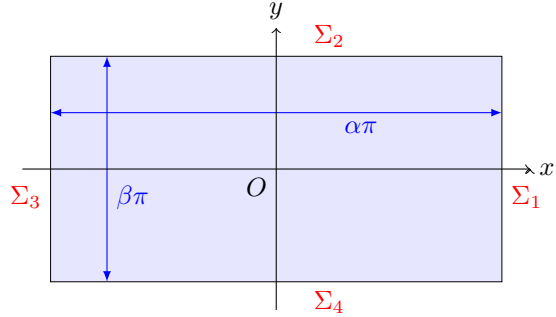

A straightforward computation yields 
\begin{multline*}
J_\infty(a)= \frac{4\alpha\beta }{\pi^2(n^2\beta^2+k^2\alpha^2)}\; \inf_{(n,k)\in\N^{*2}}  \displaystyle \left(   \frac{n^2}{\alpha^2}\int_{\Sigma_1\cup\Sigma_3}a(x,y)\sin^2\left(\frac{k}{\beta}\left(y+\frac{\pi\beta}{2}\right)\right)\, dy \right. \\
\left. +\frac{k^2}{\beta^2}\int_{\Sigma_2\cup\Sigma_4}a(x,y)\sin^2\left(\frac{n}{\alpha}\left(x+\frac{\pi\alpha}{2}\right)\right)dx \right).
\end{multline*}
Let us simplify the expression of $J_\infty(a)$. For $n\in \N^*$ and $a\in\Ub$, we set
\begin{eqnarray*}
A_{k,\beta}(a)&=&\int_{\Sigma_1\cup\Sigma_3} \hspace{-0.35cm}a(x,y)\sin^2\left(\frac{k}{\beta}\left(y+\frac{\pi\beta}{2}\right)\right)\,  dy\\
B_{n,\alpha}(a)&=& \int_{\Sigma_2\cup\Sigma_4} \hspace{-0.35cm} a(x,y)\sin^2\left(\frac{n}{\alpha}\left(x+\frac{\pi\alpha}{2}\right)\right)\, dx.
\end{eqnarray*} 
Using that $\frac{n^2}{\alpha^2}A_{k,\beta}(a) +\frac{k^2}{\beta^2}B_{n,\alpha}(a)  \geq  (\frac{n^2}{\alpha^2}+\frac{k^2}{\beta^2}) \min\{A_{k,\beta}(a),B_{n,\alpha}(a)\}$, we have
$$
J_\infty(a)\geq \frac{4}{\pi^2\alpha\beta} \inf_{(n,k)\in\N^{*2}} \left\{A_{k,\beta}(a),B_{n,\alpha}(a) \right\}.$$
The converse inequality is established by letting separately $n$ and $k$ go to $+\infty$ in the expression of $J_\infty(a)$  and using Riemann-Lebesgue lemma. Therefore, we obtain
\begin{equation}\label{CritSquare}
J_\infty(a) = \frac{4}{\pi^2\alpha\beta}\min \left( \inf_{k\in\N^*}A_{k,\beta}(a),\inf_{n\in\N^{*}}B_{n,\alpha}(a) \right) .
\end{equation}

\begin{proposition}\label{prop:rect}
Let $L^c_{n} = \min \left( \frac{2\alpha}{\alpha+\beta},\frac{2\beta}{\alpha+\beta}\right)$. There is no-gap between the optimal values for Problem \eqref{mainPbOpt} and its convexified version \eqref{defJrelax}, and 
$$
\max_{a\in\Ub} J_\infty(a)=\sup_{\chi_\Gamma \in\U} J_\infty(M\chi_\Gamma-M\chi_{\partial\Omega\backslash \Gamma})= \left\{\begin{array}{ll}
\frac{2L(\alpha+\beta)}{\pi\alpha\beta} & \textrm{if }|L|\leq L^c_{n},\\
\frac{2L^c_{n}(\alpha+\beta)}{\pi\alpha\beta}\operatorname{sgn}(L) & \textrm{if }|L|>L^c_{n}.
\end{array}\right.
$$
Finally, there exists a finite set $\mathcal{L}\subset [-1,1]$ such that the optimal design problem \eqref{mainPbOpt} has a solution if, and only if $L\in \mathcal{L}$. 
\end{proposition}

The proof of this proposition is done in Section \ref{sec:proofCaspart1}. The precise determination of $\mathcal{L}$ could easily be done since it is possible to derive from the proof a construction of all solutions. Such a construction, although a bit technical leads to
$$
\mathcal{L}=\left\{0,\frac{\pm\alpha}{\alpha+\beta}, \frac{\pm 2\alpha}{\alpha+\beta},\pm 1 \right\} \cap \left\{ 0,\frac{\pm\beta}{\alpha+\beta}, \frac{\pm2\beta}{\alpha+\beta} ,\pm 1\right\}.
$$
Two examples of solutions are pictured on Figure \ref{fig:squarePartsol} in the case where $\alpha=\beta=1$ and $L=\frac{1}{2}\}$. 

Finally, it is interesting to note that the conclusion of Lemma \ref{prop:Linftynormmax} does not hold true for such a choice of domain $\Omega$. This emphasizes the influence of the regularity of $\partial\Omega$ on the positive number $L_n^c$.
\begin{figure}[h]
\begin{center}
\includegraphics[width=4cm]{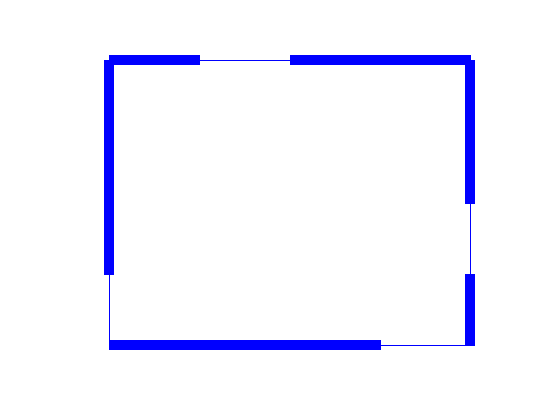}\quad 
\includegraphics[width=4cm]{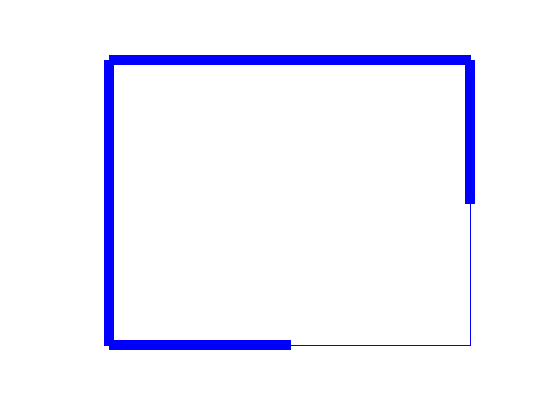}
\caption{Two particular solutions for $L=1/2$. \label{fig:squarePartsol}}
\end{center}
\end{figure}
\end{paragraph}


\begin{paragraph}{Case of the unit disk.}
We investigate here the case where $\Omega$ is the unit disk of $\R^2$ and we consider the normalized eigenfunctions of the Dirichlet-Laplacian given by the  triply indexed sequence
\begin{equation}\label{def_Phi_jkm}
\phi_{jkm}(r,\theta) = 
\left\{ \begin{array}{ll}
R_{0k}(r)/\sqrt{2\pi} & \ \textrm{if}\ j=0,\\
R_{jk}(r)Y_{jm}(\theta) & \ \textrm{if}\ j\geq 1,
\end{array} \right.
\end{equation}
for $j\in\N$, $k\in\N^*$ and $m=1,2$, where $(r,\theta)$ are the usual polar coordinates.
The functions $Y_{jm}(\theta)$ are defined by $Y_{j1}(\theta)=\frac{1}{\sqrt{\pi}}\cos(j\theta)$ and $Y_{j2}(\theta)=\frac{1}{\sqrt{\pi}}\sin(j\theta)$, and $R_{jk}$ by
\begin{equation}\label{def_R_jk}
R_{jk}(r) = \sqrt{2}\,\frac{J_j(z_{jk}r)}{\vert J'_{j}(z_{jk}) \vert},
\end{equation}
where $J_j$ is the Bessel function of the first kind of order $j$, and $z_{jk}>0$ is the $k^\textrm{th}$-zero of $J_{j}$.
The eigenvalues of the Dirichlet-Laplacian are given by the double sequence of $-z_{jk}^2$ and are of multiplicity $1$ if $j=0$, and $2$ if $j\geq 1$.

Easy computations show that for every $ a \in \Ub$, the criterion $J_\infty(a)$ rewrites, up to a multiplicative constant, 
\begin{equation}\label{CriterionDisk}
J_\infty(a)= \min \left( \inf_{n\geq 1} \int_0^{2\pi}a(\theta)\cos(n\theta)^2\, d\theta , \ \inf_{n\geq 1} \int_0^{2\pi}a(\theta)\sin(n\theta)^2\, d\theta  \right) .
\end{equation}

It is notable that all the assumptions of Theorem \ref{OptimalValue} and Theorem \ref{No-Gap} are fulfilled. In the following proposition, the optimal design problem \eqref{mainPbOpt} is completely solved in this particular case.

\begin{proposition}\label{mainDisk}
There is no-gap between the optimal values for Problem \eqref{mainPbOpt} and its convexified version \eqref{defJrelax}, and
$$ \max_{a\in\Ub}J_\infty(a)=J_\infty(L)=\pi L = \sup_{\chi_\Gamma \in\U} J_\infty(M\chi_\Gamma)$$
and $J_\infty(L)=\pi  L$.
Moreover the optimal design problem \eqref{mainPbOpt} has a solution if and only if $L\in\{ 0,\pm\frac{1}{2},\pm1\}$.
\end{proposition}

This proposition is proved in Section \ref{proof:propmainDisk}.
Several particular solutions in this case are pictured on Figure \ref{Fig:diskSolPart}.

\begin{figure}[h]
\begin{center}
\includegraphics[height=4cm]{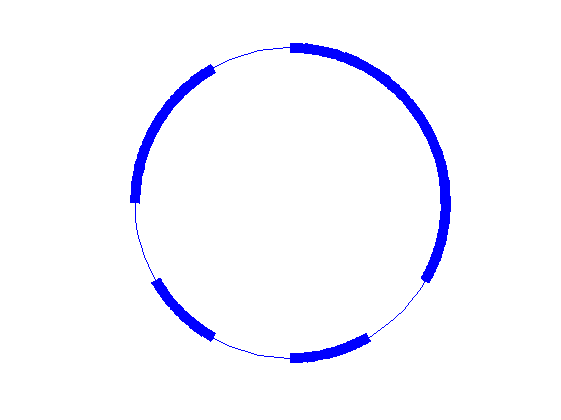}\quad
\includegraphics[height=4cm]{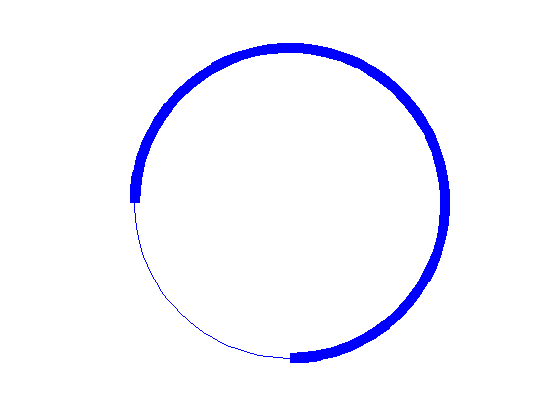}
\caption{Particular solutions for $L=1/2$. \label{Fig:diskSolPart}}
\end{center}
\end{figure}

\begin{remark}
It is also interesting to raise the same question when $\Omega$ is an ellipse. Numerical simulations suggest in particular that the optimal value behaves differently before and after the critical value $L^c_{n}$.
On Figure \ref{fig1}, we represent the graph of the optimal value for Problem \eqref{defJrelax} with respect to the constraint parameter $L$ for two ellipses. The optimal values are computed by using a large number of  random elements in $\Ub$.

\begin{figure}[h]
\begin{center}
\includegraphics[scale=0.3]{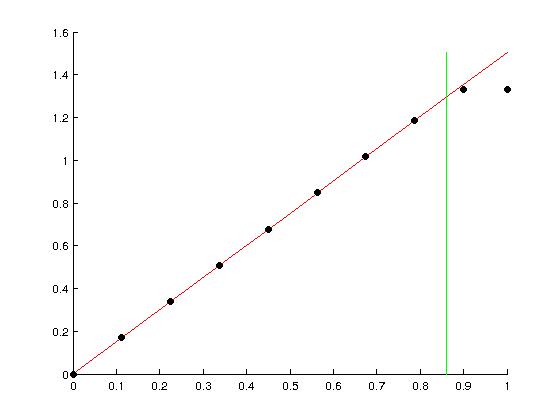} \quad \quad
\includegraphics[scale=0.3]{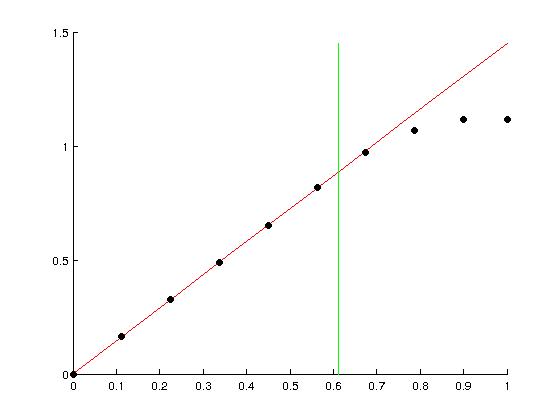}
\end{center}
\caption{Optimal value of $J_\infty$ with respect to $L$ for two ellipses respectively of radius and eccentricity $(1,1)$ and $(0.5,2)$ (dots), predicted behavior given by the Rellich functions (red) and plot of the straight line $L=L^c_{n}$, where $L^c_{n}$ is the critical value introduced in Theorem \ref{OptimalValue} (green).\label{fig1}}
\end{figure}
\end{remark}

\end{paragraph}

\begin{paragraph}{Case of an angular sector.}
We investigate here the case where $\Omega$ is the angular sector of $\R^2$ defined by
$$
\Omega=\lbrace (r,\theta), -\theta_1\leq\theta \leq \theta_1, 0\leq r\leq R \rbrace,
$$
with $\theta_1 \in (0,\frac{\pi}{4}]$ and $R>0$. 

We consider the normalized eigenfunctions $(\varphi_{n,k})_{(n,k)\in \N^{*2}}$ of the Dirichlet-Laplacian given by 
$$ \varphi_{n,k}(r,\theta) = \frac{\sqrt{2}}{R \sqrt{\theta_1}} \frac{J_{\pi n /2\theta_1}(z_{n,k}\frac{r}{R})}{|J_{\pi n/2\theta_1}'(z_{n,k})|}  \sin\left(\frac{n\pi}{2\theta_1}(\theta+\theta_1)\right), $$ 
where $z_{n,k}$ denotes the $k$-th zero of the first kind Bessel function $J_{\pi n /\theta_1}$, associated to the eigenvalue $\lambda_{n,k}=\frac{z_{n,k}^2}{R^2}$ (see, e.g., \cite{BonnaillieLena}).

A tedious but straightforward computation leads to
\begin{multline*}
\lefteqn{J_\infty(a)=  \frac{2}{R^2\theta_1}  \inf_{(n,k)\in\N^{*2}}  \left( \int_{-\theta_1}^{\theta_1}a(R,\theta)\sin\left(\frac{n\pi}{2\theta_1}(\theta+\theta_1)\right)^2R d\theta \right.} \\
\left.   +  \frac{R^2n^2\pi^2}{4 z_{n,k}^2\theta_1^2}  \int_0^R \frac{J_{\pi n /2\theta_1}(z_{n,k}\frac{r}{R})^2}{|J_{\pi n/2\theta_1}'(z_{n,k})|^2} \frac{(a(r,-\theta_1)+a(r,\theta_1))}{r^2} dr   \right) 
\end{multline*}
for every $a\in \Ub$, with 
$$
\Ub = \left\lbrace a\in L^\infty(\partial\Omega,[-M,M]) \ \mid\ \int_{-\theta_1}^{\theta_1}a(R,\theta)Rd\theta + \int_0^R \left[  a(r,-\theta_1)+a(r,\theta_1) \right] dr = (2\theta_1+2)RL \right\rbrace .
$$
The situation on the straight sides of the angular sector is a bit intricate because of the presence of Bessel functions. For this reason, even if we completely solve the convexified optimal design problem \eqref{defJrelax}, we only provide a partial answer for Problem \eqref{mainPbOpt}. 
\begin{proposition}\label{prop:angsectRes}
Let $L^c_{n} = \min \{1,\frac{\theta_1 (1+\tan \theta_1)}{(1+\theta_1)\tan \theta_1}\}$. We have
\begin{equation}\label{NoBifurcation}
\max_{a\in\Ub} J_\infty(a)=
\left\{\begin{array}{ll}
 \frac{L\Hn(\partial\Omega)}{|\Omega|} &  \text{ if }  |L|\leq L^c_{n} , \\
\frac{(L+L^c_{n})\Hn(\partial\Omega)}{2|\Omega|} & \text{ if } L^c_{n} < |L|.
\end{array}\right.
\end{equation}
Moreover, if $|L|\geq L^c_{n}$, there is no-gap between the optimal values for Problem \eqref{mainPbOpt} and its convexified version \eqref{defJrelax}.
\end{proposition}
Notice that we were not able to conclude in the case where $|L|<L^c_{n}$. The proof of this proposition is done in Section \ref{proof:propSectAng}.
\end{paragraph}


\section{Conclusion and further comments}\label{FurtherComments}

\subsection{Relationships with spectral shape sensitivity analysis}\label{sec:shapesens}

Theorem \ref{prop:metz1756} happens to have interesting consequences in terms of spectral shape sensitivity analysis that provide another motivation of our study. In particular, we will investigate the problem of minimizing the sensitivity of (Dirichlet-Laplacian) eigenvalues with respect to shape perturbations, searching whether there exists a domain such that any perturbation of it would make all eigenvalues decrease.

Assume that $\Omega$ is a bounded connected domain with a boundary at least of class $\mathcal{C}^2$. Let $\mathbf{V}\in W^{3,\infty}(\R^n, \R^n)$ be a vector field. Then, the mapping $\Phi_{\mathbf{V}}=\operatorname{Id}+\mathbf{V}$ is a diffeomorphism provided that $\Vert \mathbf{V}\Vert_{3,\infty}\leq \varepsilon$ for some $\varepsilon>0$ small enough (see, e.g., \cite{henrot-pierre}). Moreover, under this condition, $\Phi_{\mathbf{V}}(\Omega)$ is a bounded connected domain having a $\mathcal{C}^2$ boundary. 

Let us assume that the Dirichlet-Laplacian spectrum of $\Omega$ is simple (in the sense that it consists of simple eigenvalues). It is well known that this assumption is generic with respect to the domain $\Omega$ (see, e.g., \cite{micheletti,ule,hillairet-judge}). 

In a nutshell, the shape derivative of $\lambda_j(\Omega)$, denoted $\langle d\lambda_j(\Omega),\mathbf{V}\rangle$, is the first-order term in the asymptotic expansion of $\lambda_j(\Phi_{\mathbf{V}}(\Omega))$ with respect to $\mathbf{V}$,  whenever it exists. Under the previous assumptions on $\Omega$, there holds
$$
\langle d\lambda_j(\Omega),\mathbf{V}\rangle =\lim_{t\searrow 0}\frac{J_\infty(\Phi_{t\mathbf{V}}(\Omega))-J_\infty(\Omega)}{t}=-\int_{\partial \Omega} \left(\frac{\partial \phi_j}{\partial \nu}\right)^2(\mathbf{V}\cdot \nu ) \, d\Hn
$$
and thus
$$
J_\infty(\mathbf{V}\cdot \nu )=-\sup_{j\geq 1}\frac{1}{\lambda_j}\langle d\lambda_j(\Omega),\mathbf{V}\rangle.
$$

We will  provide a partial answer to the following question:
\begin{quote}
\textsl{Do there exist bounded connected domains of $\R^n$ that are spectrally monotonically sensitive, in the sense that a perturbation $\mathbf{V}$ chosen as previously (preserving in particular the volume of the domain) makes all eigenvalues non-increase/decrease?}
\end{quote}
It is easy to see that, if $\Omega$ denotes any minimizer of a Dirichlet-Laplacian eigenvalue, which is for instance the case whenever $\Omega$ denotes a ball according to the Faber-Krahn inequality (see, e.g., \cite{henrot-pierre}), a perturbation field $\mathbf{V}$ enjoying the property above does not exist.

The following result is a byproduct of Theorem \ref{prop:metz1756}, dealing with shape sensitivity of the eigenvalues at the first-order.

\begin{corollary}\label{corSO}
Let us assume that $\Omega$ has a $\mathcal{C}^{2}$ boundary and that the spectrum of $\Omega$ consists of simple eigenvalues. Then, one has
$$
\sup_{j\geq 1}\langle d\lambda_j(\Omega),\mathbf{V}\rangle\geq 0
$$
for every vector field $\mathbf{V}\in W^{3,\infty}(\R^n,\R^n)$ such that $\int_{\partial \Omega}\mathbf{V}\cdot \nu \, d\Hn=0$. In other words, it is not possible to make all the eigenvalues $\lambda_j(\Omega)$ decrease at the first order under the action of a diffeomorphism $\Phi_{\mathbf{V}}$ preserving the volume of $\Omega$.
\end{corollary}
Indeed, for a vector field $\mathbf{V}$ as in the statement of Corollary \ref{corSO}, let us set $a=\mathbf{V}\cdot \nu$ and consider the problem
\begin{equation}\label{pbbbVnu}
\sup \bigg\{J_\infty(\mathbf{V}\cdot \nu )\ \ \mid\ \ \mathbf{V}\in W^{3,\infty}(\R^d,\R^d), \ \ \mathbf{V}\cdot \nu \in L^\infty(\partial\Omega)
\text{ and }\int_{\partial \Omega}\mathbf{V}\cdot \nu = 0\bigg\}.
\end{equation}
According to Theorem \ref{prop:metz1756}, there holds $J_\infty(\mathbf{V}\cdot \nu)\leq 0$ for every admissible vector field $\mathbf{V}$. Then, the optimal value of Problem \eqref{pbbbVnu} is non-positive which rewrites
$
-\sup_{j\geq 1}\langle d\lambda_j(\Omega),\mathbf{V}\rangle\leq 0.
$

\begin{remark}
In some sense, Problem \eqref{mainPbOpt} can be related to the large family of extremal spectral problems in shape optimization theory, where one looks for a domain minimizing or maximizing a numerical functions depending either on the eigenvalues or the eigenfunctions of an elliptic operator with various boundary conditions and geometric constraints (involving for instance the volume, perimeter or diameter of the domain). One of the most famous problems within this family is to minimize the first eigenvalue of the Dirichlet Laplacian operator among open subsets of $\R^n$ having a prescribed Lebesgue measure $c>0$. According to the so-called Faber-Krahn inequality, the solution is known to be the ball of volume $c$. 
For a review of such problems, we refer for instance to \cite{MR2150214,MR2731611,MR2251558,MR3681143,MR1804683}. 
\end{remark}

\subsection{Interpretation of our results in observation theory}\label{append:random}

The problem of optimizing the number, the position or the shape of sensors in order to improve the estimation of the state of the system has been widely investigated, in particular in the engineering community, with applications, for instance, to structural acoustics, piezoelectric actuators, or vibration control in mechanical structures. 
The literature on optimal observation is abundant in engineering applications, where mainly the optimal location of sensors or controllers is investigated (see, e.g., \cite{Afifi} for boundary actuators, \cite{Darde} in the context of electrical impedance tomography), but not the optimization of their shape.
In \cite{armaoua,harris,morris,vandewouwer}, numerical tools have been developed to solve a simplified version of the optimal design problem where either the partial differential equation has been replaced with a discrete approximation, or the class of optimal designs is replaced with a compact finite-dimensional set.

The problem of optimizing the shape of the sensors, without any restriction on their complexity or regularity, is infinite-dimensional and has been only little considered.
In \cite{henrot_hebrardSCL,henrot_hebrardSICON}, the authors investigated the problem of determining the best possible shape and position of the support of a damping term in the 1D wave equation, and they highlighted the so-called {\it spillover phenomenon} arising when considering spectral approximations. 
Their approach was spectral, and was based on Fourier expansions of the solution, as we do in the following.

Let $T>0$. We consider the homogeneous wave equation with Dirichlet boundary condition
\begin{equation}\label{waveEqobs}
\begin{array}{ll}
\partial_{tt} y (t,x)-\Delta y(t,x)=0 &  (t,x)\in [0,T]\times\Omega,\\
y(t,x)=0 &  (t,x)\in [0,T]\times\partial\Omega,\\
\end{array}
\end{equation}
It is well known that, for all $(y^0,y^1)\in H_0^1(\Omega,\C)\times L^2(\Omega,\C)$, there exists a unique solution $y\in C^0([0,T],H_0^1(\Omega,\C))\cap C^1((0,T),L^2(\Omega,\C))$ of \eqref{waveEqobs} such that $y(0,x)=y^0(x)$ and $\partial_t y (0,x)=y^1(x)$ for every $x\in\Omega$.

For a given measurable subset $\Gamma\subset\partial\Omega$, we consider the observable variable $z_\Gamma$ defined by 
\begin{equation}\label{obswave}
\forall (t,x)\in [0,T]\times \partial\Omega , \quad z_\Gamma(t,x) = \chi_\Gamma(x) \frac{\partial y}{\partial\nu}(t,x).
\end{equation}

By definition, the observability constant $C_T(\Gamma)$ is the largest nonnegative constant $C$ such that
\begin{equation}\label{ineqobs}
C \Vert (y(0,\cdot),\partial_ty(0,\cdot))\Vert_{H_0^1(\Omega,\C)\times L^2(\Omega,\C)}^2
\leq \int_0^T\int_{\partial\Omega}\chi_\Gamma (x) \left|\frac{\partial y}{\partial\nu}(t,x)\right|^2 \,d\Hn \,dt,
\end{equation}
for any solution $y$ of \eqref{waveEqobs}. It may be equal to $0$. If $C_T(\Gamma)>0$ then the system \eqref{waveEqobs}-\eqref{obswave} is said to be {\it observable\footnote{It is well known that observability holds true in large time if $\Gamma=\{x\in\partial\Omega\ \mid\  \langle x-x^0, \nu(x)\rangle>0\}$ for some $x^0\in \Omega$ (proof by multipliers, see \cite{ho,morawetz}). 
Within the class of $\mathcal{C}^\infty$ domains, observability holds true if $(\Gamma,T)$ satisfies the \textit{Geometric Control Condition (GCC)} (see \cite{BLR}), and this sufficient condition is almost necessary.
We refer to  \cite{TW, Zuazua} for an overview of boundary observability results for wave-like equations.
} in time $T$}.

From the point of view of applications, $\Gamma \subset \partial\Omega$ represents the domain occupied by some sensors that have been put at the boundary of the domain.
The role of the sensors is to achieve some measurements over a time horizon $[0,T]$, with which one aims at reconstructing the whole state of the system over $[0,T]$. Here, the partial measurement is $z_\Gamma(t,x) = \chi_\Gamma(x) \frac{\partial y}{\partial\nu}(t,x)$, given by \eqref{obswave}, and the complete state if the solution $y$ of \eqref{waveEqobs}. Since the solution of the wave equation is determined by its initial data, the observability inequality \eqref{ineqobs} ensures that the inverse problem of reconstructing the whole state from its partial measurement is well-posed, if $C_T(\Gamma)>0$.

Interpreting $C_T(\Gamma)$ as a quantitative measure of the well-posed character of the aforementioned inverse problem, one could be led to model the optimal shape of sensors issue as the problem of maximizing $C_T(\Gamma)$.
Nevertheless, this constant is \textit{deterministic} and provides an account for the \textit{worst possible case}. Hence, in this sense, it is a \textit{pessimistic} quantity. In practice, when realizing a large number of measures, it may be expected that this worst case does not occur so often, and then we realize that it is more desirable to have a notion of optimal observation in average, for a large number of experiments. For this reason, we adopt the point of view developed in \cite{PTZobsND, PTZparab} and inspired from the works \cite{Burq,BurqTzvetkov1}, which consists of maximizing what is referred to in these works as {\em the randomized observability constant}. This quantity can be interpreted as an average of the worst observation $L^2$-norms over almost all initial data.

A spectral expansion of the solution $y$ of Equation \eqref{waveEqobs} leads to the following expression of $C_T(\Gamma)$, namely
\begin{equation*}
\begin{split}
C_T(\Gamma) 
&= \inf_{\substack{(\widetilde a_j),(\widetilde b_j)\in\ell^2(\mathbb{C})\\ \sum_{j=1}^{+\infty} (| \widetilde a_j|^2+|\widetilde b_j|^2)=1}}\int_0^T\int_{\Gamma}\left|\sum_{j=1}^{+\infty}\left(\frac{\widetilde a_j}{\sqrt{\lambda_j}}e^{i\sqrt{\lambda_j}t}+\frac{\widetilde b_j}{\sqrt{\lambda_j}} e^{-i\sqrt{\lambda_j}t}\right)\frac{\partial \phi_j}{\partial \nu}(t,x)\right|^2\, d\Hn \, dt.
\end{split}
\end{equation*}

Making a random selection of all possible initial data for the wave equation \eqref{waveEqobs} consists in replacing $C_T(\Gamma)$ with the so-called {\em random observability constant} defined by
\begin{equation}\label{CTrand}
C_{T,\textrm{rand}}(\Gamma)= \hspace{-0.5cm} \inf_{\substack{(\widetilde a_j),(\widetilde b_j)\in\ell^2(\mathbb{C})\\ \sum_{j=1}^{+\infty}(|\widetilde a_j|^2+|\widetilde b_j|^2)=1}}  \hspace{-0.2cm} \mathbb{E}\left(\int_0^T  \int_\Gamma\left| \sum_{j=1}^{+\infty}\left(\frac{\beta_{1,j}^\omega \widetilde a_j}{\sqrt{\lambda_j}} e^{i\lambda_jt}+\frac{\beta_{2,j}^\omega \widetilde b_j}{\sqrt{\lambda_j}} e^{-i\lambda_jt}\right)\frac{\partial \phi_j}{\partial \nu}(x)  \right|^2 \hspace{-0.1cm} d\Hn \, dt\right),
\end{equation}
where $(\beta_{1,j}^\omega)_{j\in\mathbb{N}^*}$ and $(\beta_{2,j}^\omega)_{j\in\mathbb{N}^*}$ are two sequences of independent random variables on a probability space $(A,\mathcal{A},\mathbb{P})$ having mean equal to zero, variance equal to 1 and a super-exponential decay (for instance, independent Bernoulli random variables, see \cite{Burq, BurqTzvetkov1} for more details). Here, $\mathbb{E}$ is the expectation in the probability space, and runs over all possible events $\omega$.

An obvious remark is that, for any problem consisting of optimizing the observation, the optimal solution consists of observing the solutions over the whole domain $\partial\Omega$. This is however clearly not reasonable nor relevant and in practice the domain covered by sensors is limited, due for instance to cost considerations. From the mathematical point of view, we model this basic limitation by considering as the set of unknowns, the set 
$$
\mathcal{V}_{L,M}=\left\{ \chi_\Gamma\ \vert\ \Gamma\ \subset \partial\Omega \ \textrm{ and}\ \Hn(\Gamma)=L\Hn(\partial \Omega )\right\}
$$
for some $L\in [0,1]$, and therefore, it is relevant to model the problem of determining the optimal shape and location of boundary sensors as
$$
\sup\{C_{T,\textrm{rand}}(\Gamma), \ \chi_\Gamma\in\mathcal{V}_{L,M} \},
$$
which is very close to the optimal design problem \eqref{mainPbOpt} with $M=1$, as stated in the next result.

\begin{proposition}\label{propHazardCst}
Let $\Gamma\subset\partial\Omega$ be measurable. We have
\begin{equation}
C_{T,\textrm{rand}}(\Gamma)=T\inf_{j\in\mathbb{N}^*} \frac{1}{\lambda_j}\int_{\Gamma}\left(\frac{\partial \phi_j}{\partial \nu}(x)\right)^2 \, d\Hn .
\end{equation}
\end{proposition}

We claim that the approach developed within this article can be adapted without difficulty to admissible sets of characteristic functions. Moreover, up to slight changes in their formulation, all the conclusions of this article for Problems \eqref{truncPbOpt}, \eqref{defJrelax} and \eqref{mainPbOpt} remain valid.

\begin{remark}
In \cite{PTZCont1,PTZObs1,PTZparab,PTZobsND}, where {\em internal} subsets were considered to be optimized, a closely related optimal design problem has been modeled, consisting of maximizing the infimum over all modes of $\int_\omega\phi_j(x)^2\, dx$. The study required assumptions on the asymptotics of $\phi_j^2$, and led to  {\em quantum ergodicity properties}, i.e., asymptotic properties of the probability measures $\phi_j(x)^2\, dx$. 
\end{remark}


\subsection{Final comments and perspectives}\label{sec:final}
The study developed in this article can be extended in several directions. 
\begin{itemize}
\item {\bf Determination of the maximizers for Problem \eqref{defJrelax} when $L>L^c_{M,n}$.} In such a case, we know that there does not exist any admissible Rellich function (see Definition \ref{def:Rellichadmf}), in other words there does not exist $x_0\in \overline{\Omega}$ such that $\tilde a_{x_0}$ belongs to $\Ub$. According to the analysis performed in Section \ref{Invest} in the case where $\Omega$ is a two-dimensional rectangle, as well as the numerical computations on ellipses plotted on Fig. \ref{fig1}, we conjecture that a maximizer in such a case is given by
$$
a^*=\Pi (\tilde a_{x_0}),
$$
where $x_0\in \overline{\Omega}$ and $\Pi$ denotes the truncation mapping defined by $\Pi(x)=x$ whenever $0\leq x\leq M$ and $\Pi(x)=M$ if $x>M$.
\item {\bf Other boundary conditions.} 
The optimal design problem investigated in this article involves the Dirichlet-Laplacian eigenfunctions. Considering for instance the second motivations mentioned in the introduction of this article, namely the optimal location of sensors issues, it would also be relevant to investigate an optimal design problem involving now either Neumann or Robin boundary conditions in \eqref{waveEqobs}.

In such a case, defining the observable variable $z_\Gamma$ by \eqref{obswave} is no longer possible. In the case of Neumann boundary conditions, one has to consider the observable variable
$$ z_{\Gamma} (t,x) =\chi_\Gamma (x) \left( \partial_t y(t,x)+ \nabla y(t,x) \right)$$
where $y$ is the solution of the wave equation with Neumann boundary conditions\footnote{More precisely, $y$ solves
$$
\left\lbrace 
\begin{array}{ll}
\partial_{tt}y(t,x) - \Delta y(t,x) = 0 & \forall (t,x)\in(0,T)\times \Omega, \\
\partial_\nu y(t,x) = 0 & \forall (t,x)\in(0,T)\times \partial\Omega, \\
y(0,x)=y^0(x), \;  \partial_t y(0,x) = y^1(x) & \forall x\in \Omega, \\ 
\end{array} \right.
$$
with $(y^0,y^1)\in H^1(\Omega,\C)\times L^2(\Omega,\C)$.}.

The observability inequality that must be considered here writes: for all $ (y_0,y_1)$ in $H^1(\Omega,\C)\times L^2(\Omega,\C) $,
$$ 
C_T(\Sigma) \Vert (y_0,y_1) \Vert^2_{H^1(\Omega,\C)\times L^2(\Omega,\C)} \leq  \int_0^T \int_{\partial\Omega} \chi_{\Gamma}(x) (\partial_t y^2(t,x)+ |\nabla y(t,x)|^2)d\Hn dt.
$$

Using a \textit{randomization procedure} leads to introduce the optimal design problem
$$ 
\sup_{v\in \Ub} \inf_{j\in\N} \int_{\partial\Omega} v\left(\phi_j^2 + \frac{1}{\lambda_j}|\nabla \phi_j|^2 \right) d\Hn .
$$
where the sequence $(\phi_j)_{j\in \N}$ now denotes a Hilbert basis of eigenfunctions of the Neumann-Laplacian operator.

Under appropriate quantum ergodic assumptions on the domain $\Omega$ (see \cite{BurqErgo,HasselErgo}), the sequence $ \left( (\phi_j^2 + \frac{1}{\lambda_j}\nabla \phi_j^2) d\Hn \right)_{j\geq 1}$ converges vaguely to the uniform measure $ \frac{2}{n |\Omega|} d\Hn$ and we have moreover the following identity of Rellich type (see \eqref{Tao})
$$
\int_{\partial\Omega} \langle x,\nu(x) \rangle \left(\phi_j(x)^2 +\frac{|\nabla \phi_j(x)|^2}{\lambda_j} \right)d\Hn(x) = 2 ,
$$
holding for every bounded set $\Omega$ of $\R^n$ either convex or having a $C^2$ boundary.
As a consequence, mimicking the reasoning made in the case of Dirichlet boundary conditions, one infers that the optimal value provided in Theorem \ref{OptimalValue} still holds true when considering Neumann boundary conditions. 

Notice that it is also possible to generalize our results to the case of Robin boundary conditions. One then get analogous versions of the main theorems \ref{No-Gap} and \ref{analytics} in such cases.

\item {\bf Generalization of Theorem \ref{analytics}.} We think plausible that the result stated in Theorem \ref{analytics} holds in fact true for all domains connected bounded domain $\Omega$ convex or with a $\mathcal{C}^{1,1}$ boundary, except the disk. In some aspects, such an issue appears close to the famous problems in shape optimization entitled ''Schiffer conjecture'' or ''Pompeiu's problem'' (see, e.g., \cite{henrot-pierre}).

Investigating such an issue needs another approach and tools as the ones developed within this article. 

\item {\bf Optimal boundary control domain for the wave equation.}
Investigating the optimal domain for boundary observability is related to the issue of investigating the optimal domain for boundary control. Indeed, introduce the boundary control problem
\begin{equation}\label{ContPb}
\left\lbrace \begin{array}{ll}
\partial_{tt} w(t,x)-\Delta w(t,x)=0 & \forall (t,x)\in [0,T]\times\Omega, \\
w(t,x) = 0 & \forall (t,x)\in (\partial\Omega\setminus \Gamma)\times [0,T], \\
w(t,x) =u(t,x) & \forall (t,x)\in\Gamma\times [0,T] ,\\
w(0,x)= w^0(x), \; \partial_t w(0,x) = w^1(x), & \forall x\in\Omega ,\\
\end{array} \right.
\end{equation}
where the control $u$ belongs to $ L^2([0,T]\times\Gamma,\C)$. 
The Cauchy problem \eqref{ContPb} is well posed for every initial data $(w^0,w^1) \in H^1_0(\Omega,\C)\times L^2(\Omega,\C)$ and every control $u\in L^2([0,T]\times\Gamma,\C)$.
By duality, one has that {\em System \eqref{ContPb} is controllable in time $T$ if and only if the observation problem \eqref{waveEqobs}-\eqref{obswave} is observable in time $T$} (see \cite{TW}). 

Moreover, if Problem \eqref{ContPb} is exactly controllable, then applying the so-called Hilbert Uniqueness Method (HUM, see \cite{Lions1,Lions2}), an optimal control is given by 
$$
 u_\Gamma (t,x)=\chi_\Gamma (x) y(t,x) 
 $$
where $y$ denotes the solution of \eqref{waveEqobs} with initial conditions $(y^0,y^1)$ minimizing the functional\footnote{Here, the notation $\langle \cdot,\cdot \rangle_{H^{-1},H_0^1}$ stands for the standard duality bracket in $H^{-1}$ and $\langle \cdot,\cdot \rangle_{L^2}$ for the usual inner product in $L^2$.}
$$ \mathcal{F}_\Gamma (y^0,y^1)=\frac{1}{2} \int_0^T \int_\Gamma \left( \frac{\partial y}{\partial \nu}\right)^2 d\Hn dt - \langle y^1,w^0\rangle_{H^{-1},H_0^1} + \langle y^0,w^1\rangle_{L^2} $$
over $H^1_0(\Omega,\C)\times L^2(\Omega,\C)$.
Let us define the so-called HUM operator 
\begin{equation}\label{HUMoperator}
\begin{array}{lccl}
\Lambda_\Gamma : & H^1_0(\Omega,\C) \times L^2(\Omega,\C) & \rightarrow & L^2([0,T]\times\Omega,\C)\\ 
 & (w^0,w^1) & \mapsto & u_\Gamma \\
 \end{array}.
 \end{equation}
It is relevant to look for an optimal control domain minimizing the norm operator of $\Lambda_\Gamma$ over all possible domains $\Gamma\subset \partial\Omega$. 

As pointed out in \cite{MR3632257}, a randomization modeling approach is still relevant in that case. Nevertheless, the resulting randomized control problem is much more intricate than Problem \eqref{mainPbOpt} and its analysis is widely open.
\end{itemize}


\bigskip

\appendix

\begin{center}
\fbox{\Large{\bf Appendix - Proofs of Propositions \ref{prop:rect}, \ref{mainDisk} and \ref{prop:angsectRes}}}\label{sec:proofCaspart}
\end{center}\

In what follows, we will assume that $M=1$ for the sake of readability. The general result will be easily inferred by an immediate adaptation of the reasonings.

\section{Proof of Proposition \ref{prop:rect}}\label{sec:proofCaspart1}

This proof is inspired by \cite[Proposition 1 \& Theorem 1]{PTZObs1}. For this reason, we only provide a short sketch of proof, underlining the main steps.

Let us first solve the convexified optimal design problem \eqref{defJrelax}. According to the expression of $J_\infty(a)$ given by \eqref{CritSquare}, letting $n$ and $k$ going to $+\infty$ yields
\begin{equation}\label{Tensor}
 \max_{a\in\Ub}J_\infty(a) \leq  \frac{2}{\pi^2\alpha\beta} \max_{a\in\Ub} \min\left( \int_{\Sigma_1\cup\Sigma_3}a(x,y)dy,2L\pi(\alpha+\beta)-\int_{\Sigma_1\cup\Sigma_3}a(x,y)dy\right) ,
\end{equation}
by using that $\int_{\partial\Omega}a(x,y)\, d\Hn=2L\pi(\alpha+\beta)$ and as a consequence 
$$
 \max_{a\in\Ub}J_\infty(a) \leq \frac{2}{\pi^2\alpha\beta}\max_{t\in [0,L\pi (\alpha+\beta)]}\min  \{t,2L\pi (\alpha+\beta)-t\}=\frac{2L (\alpha+\beta)}{\pi\alpha\beta}.
 $$
Observe moreover that both components of the minimum above are equal at the optimum.
To compute the optimal value for the convexified problem, we will use Theorem \ref{OptimalValue}. 
Let us investigate the existence of Rellich-admissible functions. A simple computation shows that every Rellich-admissible function, whenever it exists, is necessarily constant on each side of $\Omega$, namely on $\Sigma_1$, $\Sigma_2$, $\Sigma_3$ and $\Sigma_4$. Moreover, for a given $x_0$ in $\Omega$ and when $x$ runs over $\partial\Omega$, the quantity $\langle x-x_0,\nu\rangle$ is successively equal to the distance of $x_0$ to each side of $\Omega$. Therefore, it is enough to investigate the case where $x_0=(0,0)$ to determine the set of parameters $L$, $\alpha$ and $\beta$ for which there exists Rellich-admissible functions. In other words, this question comes to determine $L$, $\alpha$ and $\beta$ such that the function $a^*$ defined by
$$ 
a^*|_{\Sigma_1\cup\Sigma_3} = L\frac{\alpha+\beta}{2\beta}\quad \text{and } \quad a^*|_{\Sigma_2\cup\Sigma_4}=L\frac{\alpha+\beta}{2\alpha}, 
$$
belongs to $\overline{\mathcal{U}}_{L,1}$.

It follows that there exists a Rellich-admissible function if, and only if $|L|\leq L^c_{n}$ (defined in the statement of Prop. \ref{prop:rect}).
Moreover, in this case, one has
\begin{equation}
\max_{a\in\Ub}J_\infty(a)=\frac{2 L(\alpha+\beta)}{\pi\alpha\beta}.
\end{equation}

Let us now investigate the converse case. Assume without loss of generality that $\beta<\alpha$ and $L\in (\frac{2\beta }{\alpha+\beta},1]$, the case $L\in [-1,-\frac{2\beta }{\alpha+\beta})$ being treatable in a similar way. Then, one has
\begin{eqnarray*}
\int_{\Sigma_1\cup\Sigma_3}a(x,y)\, dy& \leq &\operatorname{Per}(\Sigma_1\cup\Sigma_3)=2\beta \pi\\
\textnormal{and}\quad \int_{\Sigma_2\cup\Sigma_4}a(x,y)\, dx & = & 2L\pi(\alpha+\beta)-\int_{\Sigma_1\cup\Sigma_3}a(x,y)\, dy \geq 4\beta  \pi-2\beta \pi=2\beta \pi,
\end{eqnarray*}
meaning that $J_\infty(a)\leq \frac{4}{\pi\alpha}$
for every $a\in \overline{\mathcal{U}}_{L,1}$, according to \eqref{Tensor}. The right-hand side in this inequality is in fact reached by every function $a$ equal to 1 on $\Sigma_1\cup\Sigma_3$, constant on each side $\Sigma_2$ and $\Sigma_3$, where each constant is chosen in such a way that $a$ belongs to $\overline{\mathcal{U}}_{L,1}$. Notice that such a choice is not unique, and easy computations yield that maximizers are given by
$$
a_{\mid \Sigma_1\cup\Sigma_3}=1,\quad a_{\mid  \Sigma_2}=u\frac{L(\alpha+\beta)-\beta}{\alpha}\quad\textrm{and}\quad a_{\mid  \Sigma_4}=(2-u)\frac{L(\alpha+\beta)-\beta}{\alpha}.
 $$
with $u\in \left(1-\left( \frac{\alpha}{L(\alpha+\beta)-\beta}-1\right),1+\left(\frac{\alpha}{L(\alpha+\beta)-\beta}-1\right)\right)$.

The rest of the proof is a direct adaptation of the results of \cite[Theorem 1]{PTZObs1} and in particular of the fact that 
\begin{equation}\label{eq2233}
\sup_{\substack{\omega\subset (-\alpha\pi/2,\alpha\pi/2)\\ |\omega|=V_0}} \int_\omega \sin\left(\frac{n\pi x}{\alpha}\right)^2\, dx = \frac{V_0}{2}.
\end{equation}

Finally, the necessary and sufficient condition on $L$ guaranteeing the existence of solutions for the initial optimal design problem follows by using the same Fourier series method as in the proof of \cite[Theorem 1]{PTZObs1}.

\section{Proof of Proposition \ref{mainDisk}}\label{proof:propmainDisk}

The proof of this proposition is inspired by \cite[Theorem 3.2]{henrot_hebrardSCL} and \cite[Theorem 1]{PTZObs1}. 
First, notice that for every $a$ in $\Ub$, one has
$$J_\infty(a)\leq \lim_{n\to +\infty}\int_0^{2\pi}a(\theta)\cos(n\theta)^2\, d\theta=L\pi$$
and that $J_\infty(L)=L\pi$, yielding that the optimal value for Problem \eqref{defJrelax} is $\pi L$.
Moreover, the constant function equal to $L$ belongs to $\Ub$ since $M\geq 1$ and $a\in\Ub$ reaches $\pi L$ if, and only if 
$$
\inf_{n\geq 1} \int_0^{2\pi}a(\theta)\cos(n\theta)^2 R d\theta  = \inf_{n\geq 1} \int_0^{2\pi}a(\theta)\sin(n\theta)^2 R d\theta =  \pi L
$$
or similarly 
\begin{equation}\label{munich0640}
\forall n\in \N^*, \qquad \int_0^{2\pi}a(\theta)\cos(2n\theta) d\theta=0 .
\end{equation}
Considering the Fourier expansion $ a(\theta)=L+\sum_{j=1}^{+\infty} ( \alpha_j \cos(j\theta) + \beta_j \sin(j\theta) )$, the equality \eqref{munich0640} holds if, and only if $\alpha_{2j} = 0$, $\forall j\geq 1$.

The no-gap property between the optimal values for Problem \eqref{mainPbOpt} and its convexified version \eqref{defJrelax} is a consequence of Theorem \ref{No-Gap}.

Let us now investigate the no-gap property. Assume that $a=2\chi_{\Gamma}-1$ solves \eqref{mainPbOpt} and consider $a_e:\theta\mapsto \frac{a(\theta)+a(2\pi-\theta)}{2}$, its even part. First, $a_e$ is still a solution of \eqref{defJrelax} and $ a_e(\theta) = L + \sum_{j=1}^{+\infty} \alpha_j \cos(j \theta) $.
Setting now $\tilde{a}(\theta) = \frac{a_e(\theta)+a_e(\pi-\theta)}{2}$, one has
$$ \tilde{a}(\theta) = L + \sum_{j=1}^{+\infty} \frac{\alpha_j}{2}(1+(-1)^j)\cos(j \theta) .$$
Since $\tilde{a}$ solves \eqref{defJrelax}, we have $\alpha_{2j}=0$, for all $j\geq 1$. Thus, $\tilde{a}$ is necessarily constant equal to $L$.
Finally since $a\in \mathcal{U}_{L,M} $, hence the range of $a_e$ is contained in $\lbrace -1, 0, 1\rbrace$, and finally $\tilde{a}$ in $\lbrace -1,-1/2,0,1/2,1  \rbrace $. This yields that Problem  \eqref{mainPbOpt}  has no solution if $L\notin \lbrace -1,-1/2,0,1/2,1  \rbrace$.

Conversely, if $L\in \lbrace -1,-1/2,0,1/2,1  \rbrace$, a direct adaptation of the construction of solutions done in the proof of \cite[Theorem 1]{PTZObs1} or \cite[Lemma 3.1]{henrot_hebrardSCL}  yields the expected conclusion.

\section{Proof of Proposition \ref{prop:angsectRes}}\label{proof:propSectAng}
Denote by $\Sigma_1$, $ \Sigma_2$ and $\Sigma_3$ the subsets of $\partial\Omega$ defined by 
$$
\Sigma_1=\partial\Omega\cap \lbrace \theta=-\theta_1\rbrace, \quad
\Sigma_2=\partial\Omega\cap \lbrace \theta=\theta_1\rbrace\quad\textrm{ and }\quad
\Sigma_3=\partial\Omega\cap \lbrace r=R\rbrace .
$$

Fixing $n\in\N^*$ and letting $k$ tend to $+\infty$ shows that the sequence $(\frac{1}{\lambda_{n,k}} \left( \frac{\partial \varphi_{n,k}}{\partial \nu}  \right)^2)_{(n,k)\in\N^{*2}}$ does not converge to a constant on $\partial\Omega$. Therefore, $\Omega$ does not satisfy the (QUEB) property.

We do not know if $\Omega$ satisfies the (WQEB) property. Anyway, we bypass this difficulty by showing a weaker version of this property. This is the purpose of the next two lemmas.

The results stated in the two next lemmas are inspired by \cite[Proposition 4]{PTZobsND}.

\begin{lemma}\label{Lem7} 
For $s\in (0,1)$, let us denote by $F_s$ the function $[s,1]\ni x\mapsto \int_s^x \frac{du}{u^2\sqrt{u^2-s^2}}$. We have
$$ 
\sup_{a\in \mathcal{V}} \;  \inf_{s\in (0,1)} \; \frac{1}{F_s(1)}\int_0^1  \frac{a(u) \chi_{(s,1)}(u)}{u^2 \sqrt{u^2-s^2}} du  = 1
$$  
where $\mathcal{V}= \lbrace a\in L^\infty([0,1])\ \mid\ \int_0^1 a(u)\, du = 1 \rbrace$.
\end{lemma}

\begin{proof}
Define the function $K: \Ub \ni a \mapsto \inf_{s\in (0,1)} \; \frac{1}{F_s(1)}\int_0^1  \frac{a(u) \chi_{(s,1)}(u)}{u^2 \sqrt{u^2-s^2}} \, du $. Note that $ K(a=1) = 1$ by definition of $F_s$ and that the infimum in the definition of $K$ is reached for every $s\in (0,1)$. As an infimum of linear functions, the function $K$ is concave. Therefore, to prove the lemma, it  is enough to show that the directional derivative of $K$ at $a=1$ in every admissible direction $h$ satisfies the first-order necessary optimality conditions, namely that
$dK(a=1).h \leq 0$ for every function $h$ defined on $[0,1]$ such that $\int_0^1 h(u)du =0$. 

According to Danskin's theorem\footnote{There is however a small difficulty here in applying Danskin's Theorem, due to the fact that the set $[0,1)$ is not compact. This difficulty is easily overcome by applying the slightly more general version \cite[Theorem D2]{bernhard} of Danskin's Theorem, noting that for $a=1$ every $s\in(0,1)$ realizes the infimum in the definition of $K$.}, we have
$$ 
dK(a=1).h  = \inf_{s\in(0,1)} \frac{1}{F_s(1)}\int_0^1  \frac{h(u) \chi_{(s,1)}(u)}{u^2 \sqrt{u^2-s^2}} du.
$$
By contradiction assume that there exists a function $h$ such that
$$ \frac{1}{F_s(1)}\int_0^1  \frac{h(u) \chi_{(s,1)}(u)}{u^2 \sqrt{u^2-s^2}} du > 0 ,$$
for every $s\in (0,1)$.
One has
$$
0 < \int_{s=0}^1 \int_{u=s}^1 \frac{s \; h(u)}{u \sqrt{u^2-s^2}} \, du \, ds= \int_{u=0}^1 \int_{s=0}^u \frac{s \; h(u)}{u \sqrt{u^2-s^2}} \, ds \, du = - \int_0^1 h(x) dx = 0,
$$
leading to a contradiction. The conclusion follows.
\end{proof}

\begin{lemma}\label{weakWQEB}
For every $a\in\Ub$, one has
$ J_\infty(a) \leq \frac{L \Hn (\partial\Omega)}{|\Omega|}.$
\end{lemma}
\begin{proof}
Let $a\in\Ub$. We will show the existence of a subsequence of $\left( \int_{\partial\Omega} \frac{a(x)}{\lambda_{n,k}} \left( \frac{\partial \varphi_{n,k}}{\partial \nu} \right)^2 \, d\Hn   \right)_{(n,k)\in \N^{*2}}$
converging to $\frac{1}{\theta_1 R^2} \int_{\partial\Omega} a(x)\, d\Hn=\frac{L \Hn (\partial\Omega)}{|\Omega|}$. 

Since $\frac{1}{\lambda_{n,k}} \left( \frac{\partial \varphi_{n,k}}{\partial \nu}  \right)^2 = \frac{2}{R^2 \theta_1} \sin \left(\frac{n\pi}{2\theta_1} (\theta+ \theta_1)\right)$ on $\Sigma_3$, the sequence $\left(\int_{\Sigma_3} \frac{a(x)}{\lambda_{n,k}}\left( \frac{\partial \varphi_{n,k}}{\partial \nu} \right)^2 d\Hn \right)_{n\in\N^*}$ converges to $\frac{1}{\theta_1 R^2} \int_{\Sigma_3} a(x)\, d\Hn $, according to the Riemann-Lebesgue Lemma.

To conclude the proof, it is enough to prove the existence of a positive number $\alpha$ such that, by keeping the ratio $n/k$ constant equal to $\alpha$, there holds 
$$ 
\lim_{n \to +\infty} \int_{\Sigma_1 \cup \Sigma_2} \frac{a(x)}{\lambda_{n,k}}\left( \frac{\partial \varphi_{n,k}}{\partial \nu} \right)^2 d\Hn  \leq \frac{1}{\theta_1 R^2} \int_{\Sigma_1\cup\Sigma_2} a(x)\, d\Hn .
 $$

Setting $\mu = \frac{n\pi}{2\theta_1}$ and introducing $\Phi_{n,k}: [0,1]\ni u\mapsto \frac{ n^2 \pi^2}{2 z_{n,k}^2 \theta_1^2} \left(\frac{J_\mu\left( z_{n,k}\frac{r}{R}\right)}{r J'_\mu(z_{n,k})} \right)^2$, one has 
$$ 
\frac{1}{\lambda_{n,k}}\left( \frac{\partial \varphi_{n,k}(r)}{\partial \nu} \right)^2 = \frac{1}{\theta_1}\Phi_{n,k}\left(\frac{r}{R}\right)\quad \textrm{ on }\quad \Sigma_1 \cup \Sigma_2.
$$
Hence, we have to prove that for every $\rho\in L^\infty([-1,1],[-1,1])$, there exists $\alpha>0$ such that for $n$ and $k$ chosen as previously,
\begin{equation}\label{ObjectifAngular} 
\lim_{n \to +\infty} \int_0^1 \rho (u) \Phi_{n,k}(u) \, du \leq \int_0^1 \rho (u) \, du .
\end{equation}
According to \cite[p. 257]{Luke}, one has
$ \int_0^1 \Phi_{n,k}(u)\, du =   \frac{n^2 \pi^2}{n^2\pi^2 - \theta_1}$.
Then, taking the weak limit of the sequence of measures $\Phi_{n,k}(u)\, du$ with a fixed ratio $n/k$, and making this ratio vary, we obtain the family of probability measures
$$ 
f_s(u)\, du =  \frac{1}{F_s(1)}\frac{\chi_{(s,1)}(u)}{u^2 \sqrt{u^2-s^2}} \, du$$
parametrized by $s\in (0,1)$.

To prove \eqref{ObjectifAngular}, let us argue by contradiction, by assuming that for every $s \in (0,1)$, one has 
$$ \frac{1}{F_s(1)}\int_0^1  \frac{\rho (u) \chi_{(s,1)}(u)}{u^2 \sqrt{u^2-s^2}} \, du > \int_0^1 \rho(u)\, du .$$
We obtain a contradiction by using Lemma \ref{Lem7} and the expected conclusion follows.
\end{proof}

In the next lemma, one computes the critical value $L^c_{n}$ introduced in Theorem \ref{OptimalValue}.
\begin{lemma}\label{lemma:Lcangsect}
Denote by $L^c_{n}$ the critical value for the constraint parameter $L$, as introduced in Theorem \ref{OptimalValue}. One has $L^c_{n} = \frac{\theta_1 (1+\tan \theta_1)}{(1+\theta_1)\tan \theta_1}$.
\end{lemma}
\begin{proof}
Following the proof of Theorem \ref{OptimalValue}, the critical value $L^c_{n}$ is given by \eqref{md11h30}. Therefore, the issue comes to determine the quantity
$ \delta:=\min_{x_0\in\overline{\Omega}} \max_{x\in\partial\Omega} \langle x-x_0, \nu_x \rangle$. For the sake of simplicity, the notations we will use are summed-up on  Figure \ref{croquis2}. Writing in polar coordinates $x_0=(r_0,\theta_0)$, the quantity $\langle x-x_0, \nu_x \rangle$ is equal to $r_0\sin(\theta_1+\theta_0)$ on $\Sigma_1$, $r_0\sin(\theta_1-\theta_0)$ on $\Sigma_2$ and $R-r_0 \cos(\theta_0-\theta)$ on $\Sigma_3$, where we denoted $x=(R,\theta)$ in polar coordinates on $\Sigma_3$.
 
\begin{figure}[h]
\begin{center}
\begin{tabular}{cc}
\begin{tikzpicture}[scale=0.7]
\coordinate (O) at (0,0);
\coordinate (A) at (4*0.7071,4*0.7071);
\coordinate (B) at (4*0.7071,-4*0.7071);
\coordinate (O2) at (1,0);
\coordinate (B2) at (1*0.7071,-1*0.7071);
\coordinate (A2) at (1*0.7071,1*0.7071);
\coordinate (x0) at (3*0.866,3*0.5);
\draw (1,-2) node {$\Sigma_1$};
\draw (1,2) node {$\Sigma_2$};
\draw (4.7,1) node {$\Sigma_3$};
\draw (O) node [below left] {$O$};
\draw [->] (-0.5,0)--++(5,0);
\draw (4.5,0) node [right] {$x$};
\draw [->](0,-3) --++ (0,6);
\draw (0,3) node [above] {$y$};
\draw [thick] (O) -- (A);
\draw [thick] (O) -- (B);
\draw [thick] (B) arc (-45:45:4);
\draw [<-](B2) arc (-45:0:1);
\draw [<-] (A2) arc (45:0:1);
\draw (B2) node [above] {$\theta_1$};
\draw (A2) node [below] {$\theta_1$};
\draw [->] (O)--(x0);
\draw (x0) node [right] {$x_0$};
\draw [->] (1.5,0) arc (0:30:1.5);
\draw (2,0.5) node {$\theta_0$};
\draw (2,1.5) node [below left] {$r_0$};
\draw [red, latex-latex] (2.9*0.7071,2.9*0.7071)--(x0);
\draw [red, latex-latex] (1*0.7071,-1*0.7071)--(x0);
\draw [red] (x0) node [above] {$d$};
\end{tikzpicture} & \begin{tikzpicture}[scale=0.7]
\coordinate (O) at (0,0);
\coordinate (A) at (4*0.7071,4*0.7071);
\coordinate (B) at (4*0.7071,-4*0.7071);
\coordinate (O2) at (1,0);
\coordinate (B2) at (1*0.7071,-1*0.7071);
\coordinate (A2) at (1*0.7071,1*0.7071);
\coordinate (x0) at (3*0.866,3*0.5);
\coordinate (x) at (4*0.99,-4*0.15);
\draw (1,-2) node {$\Sigma_1$};
\draw (1,2) node {$\Sigma_2$};
\draw (4.7,1) node {$\Sigma_3$};
\draw (O) node [below left] {$O$};
\draw [->] (-0.5,0)--++(5,0);
\draw (4.5,0) node [right] {$x$};
\draw [->](0,-3) --++ (0,6);
\draw (0,3) node [above] {$y$};
\draw [thick] (O) -- (A);
\draw [thick] (O) -- (B);
\draw [thick] (B) arc (-45:45:4);
\draw [->] (O)--(x0);
\draw (x0) node [right] {$x_0$};
\draw [->] (1.5,0) arc (0:30:1.5);
\draw [->] (1.6,0) arc (0:-9:1.6);
\draw (2,0.5) node {$\theta_0$};
\draw (2.2,-0.2) node  {$\theta$};
\draw (2.5,-0.5) node [below] {$R$};
\draw (2,1.5) node [below left] {$r_0$};
\draw (x) node {$\times$};
\draw (x) node [right] {$x$};
\draw [color=LimeGreen!100] (x)--++(4*0.15,4*0.99)--++(-7*0.15,-7*0.99);
\draw [thin, latex-latex](O)--(x);
\draw [color=LimeGreen!100] (x0)--++(3*0.866,3*0.5);
\draw [color=LimeGreen!100] (x0)--++(3*0.99,-3*0.15);
\draw [color=LimeGreen!100] (4.2*0.866,4.2*0.5) node [above] {$d_1$};
\draw [color=LimeGreen!100] (3*0.866+1.5*0.99,3*0.5-1.5*0.15) node [above left] {$d$};
\end{tikzpicture}\\
(a) & (b) \\
\end{tabular}
\end{center}
\caption{Notations}\label{croquis2}
\end{figure}
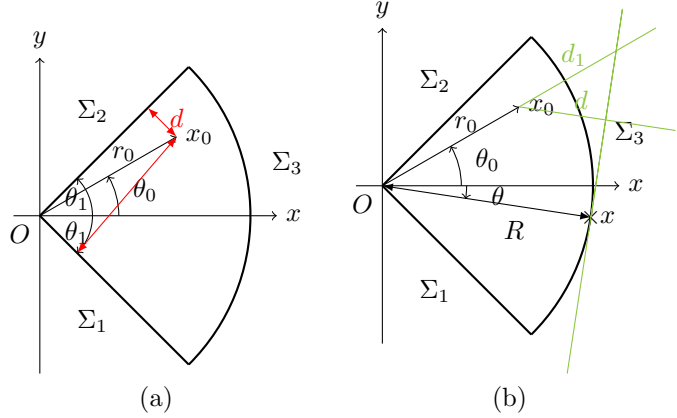

As a consequence, one computes using symmetry arguments
\begin{eqnarray*}
\delta &=& \min_{\substack{0\leq r_0\leq R\\ \theta_0\in [-\theta_1,\theta_1]}} \max \left\{ R-r_0\cos(\theta_0+\theta_1), R-r_0\cos(\theta_0-\theta_1), r_0 \sin(\theta_1-\theta_0),r_0 \sin(\theta_1+\theta_0) \right\}\\
& = & \min_{0\leq r_0\leq R} \max \left\{ R-r_0\cos \theta_1 , r_0 \sin \theta_1 \right\},
\end{eqnarray*}
and the maximum is reached provided that  $R-r_0\cos\theta_1= r_0 \sin \theta_1$. It follows that 
$ \delta= \frac{R\tan(\theta_1)}{1+\tan(\theta_1)}$, leading to the expected conclusion.
\end{proof}

In the following lemma, we compute the optimal value of Problem \eqref{defJrelax}. 
\begin{lemma}\label{NoBifurcationProp}
Let $L^c_{n} = \frac{\theta_1 (1+\tan \theta_1)}{(1+\theta_1)\tan \theta_1}$. We have
$$
\max_{a\in\Ub} J_\infty(a)=
\left\{\begin{array}{ll}
 \frac{L\Hn(\partial\Omega)}{|\Omega|} &  \text{ if }  |L|\leq L^c_{n} , \\
\frac{(L+L^c_{n})\Hn(\partial\Omega)}{2|\Omega|} & \text{ if } L^c_{n} < |L|.
\end{array}\right.
$$
\end{lemma}
\begin{proof}
According to Lemma \ref{weakWQEB}, one has $ J_\infty(a) \leq \frac{L \Hn (\partial\Omega)}{|\Omega|}$ and according to Lemma \ref{lemma:Lcangsect}, the right-hand side in this inequality is reached by every Rellich-admissible function if $|L|\leq L^c_{n}$.
 
Assume now that $L>L^c_{n}$ (the case $L<-L^c_n$ being exactly similar).
Let us introduce $a_c$, as the solution of the optimal design problem \eqref{defJrelax} in the case where $L=L^c_{n}$. One verifies that $a_c=1$ on $\Sigma_1\cup \Sigma_2$ and $a_c = R-r_0\cos(\theta)$ on $\Sigma_3$. Let $a\in \Ub$ and let us decompose $a$ as $a=a_c+h$, so that $ \int_{-\theta_1}^{\theta_1}a(R,\theta) R d\theta = L\Hn(\partial\Omega) - 2R$, $ \int_{-\theta_1}^{\theta_1}h(R,\theta) R d\theta+\int_{\Sigma_1\cup \Sigma_2}h = (L-L^c_{n})\Hn(\partial\Omega)$ and $h <0$ on $\Sigma_1\cup \Sigma_2$.

Let us apply \eqref{Tao} with $L=L^c_{n}$. One gets 
$$ 
\frac{R^2n^2\pi^2}{2 z_{n,k}^2\theta_1^2}  \int_0^R \frac{J_{\pi n /2\theta_1}(z_{n,k}\frac{r}{R})^2}{|J_{\pi n/2\theta_1}'(z_{n,k})|^2} \frac{dr}{r^2} = L^c_{n}\Hn(\partial\Omega) - \int_{-\theta_1}^{\theta_1}a_c(R,\theta)\sin\left(\frac{n\pi}{2\theta_1}(\theta+\theta_1)\right)^2R d\theta .
$$
Using this identity, we claim that $J_\infty(a)=\inf_{(n,k)\in \N^{*2}}j_{n,k}(a)$, where  
\begin{eqnarray*}
j_{n,k}(a) &=& \frac{2}{R^2\theta_1}\int_{-\theta_1}^{\theta_1}a(R,\theta)\sin\left(\frac{n\pi}{2\theta_1}(\theta+\theta_1)\right)^2\, R d\theta \\
& &  + \frac{n^2\pi^2}{ z_{n,k}^2\theta_1^3}  \int_0^R (1+h(r,\theta_1)+h(r,-\theta_1))\frac{J_{\pi n /2\theta_1}(z_{n,k}\frac{r}{R})^2}{|J_{\pi n/2\theta_1}'(z_{n,k})|^2} \frac{dr}{r^2}\\
& \leq & \frac{2}{R^2\theta_1} \left( \int_{-\theta_1}^{\theta_1}h(R,\theta) \sin\left(\frac{n\pi}{2\theta_1}(\theta+\theta_1)\right)^2R d\theta    + L^c_{n}\Hn(\partial\Omega)  \right) \\
&=&\frac{(L+L^c_{n})\Hn(\partial\Omega)}{2|\Omega|} - \frac{1}{R^2\theta_1}  \int_{-\theta_1}^{\theta_1}h(R,\theta)\cos\left(\frac{n\pi}{\theta_1}(\theta+\theta_1) \right) \, R d\theta. 
\end{eqnarray*}
As a consequence and according to the Riemann-Lebesgue lemma, one has
\begin{eqnarray*}
J_\infty(a) & \leq & \frac{(L+L^c_{n})\Hn(\partial\Omega)}{2|\Omega|} - \frac{1}{R^2\theta_1}\sup_{n\in\N} \left\lbrace \int_{-\theta_1}^{\theta_1}h(R,\theta) \cos\left(\frac{n\pi}{\theta_1}(\theta+\theta_1) \right) \, R d\theta \right\rbrace \\
& \leq & \frac{(L+L^c_{n})\Hn(\partial\Omega)}{2|\Omega|}.
\end{eqnarray*}
Note that the right-hand side in the last inequality is reached by the function $a^*$ equal to 1 on $\Sigma_1\cup\Sigma_2$ and constant on $\Sigma_3$, where the constant is chosen so that $\int_{\partial\Omega}a^*\, d\Hn=L\Hn(\partial\Omega)$. We have then proved the expected result.
\end{proof}

It remains now to prove the no-gap result between the optimal values for Problem \eqref{mainPbOpt} and its convexified version \eqref{defJrelax} in the case where $L\geq L^c_{n}$.
Since every solution for the convexified problem \eqref{defJrelax} is equal to $1$ on $\Sigma_1\cup\Sigma_2$, it is enough to exhibit a sequence $(a_m)_{m\in \N}$ of elements in $\U$ such that $a_m =1$ on $\Sigma_1\cup\Sigma_2$ and
$$ \lim_{m\rightarrow +\infty} \inf_{(n,k)\in\N^{*2}} \int_{-\theta_1}^{\theta_1} a_m(R,\theta) \sin \left(\frac{n\pi(\theta+\theta_1)}{2\theta_1} \right)^2Rd\theta = L\Hn(\partial\Omega)-2R.
$$
One refers to the proof of \cite[Theorem 1]{PTZObs1} where the construction of such a sequence is explained.



\paragraph{Acknowledgment.}
Y. Privat was partially supported by the Project ''Analysis and simulation of optimal shapes - application to lifesciences'' of the Paris City Hall.


\bibliographystyle{abbrv}
\bibliography{biblio}

\end{document}